\documentclass[3p]{elsarticle}

\usepackage{graphics}
\usepackage{graphicx}
\usepackage{epsfig}
\usepackage{amssymb}
\usepackage{amsthm}
\usepackage{amsmath}
\usepackage{float}
\usepackage{lineno}
\usepackage{textcomp}

\usepackage{epstopdf}
\usepackage{multirow}
\usepackage{color}
\usepackage{hhline}
\usepackage{booktabs}

\newtheorem{theorem}{Theorem}[section]

\definecolor{nverde}{RGB}{0,61,0} 
\definecolor{cr1}{RGB}{200,0,0}
\definecolor{cr2}{RGB}{0,0,200}
\definecolor{cr12}{RGB}{100,0,100}

\DeclareMathOperator{\dive}{\nabla \cdot}
\DeclareMathOperator{\gra}{\nabla}
\DeclareMathOperator{\grae}{\nabla}

\DeclareMathOperator{\dS}{\mathrm{dS}}
\DeclareMathOperator{\dV}{\mathrm{dV}}
\DeclareMathOperator{\dt}{\mathrm{dt}}
\DeclareMathOperator{\dbx}{\mathrm{d}\mathbf{x}}
\newcommand{\g}{\mathbf{g}}
\newcommand{\x}{\mathbf{x}}
\newcommand{\XX}{\mathbf{X}}
\newcommand{\V}{\mathbf{v}}
\newcommand{\T}{T}

\newcommand{\btau}{\boldsymbol{\tau}}
\newcommand{\press}{p}
\newcommand{\Press}{p}

\newcommand{\halb}{\frac{1}{2}}
\newcommand{\vel}{\mathbf{u}}
\newcommand{\bw}{\mathbf{w}}

\newcommand{\Vel}{\mathbf{u}}
\newcommand{\W}{\mathbf{w}}

\newcommand{\WW}[1]{\W^{#1}}

\newcommand{\Cell}{C}
\newcommand{\Dt}{\Delta t\,}

\newcommand{\St}{\mathrm{St}}
\newcommand{\Cellst}{\widetilde{\Cell}}

\newcommand{\Nst}{\widetilde{N}}
\newcommand{\Zst}{\mathcal{Z}} 
\newcommand{\Fst}{\widetilde{\mathbf{F}}}
\newcommand{\Bst}{\widetilde{\mathbf{B}}}
\newcommand{\Dst}{\widetilde{\mathbf{D}}}
\newcommand{\Xst}{\widetilde{\mathbf{X}}}
\newcommand{\xst}{\widetilde{\mathbf{x}}}
\newcommand{\Phist}{\widetilde{\Phi}}
\newcommand{\Gammast}{\widetilde{\Gamma}}
\newcommand{\nablast}{\widetilde{\nabla}}
\newcommand{\nst}{\widetilde{\boldsymbol{n}}}
\newcommand{\etast}{\widetilde{\boldsymbol{\eta}}}
\newcommand{\nbetast}{\widetilde{\eta}}

\newcommand{\mV}{\mathbf{v}}
\newcommand{\mVn}{V}
\newcommand{\Un}{u}

\allowdisplaybreaks

\begin{document}

\begin{frontmatter}

\title{An Arbitrary-Lagrangian-Eulerian hybrid finite volume/finite element method on moving unstructured meshes for the Navier-Stokes equations} 

\author[MatVig]{S. Busto}
\ead{saray.busto@unitn.it}

\author[LAM]{M. Dumbser}
\ead{michael.dumbser@unitn.it}

\author[LAM,MatCor]{L. R\'io-Mart\'in\corref{cor1}}
\ead{laura.delrio@unitn.it}

\cortext[cor1]{Corresponding author}

\address[LAM]{Laboratory of Applied Mathematics, DICAM, University of Trento, Via Mesiano 77, 38123 Trento, Italy}

\address[MatVig]{Departamento de Matem\'atica Aplicada I, Universidade de Vigo, Campus As Lagoas Marcosende s/n, 36310 Vigo, Spain}

\address[MatCor]{Departamento de Matem\'aticas, Universidade da Coru\~na, Campus Elvi\~na s/n, 15071 A Coru\~na, Spain}

\begin{abstract}
This paper presents a novel semi-implicit hybrid finite volume / finite element {(FV/FE)} scheme for the numerical solution of the incompressible and weakly compressible Navier-Stokes equations on moving unstructured meshes using a direct Arbitrary-Lagrangian-Eulerian (ALE) formulation. The scheme is based on a suitable splitting of the governing partial differential equations into subsystems and employs a staggered grid arrangement, where the pressure is defined on the primal simplex mesh, while the velocity and the remaining flow quantities are defined on an edge-based staggered dual mesh. The key idea of the scheme presented in this paper is to discretize the nonlinear convective and viscous terms at the aid of an explicit finite volume scheme that employs the space-time divergence form of the governing equations on moving space-time control volumes. For the convective terms, an ALE extension of the Ducros flux on moving meshes is introduced, which can be proven to be kinetic energy preserving and stable in the energy norm when adding suitable numerical dissipation terms. The use of closed space-time control volumes inside the finite volume scheme guarantees that the important geometric conservation law (GCL) of Lagrangian schemes is verified \textit{by construction}. 
Finally, the pressure equation of the Navier-Stokes system is solved on the new mesh configuration at the aid of a classical continuous finite element method, using traditional $\mathbb{P}_1$ Lagrange elements. 

A numerical convergence study confirms that the scheme is second order accurate in space. Subsequently, the ALE hybrid FV/FE method is applied to several incompressible test problems ranging from non-hydrostatic free surface flows over a rising bubble to flows over an oscillating cylinder and an oscillating ellipse. Via the simulation of a circular explosion problem on a moving mesh, we show that the scheme applied to the weakly compressible Navier-Stokes equations is able to capture also weak shock waves, rarefactions and moving contact discontinuities. 
{We also provide numerical evidence which shows that compared to a fully explicit ALE scheme, the semi-implicit ALE method proposed in this paper is particularly efficient for the simulation of weakly compressible flows in the low Mach number limit.} 

\end{abstract}


\end{frontmatter}


\section{Introduction}\label{sec:intro}

Lagrangian numerical methods that are able to solve time-dependent partial differential equations (PDE) on moving meshes are of fundamental importance for practical applications in science and engineering, in particular for fluid-structure interaction (FSI) problems, where a moving fluid interacts with a moving rigid or elastic solid body, see, e.g. \cite{NobileVergara,LeTallecMouro,Feistauer1,Feistauer2,Feistauer3,Feistauer4}. A nice and extensive literature review on Arbitrary-Lagrangian-Eulerian (ALE) finite element methods for incompressible flow problems and the importance of the geometric conservation law (GCL) can be found in the paper \cite{Etienne2009} and references therein. Pure Lagrangian and Arbitrary-Lagrangian-Eulerian (ALE) schemes are also important in high energy density physics, such as for the numerical simulation of inertial confinement fusion (ICF) processes and the simulation of compressible multi-material flows.
Lagrangian schemes are standard in the context of computational solid mechanics and in the field of computational fluid dynamics they allow to capture moving material interfaces and contact discontinuities without or only very little numerical dissipation. A non-exhaustive review of the most relevant past and recent developments in the field can be found, for example in~\cite{Neumann1950,Caramana1998,Despres2005,Despres2009,Maire2007,Maire2009,LoubereSedov3D,ShashkovCellCentered,ShashkovMultiMat3,chengshu1,Lagrange3D,LagrangeNC,LagrangeDG,KSX17,GBCKSD19}. A fundamental aspect of all aforementioned Lagrangian schemes on moving meshes is the exact verification of the so-called geometric conservation law (GCL) on the discrete level, which states that the rate of change of the volume of a moving control volume $\Cell_i(t)$ must be related to the divergence of the local mesh velocity $\mathbf{v}$, 
\begin{equation}
	\frac{\partial }{\partial t} \int \limits_{\Cell_i(t)} d\mathbf{x} - \int \limits_{\partial \Cell_i(t)} \mathbf{v} \cdot \mathbf{n} \, dS = 0, 
	\label{eqn.GCL}
\end{equation} 
where $\mathbf{n}$ is the outward pointing unit normal vector to the surface of $\Cell_i(t)$. 
The geometric conservation law \eqref{eqn.GCL} is also essential for Lagrangian schemes for incompressible flows, see~\cite{Farhat2001}. In \cite{Lagrange3D}, it was proven that a finite volume scheme based on conservation laws written in space-time divergence form and the integration over closed space-time control volumes automatically satisfies the integral form of the GCL \eqref{eqn.GCL} by construction. In this paper, we will therefore make use of the framework introduced and applied in \cite{Lagrange2D,Lagrange3D,LagrangeDG,LagrangeNC,GDC17,GBCKSD19,Gaburro2021}.    

In order to solve incompressible or weakly compressible flow problems, as well as for the design of all Mach number flow solvers, the use of semi-implicit schemes on staggered meshes is very convenient from a computational point of view. This family of schemes goes back to the pioneering work of Harlow and Welch \cite{HW65} and Casulli et al. \cite{CG84,CC92,CW00}. For the history and an overview of semi-implicit all Mach number flow solvers, see \cite{MRKG03,PM05,DeT11,CordierDegond,DC16,RussoAllMach,AbateIolloPuppo,Thomann2020,BDT2021}. Recent work also concerns the extension of staggered semi-implicit schemes to the discontinuous Galerkin finite element framework, see \cite{DC13,TD16,TD17,FambriDumbser,BTBD20}. 

A family of semi-implicit hybrid finite volume / finite element schemes for solving the incompressible and compressible Navier-Stokes equations and the shallow water equations on fixed staggered unstructured meshes in two and three space dimensions was introduced in a series of papers in  \cite{BFSV14,BFTVC17,BBDFSVC20,BRMVCD21,HybridMPI,HybridNNT,BD22}. These schemes employ an explicit second order accurate finite volume discretization of the nonlinear convective terms and a classical second order continuous Lagrange finite element scheme for the solution of the discrete pressure equation. The staggered mesh arrangement in this class of methods is the same as the one used in \cite{TD16,TD17,BTBD20}, and for the compressible Navier-Stokes equations, the splitting of the governing equations into subsystems is based on the Toro-V\'azquez flux vector splitting, see \cite{TV12}. Here, we extend this approach to the more general case of moving unstructured staggered meshes and combine it for the first time with a kinetic energy preserving Ducros flux \cite{Ducros1999,Ducros2000,Moin2009} instead of the classical Rusanov flux \cite{Rus62} employed in all the previous references. {Another important objective of the proposed semi-implicit ALE scheme is to achieve better computational efficiency for weakly compressible flows in the low Mach number limit compared to fully explicit ALE methods. The time step of explicit ALE schemes is limited by a CFL condition based on the sound speed, while the CFL condition of the new semi-implicit ALE scheme is only based on the fluid velocity.}

The structure of this paper is as follows. Section \ref{sec:goveq} recalls the governing partial differential equations under consideration in this article, namely the incompressible Navier-Stokes equations, and their kinetic energy preserving property in the case of vanishing viscosity, as well as the weakly compressible Navier-Stokes equations in pressure formulation. We next introduce a suitable splitting of the governing PDE system and perform a semi-discretization in time, keeping space still continuous. 
In Section \ref{sec:numdisc}, the new hybrid FV/FE scheme on moving meshes is presented. The nonlinear convective and viscous terms are discretized via an ALE finite volume method based on an integral formulation of a space-time divergence formulation of the convective and viscous subsystem. We furthermore prove the kinetic energy preserving property of the central part of the employed numerical flux on moving meshes. Subsequently, the discretization of the pressure subsystem via continuous finite elements is discussed and the modifications needed to extend the algorithm to weakly compressible flows are detailed. Finally, some key points for the treatment of boundary conditions are provided in Section~\ref{sec:bc}. In Section \ref{sec:numericalresults}, we present a set of numerical results obtained with the new ALE hybrid FV/FE scheme for different benchmark problems. The paper closes with some concluding remarks and an outlook towards future work in Section \ref{sec:conclusions}.  

\section{Governing equations} \label{sec:goveq}
We will introduce the new family of ALE hybrid FV/FE methods for two different mathematical models: the incompressible Navier-Stokes equations and the compressible Navier-Stokes equations written in terms of a non-conservative pressure equation, which is therefore valid in the context of weakly compressible flows. To this end, in this section we first recall both systems of equations and we then introduce a general space-time divergence form of the convective and viscous subsystem associated to each model, which is later needed for the description of the numerical method.

\subsection{Incompressible Navier-Stokes equations} 

We consider the incompressible Navier-Stokes equations, coupled with the temperature transport equation through the Boussinesq assumption. The governing system of PDE reads 
\begin{subequations}\label{eqn.incns}
\begin{gather}
	\dive \vel = 0,\label{eqn.mass}\\
	\frac{\partial \left(\rho  \vel\right)}{\partial t} + \dive \left(\rho  \vel \otimes \vel\right) 	+ \gra \press 	- \dive \btau   = \rho\g-\rho\beta \left( \T-\T_{\mathrm{ref}}\right) \g,\label{eqn.momentum} \\
	\frac{\partial \T}{\partial t} + \dive \left(\vel \T\right)  - \dive \left( \alpha \gra \T \right)   = 0\label{eqn.temperature}
\end{gather}
\end{subequations}
with $\rho\vel$, $\rho$ and $\vel=(u_1,u_2,u_3)$ the linear momentum, the density and the velocity vector, respectively, $\press$ the pressure, $\btau = \mu\left( \nabla \vel + \nabla \vel^T \right)$ the viscous stress tensor, $\mu$ the laminar viscosity, $\g$ the gravity vector, 
$\T$ the temperature, $T_{\mathrm{ref}}$ the reference temperature, $\alpha=\kappa/\left( \rho c_{p}\right) $ the thermal diffusivity, $\kappa$ the thermal conductivity, $c_{p}$ the heat capacity at constant pressure and  $\beta$ the thermal expansion coefficient.
Note that, in the inviscid case, $\mu=0$, and in the absence of gravity and buoyancy, $\mathbf{g}=\mathbf{0}$, $\beta=0$, the incompressible Navier-Stokes equations reduce to the incompressible Euler equations:
\begin{gather}
	\dive \vel = 0,\label{eqn.mass.euler}\\
	\frac{\partial \left(\rho  \vel\right)}{\partial t} + \dive \left(\rho  \vel \otimes \vel\right) 	+ \gra \press 	= 0.
	\label{eqn.momentum.euler} 
\end{gather}
Via multiplication of the momentum equation \eqref{eqn.momentum.euler} with the velocity vector, $\vel$, and making use of the incompressibility constraint \eqref{eqn.mass.euler} it is easy to prove that the system \eqref{eqn.mass.euler}-\eqref{eqn.momentum.euler} satisfies the following extra conservation law for the kinetic energy density 
\begin{equation*}
	\frac{\partial}{\partial t} \left(\halb \rho  \vel^2 \right) + \dive \left(\vel  \, \halb \rho  \vel^2 \right) 	+ \dive \left( \vel \press \right) = 0.
\end{equation*}

To discretize the system \eqref{eqn.mass}-\eqref{eqn.temperature}, we will make use of a hybrid FV/FE strategy, similar to the one proposed in \cite{BFTVC17,HybridMPI} within the Eulerian framework. Accordingly, the starting point of the semi-implicit ALE hybrid FV/FE method proposed in this paper is the definition of a suitable split of the governing PDE system~\eqref{eqn.incns} into two subsystems. The first one is hyperbolic-parabolic and will later be discretized at the aid of explicit finite volume schemes on moving control volumes. It contains the nonlinear convective terms, the viscous terms and eventual algebraic source terms. The second subsystem is elliptic and is therefore discretized via classical continuous finite elements. It contains the divergence free condition of the velocity field and the pressure gradient in the momentum equation. These subsystems read as follows: 

\paragraph{Convective and viscous subsystem}
\begin{gather*}
	\frac{\partial \left(\rho  \vel\right)}{\partial t} + \dive \left(\rho  \vel \otimes \vel\right) 	 	- \dive \btau   = \rho\g-\rho\beta \left( \T-\T_{\mathrm{ref}}\right) \g,    \\
	\frac{\partial \T}{\partial t} +  \dive \left( \vel T\right)   - \dive \left( \alpha \gra \T \right)   = 0.  
\end{gather*}
\paragraph{Pressure subsystem}
\begin{subequations}\label{eqn.inc.velo.sub}
\begin{gather}
	\dive \vel = 0, \\
	\frac{\partial \left(\rho  \vel\right)}{\partial t} 	+ \gra \press 	= 0. 
\end{gather}
\end{subequations}

\subsection{Weakly compressible Navier-Stokes equations}
To extend the methodology to weakly compressible flows, we will consider the compressible Navier-Stokes equations. Since our aim is not to solve high Mach number flows, but weakly compressible flows with relatively smooth solutions, the compressible Navier-Stokes equations can be reformulated by replacing the energy conservation law by a non-conservative pressure evolution equation,  thus obtaining the following system:
\begin{subequations}\label{eqn.cns}
\begin{gather}
	\frac{\partial \rho}{\partial t} + \dive  \rho\vel =0, \label{eqn.weak.mass}\\
	\frac{\partial \rho\vel}{\partial t} + \dive \left(\rho  \vel \otimes \vel\right)  + \grae \press - \dive \btau  = \rho \g,\label{eqn.weak.momentum}\\
	\frac{\partial \press}{\partial t}  + \vel \cdot \grae \press - c^2 \vel\cdot\grae  \rho + c^2 \dive \left(\rho\vel\right) + \left(\gamma-1\right)\left(\dive \mathbf{q} -  \btau  : \gra \vel \right) =0, \label{eqn.weak.pressure}
\end{gather} 
\end{subequations}
with $c=\sqrt{\frac{\gamma \press}{\rho}}$ the isentropic sound speed corresponding to the ideal gas equation of state (EOS), $\gamma$ the heat capacity ratio, $\mathbf{q} = -\kappa \grae T$ the heat flux and $\btau$ the viscous stress tensor that takes the form \mbox{$\btau =\mu \left(\gra \vel + \gra \vel^{T} \right) -\frac{2}{3} \mu \left( \dive \vel  \right) \mathbf{I}$,} with $\mathbf{I}$ being the identity matrix. 
Like for the incompressible case, we now perform a splitting of the above system into a subsystem containing the nonlinear convective and viscous terms together with the non-conservative product, arising in the pressure equation, and a pressure subsystem leading to: 

\paragraph{Convective and viscous subsystem}
\begin{subequations}	\label{eqn.cns.conv.sub}
\begin{gather}
	\frac{\partial \rho}{\partial t} + \dive  \rho\vel =0,\label{eqn.weak.mass.cv}\\
	\frac{\partial \rho\vel}{\partial t} + \dive \left(\rho  \vel \otimes \vel\right)   - \dive \btau  = \rho \g,\label{eqn.weak.momentum.cv}\\
	\frac{\partial \press}{\partial t}  + \vel \cdot \grae \press  + \left(\gamma-1\right)\dive \mathbf{q}  = 0. \label{eqn.weak.pressure.cv}
\end{gather} 
\end{subequations}

\paragraph{Pressure subsystem}
\begin{subequations}\label{eqn.cns.press.sub}
\begin{gather}
	\frac{\partial \rho\vel}{\partial t}  + \grae \press  = 0, \label{eqn.weak.momentum.p}\\
	\frac{\partial \press}{\partial t}   - c^2 \vel\cdot\grae  \rho + c^2 \dive \left(\rho\vel\right) - \left(\gamma-1\right)  \btau  : \gra \vel  = 0.\label{eqn.weak.pressure.p}
\end{gather} 
\end{subequations}

Let us note that the former non-conservative reformulation of the equations is equivalent to the compressible Navier-Stokes equations only for smooth flows, i.e. for small to medium Mach numbers.
Nevertheless, as it has been shown, e.g. in \cite{BBDFSVC20}, this non conservative formulation may result very useful with respect to a fully conservative system since it is still able to properly capture small shocks using a simpler pressure subsystem and with a lower computational cost. 
For all Mach number flow solvers, we refer the reader to \cite{Hybrid2,PM05,AbateIolloPuppo,RussoAllMach,TD17,Thomann2020} in the context of hybrid FV/FE, FV and DG schemes on fixed Eulerian meshes, while for high order ALE-FV and ALE-DG methodologies for compressible flows we refer to \cite{LagrangeNC,LagrangeDG,GBCKSD19}.

\subsection{Space-time divergence form of the convective and viscous subsystems}

With the state vectors $\W=(\rho \vel, T)$ for the incompressible case and $\W=(\rho, \rho \vel, \press)$ for the weakly compressible Navier-Stokes equations, the former convective and viscous subsystems can be written in a compact space-time divergence formulation as follows, see e.g. \cite{Lagrange3D,LagrangeDG,GBCKSD19}: 
\begin{equation}
	\nablast \cdot \Fst(\W,\nabla \W) + \Bst(\W) \cdot \nablast \W = \mathbf{S}\left( \W,\nabla \W \right), 	\label{eqn.stdiv}  
\end{equation} 
where the space-time nabla operator $\nablast$,  the space-time flux $\Fst(\W,\nabla \W)$ and the matrix $\Bst(\W)$ appearing in the non-conservative product  are defined as 
\begin{equation*}
	\nablast = \left( \frac{\partial}{\partial t}, \nabla \right), \qquad 
	\Fst = \left( \W, \mathbf{F}(\W,\nabla \W) \right), \qquad 
	\Bst(\W) = \left( 0, 
	\mathbf{B}_1(\W), 
	\mathbf{B}_2(\W), 
	\mathbf{B}_3(\W)  
	\right),
\end{equation*}
and the flux $\mathbf{F}$ can be decomposed into $\mathbf{F}= \mathbf{F}_c + \mathbf{F}_v$, being $\mathbf{F}_c$ and $\mathbf{F}_v$ the fluxes for the convective and the viscous terms, respectively.

\section{Numerical discretization} 
\label{sec:numdisc}

In this section, we present the ALE hybrid FV/FE numerical method, which can be divided into four  stages:
\begin{enumerate}
	\item Mesh motion. First, the local mesh velocity is chosen according to one of the following two options: 
	\begin{enumerate}
		\item one option is to impose a given mesh velocity on the boundaries of the computational domain and then to solve a Laplace equation in the interior in order to obtain a smooth field for the mesh velocity everywhere; 
		\item another option is to move the mesh with a smoothed version of the local fluid velocity.
	\end{enumerate}              
	\item Transport-diffusion stage. Second, the hyperbolic-parabolic subsystem containing the convective and viscous terms is solved and an intermediate solution of the conserved variables is obtained by means of an explicit finite volume method on a moving unstructured mesh. In this stage we make use of the space-time divergence form \eqref{eqn.stdiv} of the governing PDE system.  
	\item Projection stage. Next, the equation for the pressure is solved by using continuous implicit finite elements on the new mesh.
	\item Post-projection stage. Finally, the momentum is updated, with the contribution of the pressure field computed in the projection stage.
\end{enumerate}
In what follows, we first introduce the employed staggered unstructured grids on moving meshes, describing the primal and the dual meshes and the notation used on both grids.  
Then, we explain the finite volume discretization of the convective and viscous subsystems on moving unstructured meshes, describe the mesh motion, including how the mesh velocity is chosen, and provide some details about the finite element discretization of the pressure subsystem. Finally, we prove that the kinetic energy is preserved for the proposed scheme when using the Ducros flux function, and we complete the description of the weakly compressible flow solver.

\subsection{Staggered unstructured mesh} 

For the sake of simplicity, we start focusing on describing the two-dimensional meshes at time $t^n$ while its movement will be shown in Section \ref{sec:mesh_motion}. We employ so-called edge-type or face-type unstructured staggered grids, which have already been widely used in the framework of FV, hybrid FV/FE, and DG methods, see e.g. \cite{BDDV98,BFSV14,TD14,TD16,BTBD20,BBDFSVC20,HybridMPI}. From now on, the notation with tilde, $\widetilde{\cdot}$, is used for space-time control volumes, variables and mesh notation, while the corresponding space structures at a fixed time are written without it. 
Since we use an ALE approach, the computational grid may change at each time step, so, when necessary, we will indicate the time step with a superindex. 

We discretize the computational domain at time $t^{n}$, $\Omega^{n}=\Omega(t^n)\subset\mathbb{R}^{2}$, by considering a set of disjoint triangles $T_{\Omega}^{n} =\left\lbrace \T^{n}_k, \, k=1,\ldots, n_{el}\right\rbrace$, with vertices $V_{j}$, $j=\left\lbrace1,\dots, n_{ver}\right\rbrace$, such that $\Omega^{n} = \displaystyle{\bigcup^{n_{el}}_k} \,\T_k^{n}$. 
This set of triangles forms the primal mesh (see Figure~\ref{fig:2Dmesh}). 
To design the dual mesh, we consider the barycenters, $B$ and ${B}^{\prime}$, of two triangles which share an edge. 
The interior dual elements are constructed merging the two subtriangles formed by joining these two barycenters with the vertices that belong to the shared edge. 
On the other hand, the boundary dual elements are formed by joining the barycenter of a boundary primal element with the boundary edge (see right sketch in Figure~\ref{fig:2Dmesh}). 
On the dual mesh, for all $i=1,\dots, n_{nod}$, $\Cell_i$ is a cell, $N_i$ is the node of the dual element, $\Gamma_{i}=\partial \Cell_i$ is the boundary of the cell $\Cell_{i}$ with $\mathbf{n}_{i}=(n_x,n_y)$ its outward unit normal, and $\Gamma_{ij}$ is the common edge between $\Cell_{i}$ and $\Cell_{j}$. 
Since we are working in a Lagrangian framework, the processes to generate both meshes is repeated at each time step. 
It is necessary to highlight that each triangle in the primal mesh is formed by the same vertices at any time, although the coordinates of the vertices and also the area of the element may be different. In other words, we do not allow the topology to change. 
Similarly, each element on the dual mesh is formed by the same nodes, but the coordinates and the area potentially change at each time step. 
For an alternative ALE approach allowing mesh topology changes and thus large deformations highly limiting mesh degradation, we refer to \cite{GBCKSD19}. 
Once the grids at two subsequent time steps are designed, the space-time control volumes, $\Cellst_{i}$, can be built by joining the vertex of each dual element at times $t^{n}$ and $t^{n+1}$, so $\Cellst_{i\mid_{t^{n}}}=C_{i}^{n}$, $\Cellst_{i\mid_{t^{n+1}}}=C_{i}^{n+1}$, see Figure \ref{fig.spacetime}. 
We denote $\Nst_i$  the node of $\Cellst_{i}$, ${\Cellst}_{i}^\circ = \Cellst_i \backslash \Gammast_i$ the interior of cell $\Cellst_{i}$, $\Gammast_{i}=\partial \Cellst_i$ the boundary, $\nst_{i}=(\tilde{n}_t,\tilde{n}_x,\tilde{n}_y)$  its corresponding space-time outward unit normal and $\Gammast_{ij}$ the face shared by $\Cellst_{i}$ and $\Cellst_{j}$.
\begin{figure}[h]
	\begin{center}
	\includegraphics[width=0.35\linewidth]{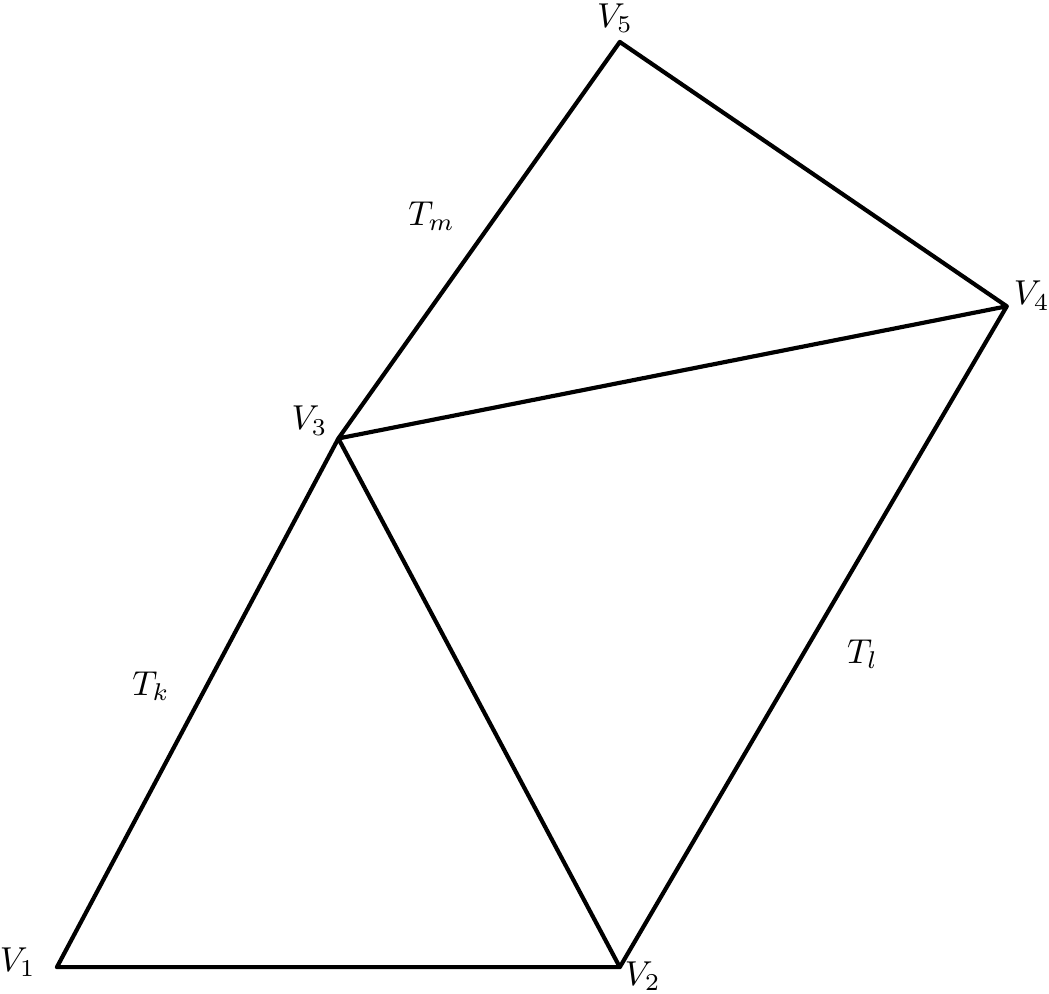}
	\includegraphics[width=0.35\linewidth]{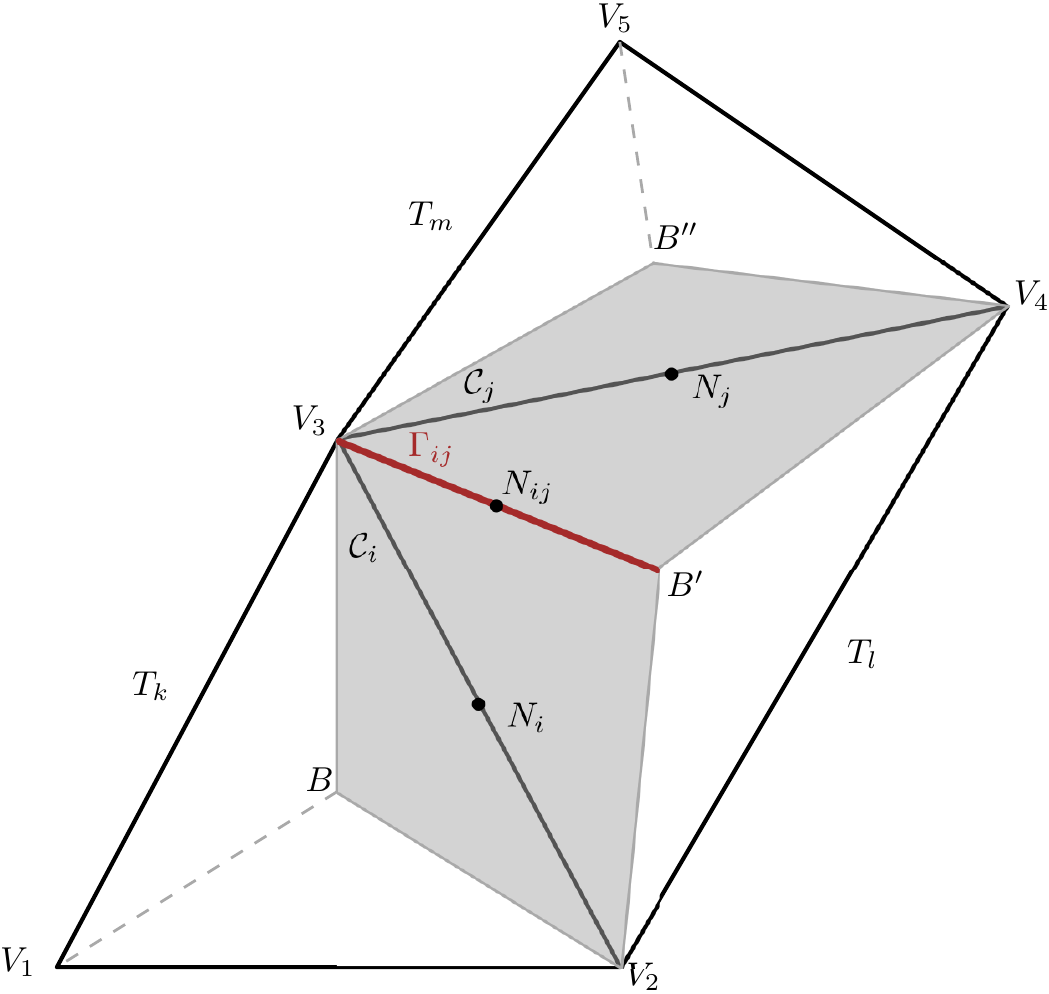}
	\end{center}
	\caption{Sketch of the edge-based unstructured grids in 2D at a fixed time $t^n$. The triangles of the primal mesh are $T_{k}$, $T_{l}$, $T_{m}$ and $V_{j},\, j=1,\dots, 5$ are their vertices (left). The interior elements of the dual mesh, $\Cell_{i}$, are built by joining the barycenters of two adjacent primal elements with the shared edge (grey elements in the right figure) and the boundary elements are built joining the barycenter of one boundary primal element, with the boundary edge (white elements in the right figure).}
	\label{fig:2Dmesh}
\end{figure}
\begin{figure}[h]
	\begin{center}
		\includegraphics[width=0.4\linewidth]{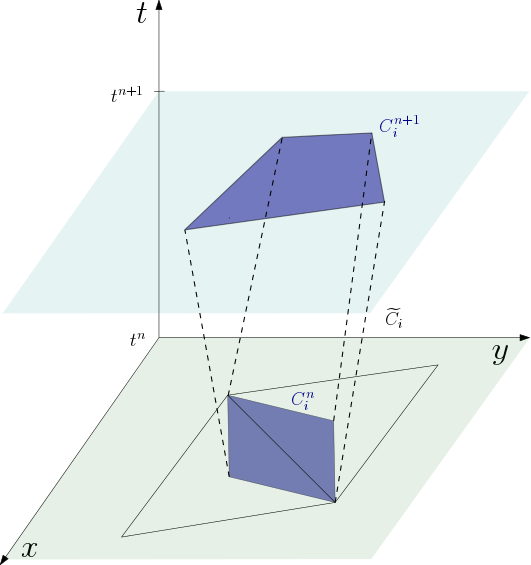}
	\end{center}
	\caption{Space-time dual element, where the dual cells at time $t^n$ and $t^{n+1}$, $\Cell_{i}^{n}$, $\Cell_{i}^{n+1}$, are highlighted in purple. {The space-time control volume, $\Cellst_{i}$, is determined by these two cells and the space-time faces joining each edge at time $t^n$ with itself at time $t^{n+1}$(faces delimited by the dashed lines).} The triangles depicted at time $t^n$ correspond to the two primal elements used for the generation of the dual element $\Cell_{i}^{n}$.}
	\label{fig.spacetime}
\end{figure}

\subsection{ALE finite volume discretization based on the space-time divergence form of the convective and viscous subsystem}\label{sec:numdisc.ALEFV}

In order to discretize the convective and viscous subsystems, we denote $\WW{n}$ an approximation of the solution $\mathbf{w}(\x,t^n)$, that is, $\WW{n}=(\rho \Vel^n , T^n)$ for incompressible flows and $\WW{n}=(\rho^n,\rho \vel^n , \press^n)$ for the weakly compressible case, and consider a space-time control volume $\Cellst_i=[t^{n},t^{n+1}]\times\Cell_i$. Integrating over $\Cellst_i$, equation~\eqref{eqn.stdiv} yields
\begin{equation}
\int_{\Cellst_i}\nablast \cdot \Fst(\WW{n},\nabla \WW{n}) \,\dt\dbx
+ \int_{\Cellst_i}	\Bst(\WW{n}) \cdot \nablast \WW{n}\,\dt\dbx 
= \int_{\Cellst_i}	\mathbf{S}\left( \WW{n},\nabla \WW{n} \right)\,\dt\dbx. 
	\label{eqn.int_stdiv}  
\end{equation}

We now apply Gauss theorem to write the integral of the space-time flux, $\Fst$, as an integral over the space-time surface $\Gammast_{i}$ . Moreover, the space-time surface $\Gammast_i$, is split into six sub-surfaces (see Figure~\ref{fig.spacetime}) as 

\begin{equation*} 
	\Gammast_i = \left(\bigcup \limits_{\Cell_j \in {\cal K}_{i}} \partial \Cellst_{ij}\right) \cup \Cell_i^n \cup \Cell_i^{n+1},
\end{equation*}
where ${\cal K}_{i}$ denotes the adjacent cells to $\Cellst_{i}$, i.e. the Neumann neighbourhood.
Note that, due to the orthogonality of faces $\Cell_{i}^{n}$ and $\Cell_{i}^{n+1}$ with the time axis, their exterior unit normal vectors are $\nst=(-1,0,0)$ and $\nst=(1,0,0)$, respectively, while for the remaining boundaries we have general space-time normals of the form $\nst=(\tilde{n}_t,\tilde{n}_x,\tilde{n}_y)$.
To integrate the non-conservative product, a path-conservative approach is used following~\cite{GCD18,Par06,DCPT09,CFFP09,DHCPT10}. Consequently, the related contributions can be split in two addends: the first one, that involves the potential jumps of $\WW{n}$ between cells and the second one which corresponds to the smooth contribution inside each cell. Thus, \eqref{eqn.int_stdiv} can be rewritten as 
\begin{align}
\int_{\Gammast_i}\Fst(\WW{n},\nabla \WW{n})\cdot \nst \,\dt\dS 
+  \int_{\Gammast_i}\Dst(\WW{n},\nabla \WW{n})\cdot \nst  \,\dt\dS
+\int_{{\Cellst}_{i}^\circ}	\Bst(\WW{n}) \cdot \nablast \WW{n}\,\dt\dbx = \int_{\Cellst_i}	\mathbf{S}\left( \WW{n},\nabla \WW{n} \right)\,\dt\dbx 
	\label{eqn.intsurf_stdiv}  
\end{align}
with $\Fst(\WW{n},\nabla \WW{n})\cdot \nst$ the numerical flux in normal direction and $\Dst(\WW{n},\nabla \WW{n}) = \left( 0, \mathbf{D} (\WW{n},\nabla \WW{n})	\right)$, the path-conservative jump term that will be defined later. { Note that in the above space-time conservation formulation \eqref{eqn.intsurf_stdiv} the local mesh velocity of the ALE scheme is hidden in the time component of the space-time normal vector $\nst$, see also \cite{Lagrange2D,Lagrange3D} for details}. 

As illustrated in Figure~\ref{fig.map}, the space-time sub-faces, $\Gammast_{ij}=\partial\Cellst_{ij}$, can be mapped in a local reference coordinate system $\chi-\tau$ as 

\begin{equation*}
\xst(\chi,\tau) = \Phist_1\Xst_1+\Phist_2\Xst_2+\Phist_3\Xst_3+\Phist_4\Xst_4,
\end{equation*}
where the basis functions $\Phist_i=\Phist_i(\chi,\tau)$ are given by
\begin{equation}
	\Phist_1 = (1-\chi)(1-\tau), \qquad
	\Phist_2 = \chi(1-\tau),     \qquad
	\Phist_3 = (1-\chi)\tau,     \qquad
	\Phist_4 = \chi\tau,        
\end{equation}
with $0\le\chi\le 1$ and $0\le\tau\le 1$, and $\Xst_i$ are the space-time coordinate vectors of the four vertices which form the face, i.e.  
if $\XX_1^n=(x_1,y_1)$ and $\XX_2^n=(x_2,y_2)$ are the coordinates of the two vertices defining one edge at time $t^{n}$, $\XX_1^{n+1}=(x_4,y_4)$ and $\XX_2^{n+1}=(x_3,y_3)$ the coordinates after the mesh movement at time $t^{n+1}$, then the vectors $\Xst_i$ are defined as 
\begin{align*}
\Xst_1 &= (t^{n}, \XX_1^{n})=(t^{n}, x_1, y_1),& \qquad
&\Xst_2 = (t^{n}, \XX_2^{n})=(t^{n}, x_2, y_2),&     \\
\Xst_3 &= (t^{n+1}, \XX_2^{n+1})=(t^{n+1}, x_3, y_3), &    \qquad
&\Xst_4 = (t^{n+1}, \XX_1^{n+1})=(t^{n+1}, x_4, y_4).&        
\end{align*}
Furthermore, the space-time unit normal vector on the space-time face reads $$\nst_{ij}=(\tilde{n}_{ij}^t,\tilde{n}_{ij}^x,\tilde{n}_{ij}^y) = \left(\frac{\partial \xst}{\partial \chi}\times \frac{\partial \xst}{\partial \tau}\right)/\|\frac{\partial \xst}{\partial \chi}\times \frac{\partial \xst}{\partial \tau}\|,$$ 
and its integral is 
\begin{equation}
	\etast_{ij} := \int_{\Gammast_{ij}} \nst_{ij} \, \dt \dS .  
\end{equation}
{Note that the space-time averaged local mesh velocity of the ALE scheme is hidden in the time component of the averaged space-time normal vector $\etast_{ij}$. }  
In particular, in two space-dimensions plus time its components, $\etast_{ij} = (\nbetast_{ij}^t,\nbetast_{ij}^x,\nbetast_{ij}^y)$, read 
\begin{align}
\nbetast_{ij}^t&=\frac{1}{2}\left((y_2 - y_3)(x_1-x_4) + (y_1-y_4)(x_3-x_2)\right),\\
\nbetast_{ij}^x &= \frac{t^n - t^{n+1}}{2}\left(y_1+y_3-y_2-y_4\right),\\
\nbetast_{ij}^y &= \frac{t^n - t^{n+1}}{2}\left(-x_1-x_3+x_2+x_4\right).
\end{align}
\begin{figure}[h]
	\begin{center}
		\includegraphics[width=0.9\linewidth]{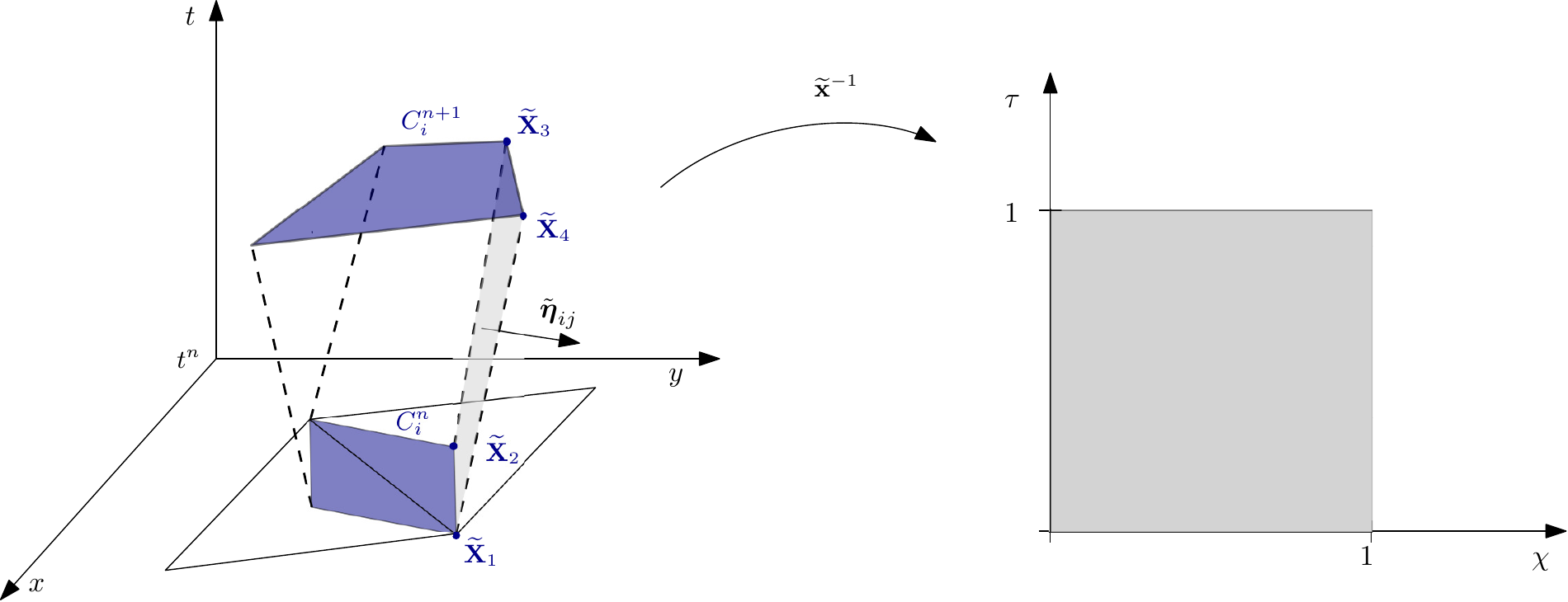}
	\end{center}
	\caption{Parametrization of a dual space-time face. The the cells at time $t^n$ and $t^{n+1}$ are highlighted in purple in the dual element. A face formed by the same edge of a cell at time $t^n$ and $t^{n+1}$ with non-unitary normal vector $\etast_{ij}$ is highlighted in grey (left). The physical space-time face can be mapped to a reference element in a local reference coordinate system $\chi-\tau$ (right).}
	\label{fig.map}
\end{figure}
Then, since the boundary integrals can be split as sum of integrals on the space-time faces, the convective term and the non-conservative jumps in equation~\eqref{eqn.intsurf_stdiv} become

\begin{equation}
\int_{\Gammast_i}\left( \mathbf{F}_c(\WW{n},\nabla \WW{n})\cdot \nst 
+  \mathbf{D}(\WW{n},\nabla \WW{n})\cdot \nst \right) \dt\dS 
=\sum_{\Cell_j \in \mathcal{K}_i} \Zst(\WW{n}_i,\WW{n}_j,\etast_{ij})
\label{eqn.stflux} 
\end{equation}
with 
\begin{equation*}
	\Zst(\WW{n}_i,\WW{n}_j,\etast_{ij}) = \mathcal{F}_c(\WW{n}_i,\WW{n}_j,\etast_{ij}) + \mathcal{D}(\WW{n}_i,\WW{n}_j,\etast_{ij})
\end{equation*}
including both, the numerical flux function for the convective flux, $\mathcal{F}_c$, and the jump term, $\mathcal{D}$. 

The former path-conservative jump term reads 
\begin{equation*}
	\mathcal{D}(\WW{n}_i,\WW{n}_j,\etast_{ij})  =\frac{1}{2}\Bst_{ij} \left(\WW{n}_{j}-\WW{n}_{i}\right), \qquad 
	\Bst_{ij} = \int_{0}^{1} \Bst\left(\psi\left(\WW{n}_{i},\WW{n}_{j},s\right)\right) \cdot\etast_{ij} \mathrm{ds}, 
\end{equation*}
where the chosen straight-line segment path, with $s\in[0,1]$, is defined as 
\begin{equation*}
	\psi = \psi\left(\WW{n}_{i},\WW{n}_{j},s\right) = \WW{n}_{i} + s \left( \WW{n}_{j} - \WW{n}_{i} \right).
\end{equation*}

In order to approximate the numerical flux for the convective part of the incompressible Navier-Stokes equations, 
we consider the Ducros numerical flux function, \cite{Ducros1999,Ducros2000},
\begin{equation}
	\mathcal{F}_c(\WW{n}_i,\WW{n}_j,\etast_{ij})= \displaystyle\frac{1}{2} \left( \WW{n}_i + \WW{n}_j \right) \frac{1}{2} \left(\widetilde{\Vel}^{n}_i+\widetilde{\Vel}^{n}_j\right) \cdot \etast_{ij}
	-\frac{1}{2}	\widetilde{\alpha}^{n}_{ij} \left( \WW{n}_{j}-\WW{n}_{i}\right), \label{eqn.Ducros_flux}
\end{equation}
where $\widetilde{\Vel} = (1, \Vel)$ is the four velocity that allows to write the convective space-time flux tensor of the incompressible Navier-Stokes equations very compactly as \mbox{$\Fst_c(\WW{}) = \WW{} \otimes \widetilde{\Vel} $} and $\widetilde{\alpha}^{n}_{ij}$ is the maximum absolute eigenvalue of the transport-diffusion subsystem, plus some artificial viscosity, $c_{\alpha}$. 
{Although in most test cases $c_{\alpha}=0$, this coefficient can be set to a positive value to increase the robustness of the final scheme. More specifically, in compressible flows large gradients in the pressure and density fields can arise in presence of small velocities. But since in our semi-implicit algorithm the numerical viscosity of the Rusanov and Ducros fluxes scales only with the bulk flow velocity, it may be very small. This causes instabilities that can be avoided by setting an artificial viscosity coefficient $c_{\alpha}>0$, see the numerical results Section. } 
As an alternative numerical flux, which will be always employed for weakly compressible flows, we consider the Rusanov~\cite{Rus62,BBDFSVC20} flux function,
\begin{equation}
	\mathcal{F}_c(\WW{n}_i,\WW{n}_j,\etast_{ij})= \displaystyle\frac{1}{2} \left(\mathbf{F}_c(\WW{n}_i)+\mathbf{F}_c(\WW{n}_j)\right) \cdot \etast_{ij} 
	- \frac{1}{2}	\widetilde{\alpha}^{n}_{ij} \left( \WW{n}_{j}-\WW{n}_{i}\right).
	\label{eqn.Rusanov_flux}
\end{equation}
{Note that the instantaneous local mesh velocity normal to the element boundary {$V_{ij}$, which will be needed later,} is related to the time component of the space-time normal vector {appearing in \eqref{eqn.intsurf_stdiv}, \eqref{eqn.Ducros_flux} and \eqref{eqn.Rusanov_flux}} by the relation 
\begin{equation}
	V_{ij} = \mathbf{V} \cdot \mathbf{n}_{ij} = - \frac{\tilde{n}^t_{ij}}{\sqrt{ \left(\tilde{n}^x_{ij} \right)^2 + \left(\tilde{n}^y_{ij} \right)^2}}, 
	\label{eqn.normal.meshvelocity} 
\end{equation}
see \cite{Lagrange2D} for more details. 
}

Introducing the notation $\WW{*}$ for the intermediate solution of the conservative variables vector and substituting \eqref{eqn.stflux} into \eqref{eqn.intsurf_stdiv} yields the finite volume scheme

\begin{align}
	|\Cellst_i|\WW{*}_i =|\Cellst_i|\WW{n}_i -\sum_{\Cell_j \in \mathcal{K}_i} \Zst(\WW{n}_i,\WW{n}_j,\etast_{ij}) 
	-\int_{\Gammast_i}\boldsymbol{\mathcal{T}}^{n} \etast_i \,\dt\dS 
	+ \int_{{\Cellst}_{i}^\circ} \left( \mathbf{S}\left( \WW{n},\nabla \WW{n} \right)-\Bst(\WW{n}) \cdot \nablast \WW{n} \right) \,\dt\dbx
	\label{eqn.intsurf_stdiv_interm}  
\end{align}
{with $\boldsymbol{\mathcal{T}}^{n}=\left( \btau^{n}, \alpha\nabla T^{n}\right)$ for the incompressible Navier-Stokes equations coupled with the temperature transport equation and $\boldsymbol{\mathcal{T}}^{n}=\left( \btau^{n}, \mathbf{0}\right)$ for weakly compressible flows.}

To complete the discretization of the former system, we still need to specify the approximation of the viscous and source terms. 
Regarding diffusion terms, {$\boldsymbol{\mathcal{T}}$,} we write the integral over the space-time surface $\Gammast_i$ in~\eqref{eqn.intsurf_stdiv_interm} as a sum of integrals on the space-time faces, that is, for the incompressible Navier-Stokes momentum equations we get
\begin{align*}
\int_{\Gammast_i}\btau^n \etast_i \,\dt\dS = 
\sum_{\Cell_j \in \mathcal{K}_i} \int_{\Gammast_{ij}}\btau^n \etast_{ij} \,\dt\dS= \sum_{\Cell_j \in \mathcal{K}_i} \int_{\Gammast_{ij}}\mu\left( \nabla \Vel^n + \nabla \left(\Vel^n\right)^T \right)\etast_{ij} \,\dt\dS, 
\end{align*}
while for weakly compressible flows we have
\begin{align*}
\int_{\Gammast_i}\btau^n \etast_i \,\dt\dS = 
\sum_{\Cell_j \in \mathcal{K}_i} \int_{\Gammast_{ij}}\btau^n \etast_{ij} \,\dt\dS= \sum_{\Cell_j \in \mathcal{K}_i} \int_{\Gammast_{ij}}\mu\left( \nabla \Vel^n + \nabla \left(\Vel^n\right)^T -\frac{2}{3} \left( \dive \Vel^n  \right) \mathbf{I} \right)\etast_{ij} \,\dt\dS.
\end{align*}
Then, it is necessary to compute the normal projection of the gradients and the divergence for each surface $\Gammast_{ij}$. Those derivatives are approximated by a classical Galerkin finite element approach that makes use of classical {Crouzeix}-Raviart (CR) elements on the primal grid. {More precisely: The basis functions of CR elements are constructed with the FE nodes located at the \textit{barycenters} of the edges of the primal mesh, which are the centers of the staggered edge-based \textit{dual mesh} used in our hybrid FV/FE method and which is therefore exactly the location where the discrete velocity field is defined in our method. Since the CR basis functions are well-known, the discrete  derivatives of the velocity field are immediately available within each primal element. Finally, the value of the gradient at each $\Gammast_{ij}$ is taken to be the one of the primal element $T_k$ containing the related edge, $\Gamma_{ij}$, see also \cite{BFTVC17,BBDFSVC20,HybridNNT}.} 
{An analogous procedure is followed to discretize the diffusion term in the temperature transport-diffusion equation \eqref{eqn.temperature}, i.e.,
\begin{equation*}
	\int_{\Gammast_i}\alpha \nabla T^{n} \cdot \etast_i \,\dt\dS = \sum_{\Cell_j \in \mathcal{K}_i} \int_{\Gammast_{ij}}\alpha \nabla T^{n} \cdot \etast_{ij} \,\dt\dS,
\end{equation*}
where the temperature gradient is computed at each boundary at the aid of Crouzeix-Raviart elements.}

On the other hand, the source term $\mathbf{S}$ can be split as \mbox{$\mathbf{S}(\WW{n},\nabla \WW{n})= \mathbf{S}_1(\WW{n}) + \mathbf{S}_2( \WW{n},\nabla \WW{n})$,} where $\mathbf{S}_1$ is an algebraic source that may depend on the conservative variable vector, and $\mathbf{S}_2$ depends also on the gradient of the conservative vector or on a post-processed variable gradient. For incompressible flows, 
$$\mathbf{S}_1\left( \WW{n} \right)= \begin{pmatrix} \rho\g-\rho\beta \left( \T^{n}-\T_{\mathrm{ref}}\right) \g \\ 0 \end{pmatrix}  \qquad \mathrm{and} \qquad  \mathbf{S}_2( \WW{n},\nabla \WW{n})=0,$$
so the source term $\mathbf{S}_1$ can be directly integrated on each spatial-time control volume $\Cellst_i$, considering the cell-averaged temperature computed at the previous time step. For weakly compressible flows, the source terms are given by 
$$\mathbf{S}_1\left( \WW{n} \right)= \begin{pmatrix} 0\\\rho^n \g\\0\end{pmatrix} \qquad \mathrm{and} \qquad \mathbf{S}_2\left( \WW{n},\nabla \WW{n} \right)= \begin{pmatrix} 0\\0\\-\left(\gamma-1\right)\dive \mathbf{q}^{n}\end{pmatrix}.$$
To compute the contribution of $\mathbf{S}_2$, we first need to approximate the heat flux $\dive \mathbf{q}$ using the Galerkin approach already employed to calculate viscous terms after having computed the temperature at the dual elements from the density and pressure obtained at the previous time step. 

As defined above, \eqref{eqn.intsurf_stdiv_interm} is only of first order. To increase the accuracy of the proposed methodology in space and time, we employ the second order accurate local ADER-ENO approach, {which was introduced and analyzed in \cite{BFTVC17}. Thanks to the employed staggered mesh, the spatial gradients can again be easily computed at the aid of the discrete derivatives of the Crouzeix-Raviart finite elements. In time we make use of the classical Cauchy-Kovalewskaya procedure that allows to compute a discrete time derivative in terms of the known spatial gradients. With this information we can then compute all terms in \eqref{eqn.intsurf_stdiv_interm} by the boundary-extrapolated values of the reconstruction, evolved up to the half time level, see \cite{TT02,TT05,Toro:2006a,Toro,BFTVC17,BBDFSVC20,BRMVCD21,BD22,BCDGP20} for further details on ADER schemes and their use within hybrid FV/FE methods}.
 
Finally, let us note that at each time step the use of closed space-time control volumes guarantees that the geometric conservation law (GCL) is verified, see \cite{Lagrange3D}, because
\begin{equation}
	\int_{\Gammast_i}\nst \dS = \sum_{\Cell_j \in \mathcal{K}_i}  \etast_{ij} = 0.
	\label{eqn.GCL2}
\end{equation}
As such, the numerical method is exactly \textit{free-stream preserving} by construction since all integrals of the space-time normal vectors are calculated exactly. 

\subsection{Mesh motion and choice of the mesh velocity}\label{sec:mesh_motion}
Since we are considering moving meshes, as the problem evolves in time, the elements of the primal mesh are deformed. Two different types of mesh motion are considered in this work: the primal elements are moved with a smoothed version of the local fluid velocity, or with a velocity that is prescribed from the boundaries and where the mesh velocity inside the domain is obtained by solving a Laplace equation with the boundary velocity imposed as Dirichlet boundary conditions. Then, if $\x_i^n$ is the coordinate of a finite element vertex, the coordinate of the vertex at the next time step is computed via an explicit Euler scheme as 
\begin{equation*} 
\x_i^{n+1} = \x_i^n + \Delta t\, \V_i^n,
\end{equation*} 
where $\V_i^n$ is the local mesh velocity in vertex $\x_i^n$ at time $t^n$. After moving the primal grid it is necessary to recompute the dual mesh as well as all mesh data related to the physical coordinates of the cells. As a consequence of the mesh motion, two different spatial grids are present at each time step. The first mesh, coming from the previous time step, is employed to define the space-time control volumes needed within the finite volume scheme. Meanwhile, the pressure system will only see in the newly computed spatial mesh. 
In what follows, we provide further details on the two types of mesh motion.

\subsubsection{Mesh velocity imposed by the velocity on the domain boundaries}
This is the easiest way of imposing a local mesh velocity in the context of ALE schemes and particularly suitable for future fluid structure interaction problems. The local mesh velocity is simply determined by the solution of the Laplace equation 
\begin{equation}
	  \nabla^2 \V^n = 0, \qquad \forall \x \in \Omega^n, 
	  \label{eqn.mesh.laplace} 
\end{equation} 	
with the boundary velocity $\mV^n_{BC}$ prescribed as a Dirichlet boundary condition on $\partial \Omega^n$, 
\begin{equation}
	\V^n(\x) = \mV^n_{BC} \qquad \forall \x \in \partial \Omega^n, 
	  \label{eqn.mesh.laplace.bc}
\end{equation} 	
and $\nabla^2 = \nabla \cdot \nabla$ the usual Laplace operator.
The problem \eqref{eqn.mesh.laplace}-\eqref{eqn.mesh.laplace.bc} is solved at the aid of a classical continuous $\mathbb{P}_1$ Lagrange finite element method, which is identical to the one used to solve the pressure Poisson equation used in the present hybrid FV/FE method and described in the following section. 

\subsubsection{Mesh velocity imposed by a smoothed local fluid velocity}
Since in Arbitrary-Lagrangian-Eulerian schemes the mesh velocity can, in principle, be chosen arbitrarily, we may also choose it equal to the local fluid velocity. However, as already mentioned before, in the case when the mesh velocity is set to be the local fluid velocity, an excessive distortion of the finite elements can appear after finite times, especially for flows with differential rotation and strong shear, see, e.g. \cite{Lagrange2D,Lagrange3D,GBCKSD19}. In order to avoid it, we can employ a regularization algorithm that helps to smooth the deformation.

In this case we compute the mesh velocity by solving a parabolic equation with the local fluid velocity as initial condition and the diffusion parameter chosen in such a manner as to regularize the mesh sufficiently. In this setting, the local mesh velocity is determined by the solution of the following set of semi-discrete parabolic equations  
\begin{equation}
	\V^n = \V^* + \Delta t \varsigma \nabla^2 \V^n, \qquad \forall \x \in \Omega^n,  
	\label{eqn.mesh.heat} 
\end{equation} 	
with the initial condition given by the local fluid velocity 
\begin{equation*}
	\V^* = \Vel^n 
\end{equation*}
and the local fluid velocity also prescribed as Dirichlet boundary condition on $\partial \Omega^n$, 
\begin{equation}
	\V^n(\x) = \Vel^n(\x) \qquad \forall \x \in \partial \Omega^n. 
	\label{eqn.mesh.heat.bc}
\end{equation} 	
System \eqref{eqn.mesh.heat}-\eqref{eqn.mesh.heat.bc} is again solved employing continuous $\mathbb{P}_1$ Lagrange FE and the regularization parameter is chosen according to each test problem. For $\varsigma = 0$ we retrieve 
$\V^n = \Vel^n$ and for $\varsigma \to \infty $ we obtain a Laplace-type smoothing of the mesh velocity.

\subsection{Discretization of the pressure subsystem via continuous finite elements for incompressible flows}\label{sec:numdisc.press}

Once the new mesh has been computed and the solution of the convective and diffusive subsystem is obtained, the pressure subsystem needs to be solved to update the final velocity and pressure variables. 
Let us recall that at this stage all variables live exclusively in the already moved mesh and we do not need to take into account any space-time control volume for the solution of the pressure subsystem. Therefore the methodology to be followed coincides with the projection stage of an Eulerian scheme, \cite{HybridNNT,BBDFSVC20}.  

Semi-discretization of the pressure subsystem \eqref{eqn.inc.velo.sub}  in time, keeping space still continuous, yields 
\begin{gather}
	\dive \Vel^{n+1} = 0, \label{eqn.div.sub}  \\
	\frac{\rho  \Vel^{n+1} - \rho \Vel^*}{\Delta t} + \gra \Press^{n+1} 	= 0, \label{eqn.press.sub}  
\end{gather}
where $\rho \Vel^*$ represents the momentum obtained solving the convective and viscous subsystem via the ALE finite volume scheme that has been discussed in Section \ref{sec:numdisc.ALEFV}. 
Once $\WW{*}$ has been computed, its value can be interpolated from the dual cells as the weighted average of the values at the subelements contained at each primal element.

Inserting \eqref{eqn.press.sub} into \eqref{eqn.div.sub} multiplied by the constant density $\rho$ leads to the following pressure Poisson equation that needs to be solved at each time step, as usual in projection methods: 
\begin{gather}
	\nabla^2 \Press^{n+1} = \frac{1}{\Delta t } \dive \rho \Vel^{*} .   \label{eqn.press.poisson}  
\end{gather}
Integrating \eqref{eqn.press.poisson} on the computational domain $\Omega$ multiplied by a test function $z\in V_{0}$, and applying Green's formula, we obtain the weak problem:
	\begin{align}
		&\textit{Find $ \Press^{n+1} \in V_0$ such that} \notag\\
		&\Delta t \int \limits_{\Omega} \grae \Press^{n+1} \cdot \grae z \dV = \int \limits_{\Omega} \rho \Vel^{*} \cdot \grae z \dV -  \int \limits_{\Gamma_{D}} \rho \Vel^{n+1} \cdot \mathbf{n}\, z \dS,\qquad \forall z\in V_0=\left\lbrace z\in H^{1}(\Omega): \int_{\Omega} z \dV = 0 \right\rbrace,  \label{eqn.weak.incns}  
	\end{align}
\noindent where we have taken into account that Dirichlet boundary conditions for the velocity unknown could be defined on $\Gamma_{D}\subset \partial \Omega$.
%
%
%
This weak problem is discretized using classical $\mathbb{P}_1$ Lagrange finite elements and the resulting linear system is solved using a matrix-free conjugate gradient method.

Once $\Press^{n+1}$ is computed at each primal vertex $V_{\ell}$, we interpolate it on the dual grid as the average of its value at the vertex of each dual cell. Moreover, the gradients of $\Press^{n+1}$ are computed at each dual element as a weighted average of the gradients obtained on the primal elements using the FE framework and the final momentum is obtained as
\begin{equation*}
	\WW{n+1}_{\vel} = \WW{*}_{\vel} + \Delta t\;  \nabla \Press^{n+1}.
\end{equation*} 

An alternative approach to the former scheme consists in solving for the difference between the pressure in the new time step and in the previous one, $\delta \Press^{n+1}=\Press^{n+1}-\Press^n$, instead of the pressure itself. This approach, generally known as pressure correction formulation, requires for the joint semi-discretization of the Navier-Stokes equations before performing the splitting. As a result, the gradient of the pressure at time $t^n$ is kept in the momentum equation of the convective-diffusive subsystem, see, e.g. \cite{BFTVC17}. Therefore, to compute $\WW{**}_{\vel}=\rho  \Vel^{**} :=\WW{*}_{\vel}+\int\limits_{{\Cellst}_{i}}\nabla\Press^{n}\dt\dbx$ it suffices to add the term $$\displaystyle -\int\limits_{\Cellst_i} \nabla \Press^n \dt\dbx$$ at the right hand side of \eqref{eqn.intsurf_stdiv_interm} and approximate it by transforming the integral over the control volume into a boundary integral of the normal projection of the pressure.
On the other hand, the final weak problem for the pressure, analogous to \eqref{eqn.weak.incns}, reads

\textit{Find $\delta \Press^{n+1} \in V_0$ such that}
\begin{equation}
\Delta t \int \limits_{\Omega} \grae \delta \Press^{n+1} \cdot \grae z \dV = \int \limits_{\Omega} \rho \Vel^{**} \cdot \grae z \dV -  \int \limits_{\Gamma_{D}} \rho \Vel^{n+1} \cdot \mathbf{n}\, z \dS,\qquad \forall z\in V_0. \label{eqn.weak.incnspc}  
\end{equation}

\noindent Once its solution is computed, the pressure and momentum are obtained as 
$$	\Press^{n+1} = \Press^{n}  + \delta \Press^{n+1}, \qquad \WW{n+1}_{\vel} = \WW{**}_{\vel} + \Delta t \; \nabla\delta \Press^{n+1}.$$

{In general, the main advantage of the pressure correction approach is that it performs better when combining an explicit treatment for the convective terms with an implicit treatment of the viscous terms, as presented in \cite{HybridNNT} within the Eulerian framework. It furthermore corresponds to the original Eulerian formulation of the hybrid FV/FE scheme presented in \cite{BFTVC17}. However, in the framework of this work, where both convective and viscous terms are discretized explicitly, no significant difference is observed between using the pressure correction formulation, or not. 
}

\subsection{Semi-discrete kinetic energy preservation and energy stability of the ALE FV scheme}

In this section, we consider the inviscid case without gravity by setting $\mu = 0$, $\beta = 0$ and $\mathbf{g}=0$. 
We further assume that the density $\rho$ is a global constant in space and time. 

\begin{theorem} 
\label{thm.ale.ducros}
Assuming that the discrete velocity field is divergence free in the sense 
\begin{equation}
	\sum \limits_{\Cell_j \in \mathcal{K}_i} |\Gamma_{ij}| u_{ij} = 0, \qquad 
	\Un_{ij} = \frac{1}{|\Gamma_{ij}|} 
	\int \limits_{\Gamma_{ij}} 
	\halb \left( \vel_i + \vel_j \right) \cdot \mathbf{n}_{ij} \, \dS, \qquad
	 |\Gamma_{ij}| = \int \limits_{\Gamma_{ij}} \dS, 	   	
	\label{eqn.divfree} 
\end{equation}
with $\mathbf{n}_{ij} = -\mathbf{n}_{ji}$ the unit normal vector pointing from cell $\Cell_i$ to cell $\Cell_j$ and assuming that the semi-discrete geometric conservation law (GCL) with edge-averaged normal mesh velocity $\mVn_{ij}$  
\begin{equation}
	\frac{\partial  |\Cell_i(t)| }{\partial t} - \sum \limits_{\Cell_j \in \mathcal{K}_i} 
	|\Gamma_{ij}| \mVn_{ij} = 0, 
	\qquad
	\mVn_{ij} = \frac{1}{|\Gamma_{ij}|} \int \limits_{\Gamma_{ij}} \mV \cdot \mathbf{n}_{ij}  \, \dS,  
	\qquad
|\Cell_{i}| = \int \limits_{\Cell_{i}(t)} \dbx 	
	\label{eqn.gcl2} 
\end{equation}
is satisfied, then the \textit{semi-discrete}  ALE FV scheme with the Ducros flux, \cite{Ducros1999,Yee1999,Ducros2000,Moin2009,Pirozzoli2010},
\begin{equation}
	 \frac{\partial  |\Cell_i(t)| \rho \vel_i  }{\partial t} + \sum \limits_{\Cell_j \in \mathcal{K}_i} 
	 |\Gamma_{ij}| \left( \Un_{ij} - \mVn_{ij} \right) \halb \rho \left( \vel_{i} + \vel_{j} \right) = 0  
	 \label{eqn.ducros2} 
\end{equation}
verifies a discrete conservation law for the kinetic energy density in the sense that 
\begin{equation}
	 	\frac{\partial  |\Cell_i(t)| \halb \rho \vel_i^2  }{\partial t} + \sum \limits_{\Cell_j \in \mathcal{K}_i} 
	 	|\Gamma_{ij}| \left( \Un_{ij} - \mVn_{ij} \right) \halb \rho \vel_{i} \cdot \vel_{j} = 0,
	 	\label{eqn.kinen}   
\end{equation}
where $f^{\rho k}_{ij} = \left( \Un_{ij} - \mVn_{ij} \right) \halb \rho \vel_{i} \cdot \vel_{j}$ is a consistent and symmetric approximation of the ALE numerical flux of the kinetic energy. 
\end{theorem}   
\begin{proof}
Multiplication of the discrete momentum conservation law \eqref{eqn.ducros2} with $\vel_i$ yields 
\begin{equation*}
	\vel_i \cdot \frac{\partial  |\Cell_i(t)| \rho \vel_i  }{\partial t} + \vel_i \cdot \sum \limits_{\Cell_j \in \mathcal{K}_i} 
	|\Gamma_{ij}| \left( \Un_{ij} - \mVn_{ij} \right) \halb \rho \left( \vel_{i} + \vel_{j} \right) = 0.  
\end{equation*}
Since $\vel_i$ does not depend on the index $j$ we can rewrite the above equation as 
\begin{equation*}
	\vel_i \cdot \frac{\partial  |\Cell_i(t)| \rho \vel_i  }{\partial t} + 
	\halb \rho \vel_{i}^2 
	\sum \limits_{\Cell_j \in \mathcal{K}_i} 
	|\Gamma_{ij}| \left( \Un_{ij} - \mVn_{ij} \right)   +
	 \sum \limits_{\Cell_j \in \mathcal{K}_i} 
	|\Gamma_{ij}| \left( \Un_{ij} - \mVn_{ij} \right) \halb \rho \vel_{i} \cdot \vel_{j} 
	= 0,  
\end{equation*}
which after using the product rule on the first term reduces to 
\begin{equation}
	\rho \vel_i \cdot  \vel_i \frac{\partial  |\Cell_i(t)|   }{\partial t} + 
	\vel_i \cdot |\Cell_i(t)| \frac{\partial   \rho \vel_i  }{\partial t}
	- 
	\halb \rho \vel_{i}^2 
	\frac{\partial  |\Cell_i(t)|   }{\partial t}    +
	\sum \limits_{\Cell_j \in \mathcal{K}_i} 
	|\Gamma_{ij}| \left( \Un_{ij} - \mVn_{ij} \right) \halb \rho \vel_{i} \cdot \vel_{j} 
	= 0,  
	\label{eqn.aux1}
\end{equation}
due to the incompressibility condition \eqref{eqn.divfree} and the GCL \eqref{eqn.gcl2}. Since 
\begin{equation*}
	\frac{\partial}{\partial t} \left( |\Cell_i(t)| \halb \rho \vel_i^2 \right)  = 
	\halb \rho \vel_i^2 \, 
	\frac{\partial  |\Cell_i(t)| }{\partial t}   + 
	|\Cell_i(t)| \frac{\partial}{\partial t} \left(  \halb \rho \vel_i^2 \right) = 
	\halb \rho \vel_i^2 \, 
	\frac{\partial  |\Cell_i(t)| }{\partial t}   + 
	\vel_i \cdot 
	|\Cell_i(t)| \frac{\partial   \rho \vel_i  }{\partial t}.
\end{equation*}
Equation~\eqref{eqn.aux1} finally becomes 
\begin{equation}
	\frac{\partial}{\partial t} \left( |\Cell_i(t)| \halb \rho \vel_i^2 \right)    +
	\sum \limits_{\Cell_j \in \mathcal{K}_i} 
	|\Gamma_{ij}| \left( \Un_{ij} - \mVn_{ij} \right) \halb \rho \vel_{i} \cdot \vel_{j} 
	= 0,  
	\label{eqn.aux2}
\end{equation}   
which is the sought discrete kinetic energy conservation law \eqref{eqn.kinen} and which completes the proof.  
\end{proof}

A direct consequence of Theorem \ref{thm.ale.ducros} is that for vanishing boundary fluxes of the kinetic energy the scheme is kinetic energy preserving. Indeed, summation over all elements yields 
\begin{equation*}
	\int \limits_{\Cell_i} \frac{\partial}{\partial t} \left( \halb \rho \vel^2 \right) \dbx = 
	\sum_i \frac{\partial  |\Cell_i(t)| \halb \rho \vel_i^2  }{\partial t} = - \sum_i \sum \limits_{\Cell_j \in \mathcal{K}_i}  
	|\Gamma_{ij}| \left( \Un_{ij} - \mVn_{ij} \right) \halb \rho \vel_{i} \cdot \vel_{j} = 0,  
\end{equation*}
since the internal fluxes cancel due to the obvious properties $\Un_{ij} = -\Un_{ji}$ and $\mVn_{ij} = -\mVn_{ji}$ and the boundary fluxes of kinetic energy are assumed to vanish. 

In the following we add a Rusanov-type numerical dissipation to the Ducros flux obtaining the scheme 
\begin{equation}
	\frac{\partial  |\Cell_i(t)| \rho \vel_i  }{\partial t} 
	+ \sum \limits_{\Cell_j \in \mathcal{K}_i} 
	|\Gamma_{ij}| \left( \Un_{ij} - \mVn_{ij} \right) \halb \rho \left( \vel_{i} + \vel_{j} \right) = 
	\sum \limits_{\Cell_j \in \mathcal{K}_i} 
	|\Gamma_{ij}| \halb s_{\max} \rho \left( \vel_{j} - \vel_{i} \right),   
	\label{eqn.ducros3} 
\end{equation}
where $s_{\max} \geq 0 $ is an estimate for the maximum local signal velocity. 
We now show that the scheme \eqref{eqn.ducros3} is energy dissipative. To that end, it is sufficient to analyze the new terms on the right hand side of the equation, since the terms on the left hand side immediately lead to \eqref{eqn.aux2}. Multiplication of 	\eqref{eqn.ducros3} with $\vel_i$ yields 
\begin{gather*}
	\frac{\partial}{\partial t} \left( |\Cell_i(t)| \halb \rho \vel_i^2 \right)    +
	\sum \limits_{\Cell_j \in \mathcal{K}_i} 
	|\Gamma_{ij}| \left( \Un_{ij} - \mVn_{ij} \right) \halb \rho \vel_{i} \cdot \vel_{j}  = \sum \limits_{\Cell_j \in \mathcal{K}_i} 
	|\Gamma_{ij}| \halb s_{\max} \rho \left( \vel_{j} - \vel_{i} \right) \cdot \vel_i,  \\
	= \sum \limits_{\Cell_j \in \mathcal{K}_i} 
	|\Gamma_{ij}| \halb s_{\max} \rho \left( \vel_{j} - \vel_{i} \right) \cdot \halb \left( \vel_i + \vel_j \right) +
	\sum \limits_{\Cell_j \in \mathcal{K}_i} 
	|\Gamma_{ij}| \halb s_{\max} \rho \left( \vel_{j} - \vel_{i} \right) \cdot \halb \left( \vel_i - \vel_j \right) \\
	= \sum \limits_{\Cell_j \in \mathcal{K}_i} 
	|\Gamma_{ij}| \halb s_{\max} \halb \rho \left( \vel_{j}^2 - \vel_{i}^2 \right) -
	\sum \limits_{\Cell_j \in \mathcal{K}_i} 
	|\Gamma_{ij}| \halb s_{\max} \halb \rho \left( \vel_{j} - \vel_{i} \right)^2.   	
\end{gather*}
This means the evolution equation of the kinetic energy density is 
\begin{gather*}
	\frac{\partial}{\partial t} \left( |\Cell_i(t)| \halb \rho \vel_i^2 \right)    +
	\sum \limits_{\Cell_j \in \mathcal{K}_i} 
	|\Gamma_{ij}| \left( \Un_{ij} - \mVn_{ij} \right) \halb \rho \vel_{i} \cdot \vel_{j}  - \sum \limits_{\Cell_j \in \mathcal{K}_i} 
	|\Gamma_{ij}| \halb s_{\max} \halb \rho \left( \vel_{j}^2 - \vel_{i}^2 \right) = 
 \\ 
	-
	\sum \limits_{\Cell_j \in \mathcal{K}_i} 
	|\Gamma_{ij}| \halb s_{\max} \halb \rho \left( \vel_{j} - \vel_{i} \right)^2 \leq 0.   	
\end{gather*}
The third term on the left hand side corresponds to the dissipative kinetic energy flux and the sign of the term on the right hand side in the previous equation is known, hence energy stability is now easily proven by simply summing up over all elements. Assuming no fluxes through the boundary the sum over all fluxes on the left hand side cancels and we thus obtain
\begin{equation*}
	\int \limits_{\Cell_i} \frac{\partial}{\partial t} \left( \halb \rho \vel^2 \right) \dbx = 
	\sum_i \frac{\partial  |\Cell_i(t)| \halb \rho \vel_i^2  }{\partial t} = 
	-
	\sum \limits_{\Cell_j \in \mathcal{K}_i} 
	|\Gamma_{ij}| \halb s_{\max} \halb \rho \left( \vel_{j} - \vel_{i} \right)^2 \leq 0,  
\end{equation*}
which means that the integral over the kinetic energy is non increasing in time for $s_{\max} > 0$.

\subsection{Discretization of the pressure subsystem for weakly compressible flows} \label{sec:weakly}
To conclude the extension of the proposed ALE hybrid FV/FE methodology to deal with weakly compressible flows, we proceed like for the incompressible case and discretize the pressure subsystem \eqref{eqn.cns.press.sub} using classical finite elements. Performing the semi-discretization in time of \eqref{eqn.cns.press.sub} one gets
\begin{gather}
	\frac{\rho  \Vel^{n+1} - \rho \Vel^*}{\Delta t}  + \grae \Press^{n+1} = 0, \label{eqn.weak.momentum.p.sdt}  \\
	\frac{\Press^{n+1} - \Press^*}{\Delta t}   - c^2 \Vel^{n}\cdot\grae  \rho^{n} + c^2 \dive \left(\rho\Vel^{n+1}\right) - \left(\gamma-1\right)  \btau^{n}  : \gra \Vel^{n} = 0. \label{eqn.weak.pressure.p.sdt}  
\end{gather}
Hence, substitution of \eqref{eqn.weak.momentum.p.sdt} into \eqref{eqn.weak.pressure.p.sdt} leads to the pressure equation 
\begin{gather}
	- \Delta t^2\,\, \nabla^2 \Press^{n+1} 
	+ \frac{1}{c^2}\Press^{n+1}   
	=\frac{1}{c^2} \Press^*
	+ \Delta t \, \Vel^{n}\cdot\grae  \rho^{n} 
	- \Delta t  \dive \rho \Vel^* 
	+ \frac{\Delta t}{c^2}\left(\gamma-1\right)  \btau^{n}  : \gra \Vel^{n} .
	\label{eqn.press.poisson.weak} 
\end{gather}
The corresponding weak problem reads
	\begin{align}
		&\textit{Find $\Press^{n+1} \in V_0$ such that}& \notag\\
		&\Delta t^2\int\limits_{\Omega}\grae \Press^{n+1}\cdot \grae  z \dV
		+  \frac{1}{c^2} \int\limits_{\Omega}\Press^{n+1}   z \dV
		=\frac{1}{c^2}\int\limits_{\Omega} \Press^{*}  z \dV 
		+ \Delta t \int\limits_{\Omega}  \Vel^{n}\cdot\grae  \rho^{n} z \dV&
		\notag\\
		&\qquad+ \Delta t \int\limits_{\Omega}  \rho \Vel^* \cdot \grae z \dV
		- \Delta t \int\limits_{\Gamma_{D}}  \rho \Vel^{n+1}\cdot\mathbf{n}\, z \dS
		+\frac{\Delta t}{c^2}  \int\limits_{\Omega}\left(\gamma-1\right)  \btau^{n}  : \gra \Vel^{n} z \dV, \quad &\forall z\in V_0,\; \Gamma_{D}\in\partial\Omega, \label{eqn.press.poisson.weakform} 
	\end{align}

\noindent where we have multiplied \eqref{eqn.press.poisson.weak} by a test function $z\in V_0$, integrated over the computational domain and applied Green's formula.
Like in the incompressible case, this system can be discretized using $\mathbb{P}_{1}$ Lagrange finite elements. 
Let us note that, to define the right hand side of \eqref{eqn.press.poisson.weakform}, we need to interpolate data between the staggered grids. Since the density unknown is computed on the dual grid, we can calculate the gradient of the density using {Crouzeix}-Raviart elements and assume a constant value for the velocity at each primal element, as it would be done in a finite difference interpretation of the data coming from the dual grid. On the other hand, the intermediate pressure, $\Press^{*}$, can be interpolated from the dual grid to get its approximation at each primal element. Hence,  we can define an updated intermediate pressure as
\begin{equation*}
	\Press^{**}  = \Press^{*}  
	+ \Delta t c^2 \Vel^{n}\cdot\grae  \rho^{n}
\end{equation*}
that can be used to compute the first two terms on the right hand side of \eqref{eqn.press.poisson.weakform} on an element-wise manner. Finally, a Galerkin approach is used to compute the gradients involved in the last term of the equation providing also a constant value per primal element that needs to be integrated, see \cite{BBDFSVC20}.

\subsection{Boundary conditions}\label{sec:bc}
An important point that still needs to be addressed is the treatment of boundary conditions. In the numerical results shown in Section \ref{sec:numericalresults}, we will make use of the following types:
\begin{itemize}
	\item Periodic boundary conditions. The periodic primal vertices are merged to compute the pressure solution and the mesh deformation. For the finite volume scheme, we keep the computations on the original dual grid containing boundary-type elements, and we then sum the contributions of each pair of neighbours through a periodic face that corresponds to a new interior cell. 
	
	\item Dirichlet boundary conditions. 
	The given momentum, $\rho \vel_{\mathrm{BC}}$, is weakly imposed in the finite volume stage by computing the convective flux and the non conservative products using the velocity coming from the interior, $\rho\vel_{i}$, and
	\begin{equation*}
		\rho\vel_{j} = 2 \rho \vel_{\mathrm{BC}} - \rho\vel_{i}.
	\end{equation*}
	The exact velocity at the barycenter of the boundary face is also employed to compute the gradients involved in the viscous terms. 
	Moreover, this exact value at the boundary face is set in the boundary integral of the pressure system, i.e. in the term $\int_{\Gamma_{\! D}}\rho\vel^{n+1}\cdot\mathbf{n} z \dS$ of \eqref{eqn.weak.incns}, \eqref{eqn.weak.incnspc} or \eqref{eqn.press.poisson.weakform}. 
	If Dirichlet boundary conditions are also selected for the pressure, its value is imposed at the corresponding vertex in the matrix-vector product of the conjugate gradient algorithm.	
	In the incompressible case, if the Boussinesq model is activated, we can choose between defining weak boundary conditions for the temperature or an adiabatic flux through the boundary.
	
	\item Viscous wall boundary conditions. In a first stage, regarding the finite volume and finite element schemes, this boundary condition is equivalent to setting weak Dirichlet  boundary conditions for the momentum unknown and Neumann boundary conditions for the pressure field. 
	However, once the intermediate velocity has been corrected with the pressure gradient, the velocity at the boundary is imposed to be exactly the one given in the boundary condition. 
	As it has already been mentioned in Section \ref{sec:mesh_motion}, if the mesh velocity is set according to a prescribed velocity on the domain boundaries, we employ it to define the velocity $\mV_{BC}$ in \eqref{eqn.mesh.laplace}-\eqref{eqn.mesh.laplace.bc}. 
	In case the mesh is moved according to the local fluid velocity, we also set the velocity of the boundary elements to match the boundary condition.
	
	\item Inviscid wall boundary conditions. The flux term and the non-conservative jumps are calculated by cancelling the normal component of the momentum using:	
	\begin{equation*}
		\rho\vel_{j} = \rho\vel_{i} - 2 \left( \rho\vel_{i}\cdot \nst\right)  \nst.
	\end{equation*}	
	Meanwhile, the momentum at the boundary face employed to approximate the gradients in the stress tensor are corrected with the given normal velocity, $\rho\vel_{i} = \rho\vel_{i} + \left( \left(\rho\vel_{\mathrm{BC}} -\rho\vel_{i}\right) \cdot \nst\right)  \nst $.
	Neumann boundary conditions are defined for the pressure field. 
	Similarly to what is done for viscous walls, the final momentum at the boundary is modified at the end of each time step to ensure that its normal component is the one prescribed for that boundary. 
	Besides, the mesh velocity at the boundary used to compute the mesh deformation is 
	\begin{equation*}
		\mV_{i} = \Vel_{i} + \left( \left(\mV_{\mathrm{BC}} -\Vel_{i}\right) \cdot \nst\right)  \nst
	\end{equation*}
	with $\mV_{\mathrm{BC}}$ the prescribed velocity at the boundary.
	
	\item Pressure outlet boundary conditions. The given pressure is imposed in the pressure subsystem, and for the momentum we simply set 
	$\vel_i^* = \vel_i^n$ for the boundary control volume $C_i$. 
	If a free surface condition is selected, the velocity is left completely free at the boundary allowing its movement. In this case, the velocity imposed at the boundary nodes to compute the mesh at a new time is the one obtained at the previous time step. 
	For fixed boundaries, the velocity of the mesh is imposed to match the one given by the boundary condition, i.e. zero for non-moving boundaries and the boundary velocity for the moving ones.	
\end{itemize}

\section{Numerical results}\label{sec:numericalresults}
We now validate the proposed numerical method by using several classical benchmarks with known analytical solution or available reference data. Unless stated the contrary, all test cases have been run using the second order approach with the Ducros flux function for the convective subsystem associated to the incompressible model and the Rusanov flux for the weakly compressible case. Besides, the time step has been dynamically determined according to the CFL condition
\begin{equation*}
	\Dt = \min_{C_{i}} \left\lbrace \Dt_{i}\right\rbrace, \qquad \Dt_{i} = \textnormal{CFL} \frac{r_{i}^2}{( \left|\zeta\right|_{\max} + c_\alpha) r_{i} +\max\left\lbrace 2\left|\nu\right|_{\max},\frac{{c_p}}{\kappa\rho}  \right\rbrace },
\end{equation*}
with $\left|\zeta\right|_{\max}$ and $\left|\nu\right|_{\max}$ the maximum absolute eigenvalues due to convection and diffusion respectively, $r_i$ the incircle diameter of each dual element and CFL$=0.5$. The GCL law~\eqref{eqn.GCL2} has been numerically verified up to machine precision for all test cases at each time step.

\subsection{Convergence study} \label{sec:convergence}
The accuracy of the scheme is analysed using a stationary incompressible vortex, in the domain \mbox{$\Omega=\left[0,10\right]^2$}, defined as
\begin{gather*}
	\rho(\x,t) 	= 1,
	\quad u_1 	= -r e^{-\halb(r^2-1)}\sin(\phi), 
	\quad u_2 	= r e^{-\halb(r^2-1)} \cos(\phi),\\
	\quad \press = -\halb e^{-(r^2-1)},\quad 
	\phi     	= \mathrm{atan}(y-5,x-5),
\end{gather*}
with $r$ the distance to the center of the computational domain, $r = \sqrt{(x-5)^2+(y-5)^2 }$. We run the ALE hybrid FV/FE method on a set of refined grids, see Table~\ref{tab.conv_mesh}, considering Dirichlet boundary conditions for all variables. The mesh is moving with the local fluid velocity, setting the smoothing parameter to  $\varsigma=0$. Let us note that, instead of using the CFL condition, the time step is fixed for each simulation, according to the initial mesh size. The obtained $L^{2}$ errors at time $t=0.1$ are reported in Table~\ref{tab.conv_errors}. We observe that the expected second order of accuracy is attained for both the pressure and velocity unknowns. 
The mesh deformation of grid $M_2$ at different time instants is depicted in Figure~\ref{fig:mesh_movementM2}.

\begin{table}[th]
		\caption{Main features of the primal grids employed to obtain the convergence table for the incompressible vortex and corresponding time steps. } \label{tab.conv_mesh}
	\renewcommand{\arraystretch}{1.2}
	\begin{center}		
		\begin{tabular}{ccccc}
			\hline 
			Mesh & Elements & Vertices & Dual elements & $\Dt$\\\hline
			$M_1$ & $512 $ & $289 $ & $800 $ & $1.25\cdot 10^{-2} $\\ 
			$M_2$ & $2048 $ & $1089 $ & $3136 $ & $6.125\cdot 10^{-3} $\\ 
			$M_3$ & $8192 $ & $4225 $ & $12416 $ & $3.125\cdot 10^{-3} $\\ 
			$M_4$ & $32768 $ & $16641 $ & $49408 $ & $1.5625\cdot 10^{-3} $\\ 
			$M_5$ & $131072 $ & $66049 $ & $197120 $ & $7.8125\cdot 10^{-4} $\\ 
			$M_6$ & $524288 $ & $263169 $ & $787456 $ & $3.90625\cdot 10^{-4} $\\ 
			$M_7$ & $2097152$ & $1050625 $ & $3147776 $ & $1.953125\cdot 10^{-4} $\\ 
			\hline 
		\end{tabular}
	\end{center}
\end{table}

\begin{table}[ht]
		\caption{Convergence study in $L^{2}$ error norm for the incompressible vortex at $t=0.1$.} \label{tab.conv_errors} 	
	\renewcommand{\arraystretch}{1.2}
	\begin{center}
		\begin{tabular}{ccccc}
			\hline 
			Mesh &$L^{2}_{\Omega}\left(\press\right)$ & $\mathcal{O}\left(\press\right)$& $L^{2}_{\Omega}\left(\vel\right)$                   
			& $\mathcal{O}\left(\vel\right)$ \\ \hline
			$M_1$ & $8.92\cdot 10^{-2}$ &        & $9.67\cdot 10^{-2}$ & \\
			$M_2$ & $3.00\cdot 10^{-2}$ & $1.57$ & $1.75\cdot 10^{-2}$ & $2.46$ \\
			$M_3$ & $7.76\cdot 10^{-3}$ & $1.95$ & $3.71\cdot 10^{-3}$ & $2.24$ \\
			$M_4$ & $2.12\cdot 10^{-3}$ & $1.87$ & $7.93\cdot 10^{-4}$ & $2.23$ \\
			$M_5$ & $5.66\cdot 10^{-4}$ & $1.90$ & $1.81\cdot 10^{-4}$ & $2.13$ \\
			$M_6$ & $1.46\cdot 10^{-4}$ & $1.95$ & $4.36\cdot 10^{-5}$ & $2.05$ \\
			$M_7$ & $3.67\cdot 10^{-5}$ & $1.99$ & $1.08\cdot 10^{-5}$ & $2.01$ \\
			\hline 
		\end{tabular}
	\end{center}
\end{table}

\begin{figure}[h!]
	\begin{center}
	\includegraphics[width=0.4\linewidth]{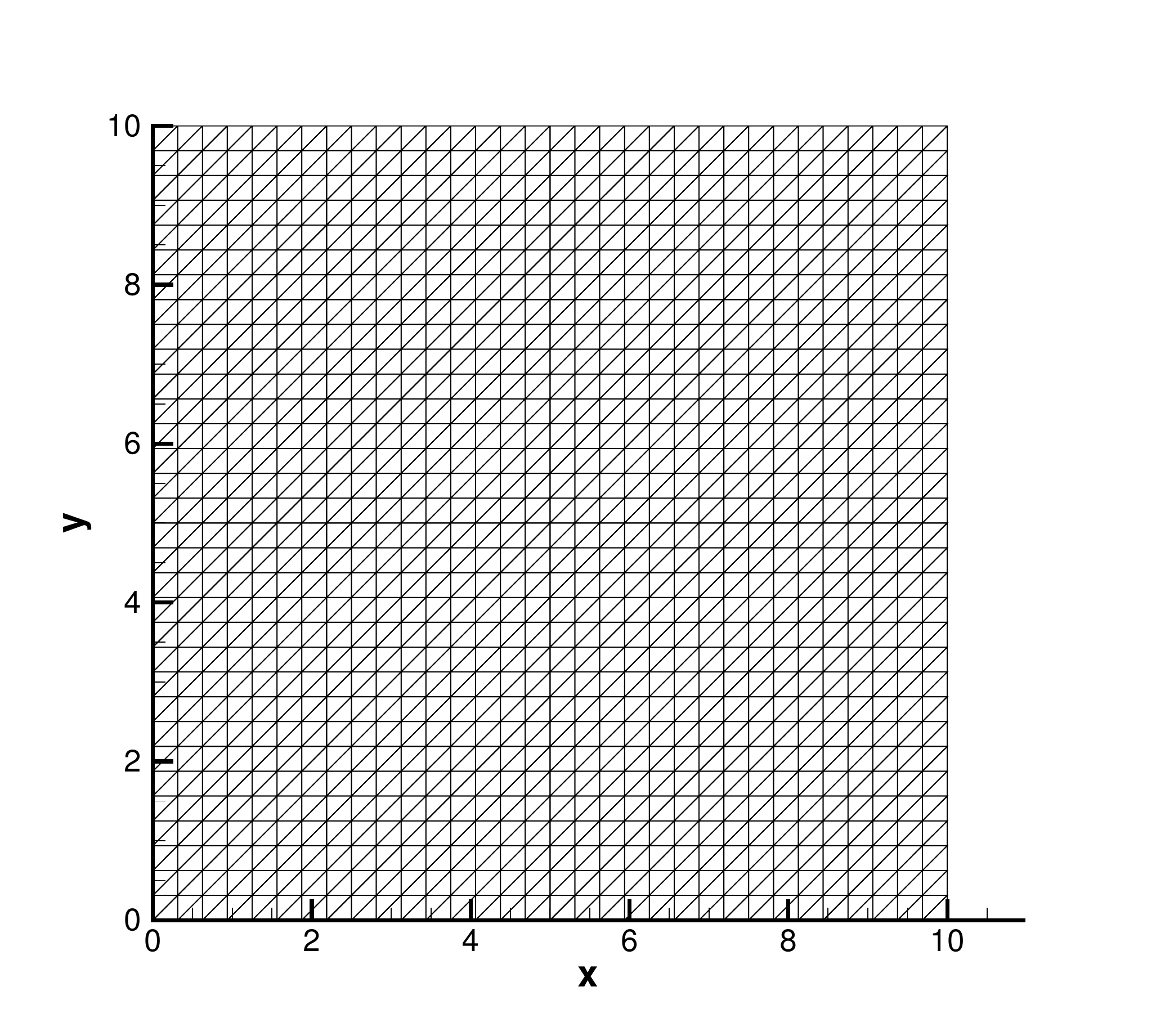}
	\includegraphics[width=0.4\linewidth]{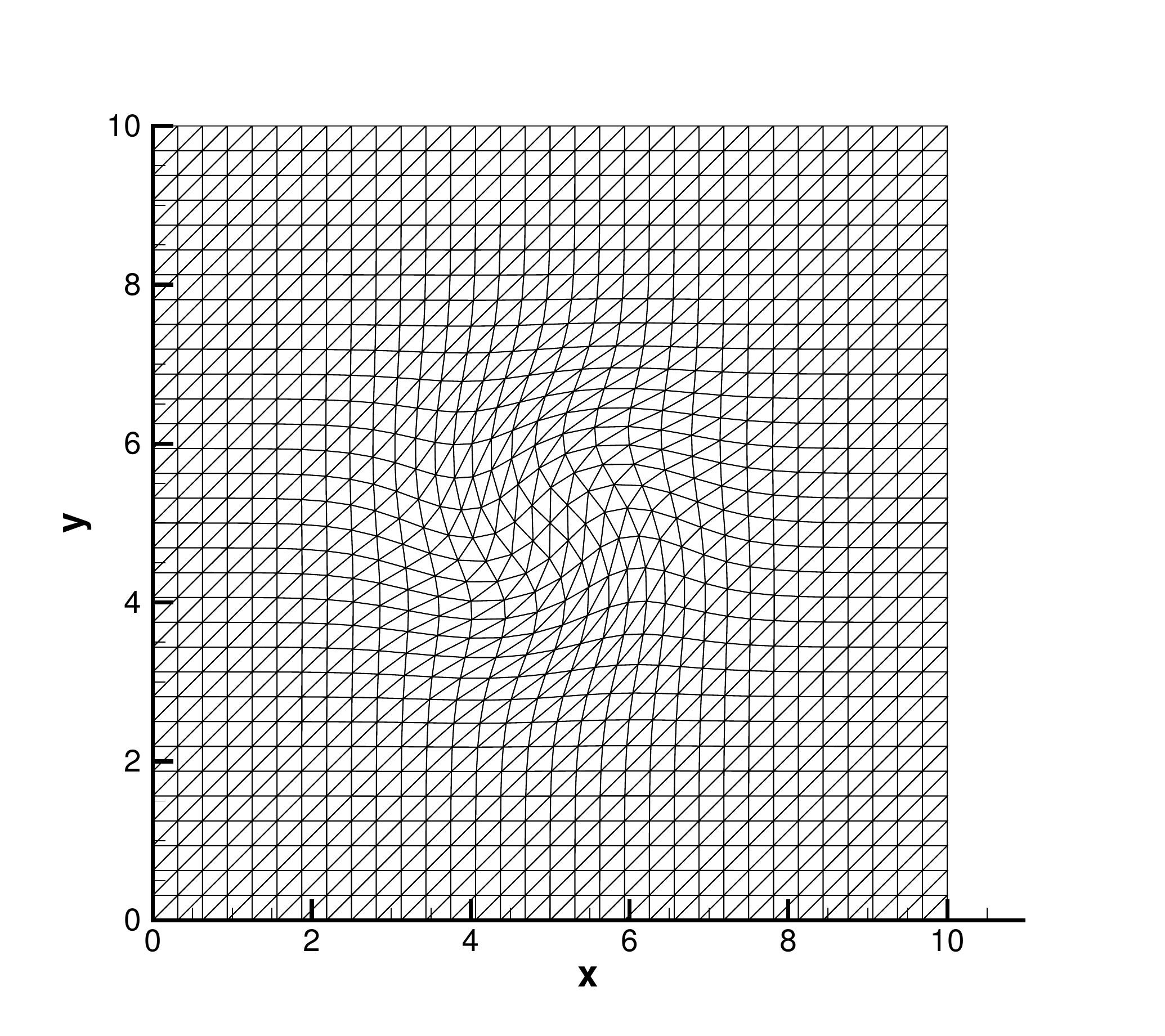}\\
	\includegraphics[width=0.4\linewidth]{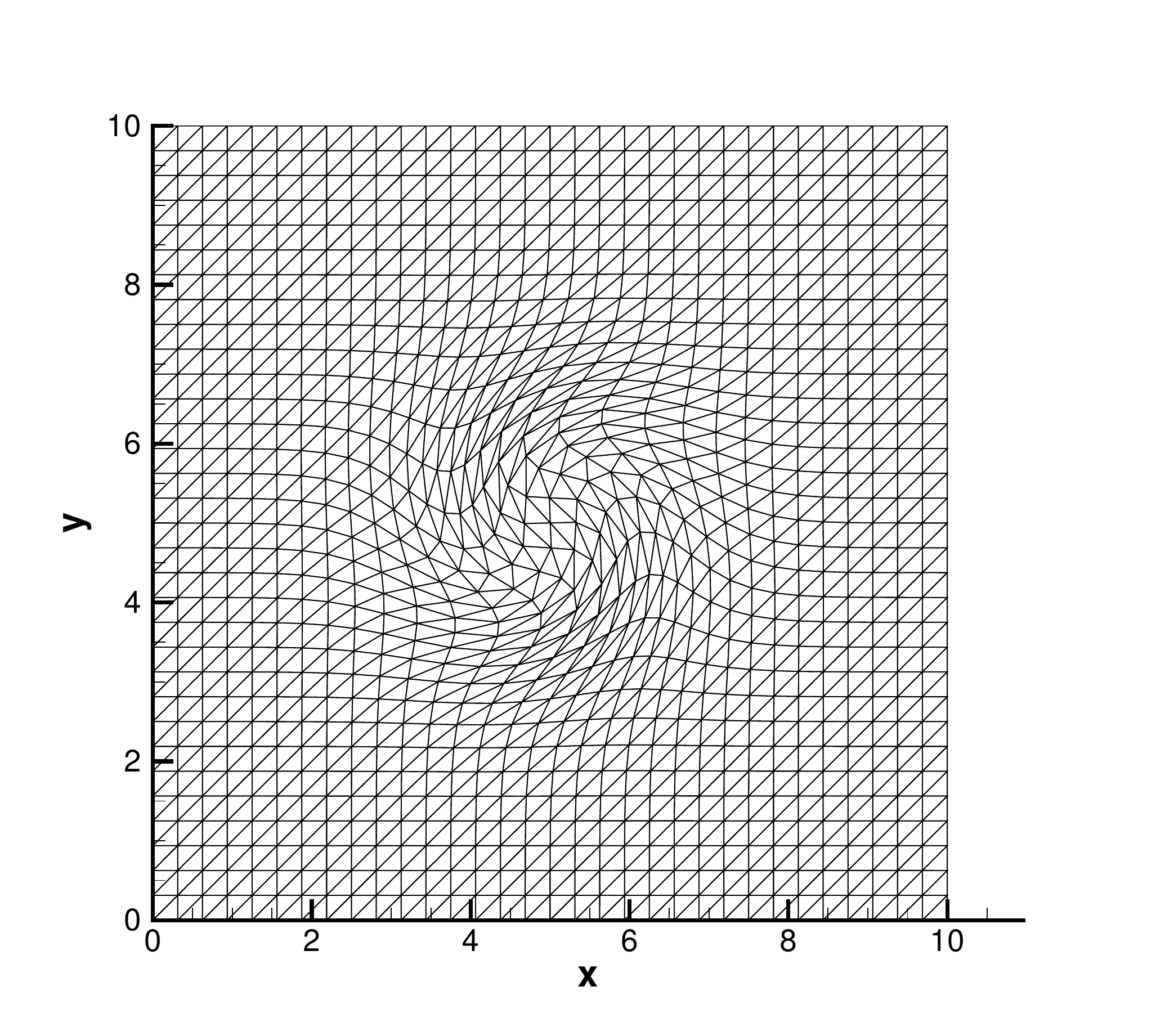}
	\includegraphics[width=0.4\linewidth]{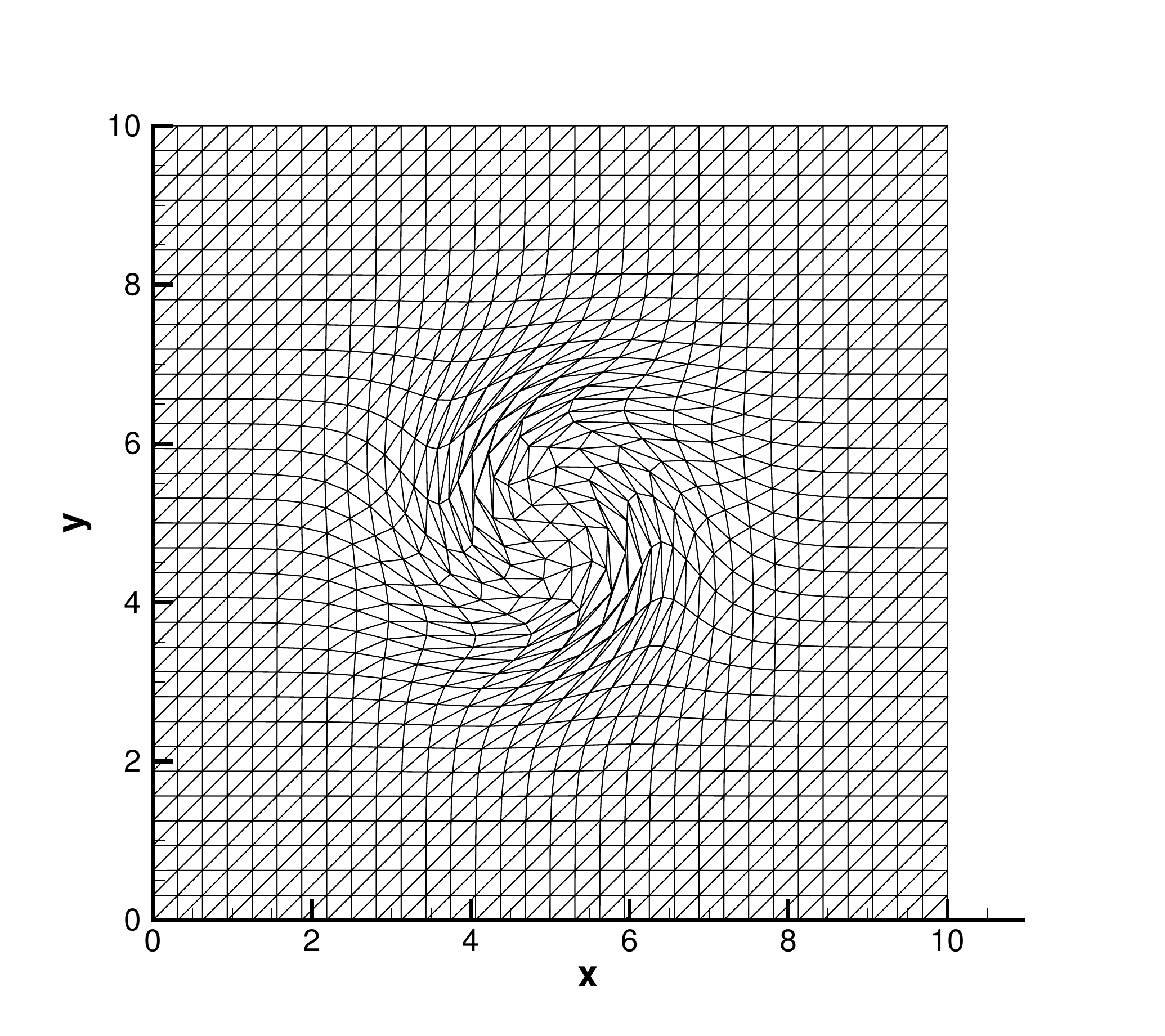}
	\end{center}
	\caption{Motion of mesh $M_2$ for the incompressible vortex test. From top left to bottom right: $t=0$, $t=0.5$, $t=1$, and $t=1.5$.}
	\label{fig:mesh_movementM2}
\end{figure}

{
	The vortex benchmark problem is also employed to numerically demonstrate	the asymptotic preserving property (AP) of the hybrid ALE FV/FE scheme for weakly compressible flows in the incompressible limit $M \to 0$. To this end, we run a set of simulations for decreasing Mach number with $\rho=1$, $p=\frac{p_0}{\gamma}-\halb e^{-(r^2-1)}$, \mbox{$p_0 \in \left\lbrace 1, 10^{2}, 10^{4}, 10^{6}, 10^{8}, 10^{10}, 10^{12}\right\rbrace$,} and compute the corresponding errors and convergence rates, see Table \ref{tab:APanalysis}. The method remains second order accurate for all Mach numbers. The results reported in Table \ref{tab:APanalysis} therefore clearly show the asymptotic preserving (AP) property of the new hybrid ALE FV/FE scheme for weakly compressible flows in the incompressible limit when $M \to 0$. }

\begin{table}[ht!]
	\renewcommand{\arraystretch}{1.2}
	\begin{center}
		{\caption{Numerical analysis of the asymptotic preserving property (AP) of the weakly compressible hybrid ALE FV/FE scheme for the vortex test case in the incompressible limit with \mbox{ $M \in\left\lbrace 1, 10^{-1}, 10^{-2}, 10^{-3}, 10^{-4}, 10^{-5},  10^{-6} \right\rbrace$ at time $t=0.1$.}}\label{SWV_Fr}\label{tab:APanalysis}} 
		
		\vspace{4pt}		
		{\begin{tabular}{ccccccc}
				\hline 
				Mesh &$L^{2}_{\Omega}\left(\rho\right)$ & $\mathcal{O}\left(\rho \right)$                  
				&$L^{2}_{\Omega}\left(\rho\vel \right)$ & $\mathcal{O}\left(\rho\vel \right)$ &  $L^{2}_{\Omega}\left(\press\right)$ & $\mathcal{O}\left(\press\right)$ \\ 
				\hline 
				\multicolumn{7}{c}{ $M \sim 1$,\hspace{0.2cm} $p_0 = 10^0$}
				\\ \hline 
				16  & $4.93 \cdot 10^{-3}$ & $    $ & $1.18 \cdot 10^{-2} $ & $    $ & $1.64 \cdot 10^{-3}$ & $   $ \\
				32  & $1.41 \cdot 10^{-3}$ & $1.8 $ & $3.18 \cdot 10^{-3} $ & $1.9 $ & $5.34 \cdot 10^{-4}$ & $1.6$ \\
				64  & $4.11 \cdot 10^{-4}$ & $1.8 $ & $8.26 \cdot 10^{-4} $ & $1.9 $ & $2.25 \cdot 10^{-4}$ & $1.2$ \\
				128 & $1.36 \cdot 10^{-4}$ & $1.6 $ & $2.26 \cdot 10^{-4} $ & $1.9 $ & $1.04 \cdot 10^{-4}$ & $1.1$ \\
				\hline 				
				\multicolumn{7}{c}{$M \sim 10^{-1}$,\hspace{0.2cm} $p_0 = 10^2$}
				\\ \hline 
				16  & $4.91 \cdot 10^{-3}$ & $   $ & $1.08 \cdot 10^{-2} $ & $   $ & $2.84\cdot 10^{-2}$ & $   $ \\
				32  & $1.39 \cdot 10^{-3}$ & $1.8$ & $2.71 \cdot 10^{-3} $ & $2.0$ & $1.13\cdot 10^{-2}$ & $1.3$ \\
				64  & $4.04 \cdot 10^{-4}$ & $1.8$ & $7.03 \cdot 10^{-4} $ & $2.0$ & $3.22\cdot 10^{-3}$ & $1.8$ \\
				128 & $1.34 \cdot 10^{-4}$ & $1.6$ & $1.99 \cdot 10^{-4} $ & $1.8$ & $8.35\cdot 10^{-4}$ & $2.0$ \\
				\hline 
				\multicolumn{7}{c}{$M \sim 10^{-2}$,\hspace{0.2cm} $p_0 = 10^4$}
				\\ \hline 
				16  & $4.89 \cdot 10^{-3}$ & $   $ & $1.11 \cdot 10^{-2} $ & $   $ & $4.65 \cdot 10^{-1}$ & $   $ \\
				32  & $1.38 \cdot 10^{-3}$ & $1.8$ & $2.83 \cdot 10^{-3} $ & $2.0$ & $1.33 \cdot 10^{-1}$ & $1.8$ \\
				64  & $3.96 \cdot 10^{-4}$ & $1.8$ & $7.39 \cdot 10^{-4} $ & $1.9$ & $3.43 \cdot 10^{-2}$ & $2.0$ \\
				128 & $1.31 \cdot 10^{-4}$ & $1.6$ & $2.07 \cdot 10^{-4} $ & $1.8$ & $8.47 \cdot 10^{-3}$ & $2.0$ \\
				\hline
				\multicolumn{7}{c}{$M \sim 10^{-3}$,\hspace{0.2cm} $p_0 = 10^6$}
				\\ \hline 
				16  & $4.89 \cdot 10^{-3}$ & $   $ & $1.11 \cdot 10^{-2} $ & $   $ & $4.53 \cdot 10^{+1}$ & $   $ \\
				32  & $1.38 \cdot 10^{-3}$ & $1.8$ & $2.82 \cdot 10^{-3} $ & $2.0$ & $1.29 \cdot 10^{+1}$ & $1.8$ \\
				64  & $3.96 \cdot 10^{-4}$ & $1.8$ & $7.35 \cdot 10^{-4} $ & $1.9$ & $3.31 \cdot 10^{0}$ & $2.0$ \\
				128 & $1.30 \cdot 10^{-4}$ & $1.6$ & $2.06 \cdot 10^{-4} $ & $1.8$ & $8.17 \cdot 10^{-1}$ & $2.0$ \\
				\hline 				
				\multicolumn{7}{c}{$M \sim 10^{-4}$,\hspace{0.2cm} $p_0 = 10^8$}
				\\ \hline 
				16  & $4.89 \cdot 10^{-3}$ & $   $ & $1.11 \cdot 10^{-2} $ & $   $ & $4.53 \cdot 10^{+3}$ & $   $ \\
				32  & $1.38 \cdot 10^{-3}$ & $1.8$ & $2.82 \cdot 10^{-3} $ & $2.0$ & $1.29 \cdot 10^{+3}$ & $1.8$ \\
				64  & $3.96 \cdot 10^{-4}$ & $1.8$ & $7.35 \cdot 10^{-4} $ & $1.9$ & $3.31 \cdot 10^{+2}$ & $2.0$ \\
				128 & $1.30 \cdot 10^{-4}$ & $1.6$ & $2.06 \cdot 10^{-4} $ & $1.8$ & $8.17 \cdot 10^{+1}$ & $2.0$ \\
				\hline 
				\multicolumn{7}{c}{$M \sim 10^{-5}$,\hspace{0.2cm} $p_0 = 10^{10}$}
				\\ \hline 
				16  & $4.89 \cdot 10^{-3}$ & $   $ & $1.11 \cdot 10^{-2} $ & $   $ & $4.53 \cdot 10^{+5}$ & $   $ \\
				32  & $1.38 \cdot 10^{-3}$ & $1.8$ & $2.82 \cdot 10^{-3} $ & $2.0$ & $1.29 \cdot 10^{+5}$ & $1.8$ \\
				64  & $3.96 \cdot 10^{-4}$ & $1.8$ & $7.35 \cdot 10^{-4} $ & $1.9$ & $3.31 \cdot 10^{+4}$ & $2.0$ \\
				128 & $1.30 \cdot 10^{-4}$ & $1.6$ & $2.06 \cdot 10^{-4} $ & $1.8$ & $8.17 \cdot 10^{{+}3}$ & $2.0$ \\
				\hline 
				\multicolumn{7}{c}{$M \sim 10^{-6}$,\hspace{0.2cm} $p_0 = 10^{12}$}				
				\\ \hline 
				16  & $4.89 \cdot 10^{-3}$ & $   $ & $1.11 \cdot 10^{-2} $ & $   $ & $4.53 \cdot 10^{+7}$ & $   $ \\
				32  & $1.38 \cdot 10^{-3}$ & $1.8$ & $2.83 \cdot 10^{-3} $ & $2.0$ & $1.28 \cdot 10^{+7}$ & $1.8$ \\
				64  & $3.99 \cdot 10^{-4}$ & $1.8$ & $7.43 \cdot 10^{-4} $ & $1.9$ & $3.26 \cdot 10^{+6}$ & $2.0$ \\
				128 & $1.53 \cdot 10^{-4}$ & $1.4$ & $2.46 \cdot 10^{-4} $ & $1.6$ & $7.50 \cdot 10^{{+}5}$ & $2.1$ \\
				\hline 
		\end{tabular}}
	\end{center}
\end{table}

{One of the main advantages of the new semi-implicit hybrid ALE FV/FE method proposed in this paper with respect to classical explicit ALE schemes, see e.g.  \cite{Despres2009,Maire2009,chengshu1,LoubereSedov3D,Lagrange2D,Lagrange3D,LagrangeDG,GBCKSD19,GCD18,Gaburro2021}, is its smaller computational cost for \textit{low Mach number flows} and its ability to solve the \textit{incompressible} Navier-Stokes equations, which explicit ALE methods cannot solve at all. Since the sound velocity does not appear in the eigenvalues of the explicit subsystem, the CFL time step restriction is less severe and thus in the incompressible limit the semi-implicit ALE schemes allow much bigger time steps compared to the explicit ones. To numerically show this property, we have run a last test case with reference pressure  $p_0=10^4$, thus $M\sim 10^{-2}$, using the semi-implicit hybrid ALE FV/FE method for incompressible and weakly compressible flows and a fully explicit second order density-based Godunov-type finite volume ALE scheme solving the compressible Navier-Stokes system. The results reported in Table \ref{tab:IV_times_errors_explicit} for a fixed Courant number of CFL=0.5 show that the CPU time needed per element and time step, $t_e=t_w/(n_e n_t)$, with $t_w$ the total wall clock time, $n_e$ the number of elements and $n_t$ the number of time steps, is significantly smaller for the explicit schemes.  However, the semi-implicit algorithm requires much less time steps due to the less restrictive CFL condition and thus results in a much lower overall CPU time for the entire simulation compared to the explicit ALE method. 
Last but not least, we can note that for all schemes analyzed before the overall overhead introduced by the moving mesh technique presented in this paper is rather \textit{low} compared to the same schemes on a fixed mesh, see the time per element update indicator $t_e$ in Table \ref{tab:IV_times_errors_explicit}. }
\begin{table}
	\begin{center}
		\renewcommand{\arraystretch}{1.2}
		{\begin{tabular}{cccccc}
			\hline
			Method & $n_t$    & CPU time (s) & $t_{e}$ ($\mu$s) &  $L^2_{M_3}(p)$ & $L^2_{M_3}(\mathbf{w}_{\mathbf{u}})$ \\\hline
			Semi-implicit Eulerian, incompressible &
			$12 $ & $4.39$ & $29.46 $ &  $3.93\cdot 10^{-3} $ & $1.29\cdot 10^{-3} $\\
			Semi-implicit ALE, incompressible &
			$12 $ & $ 4.85$ & $32.55 $ &  $3.33\cdot 10^{-3} $ & $1.33\cdot 10^{-3} $\\
			Semi-implicit Eulerian, incompressible - PC &
			$12 $ & $3.31$ & $22.20 $ &  $5.03\cdot 10^{-3} $ & $1.15\cdot 10^{-3} $\\
			Semi-implicit ALE, incompressible - PC &
			$12 $ & $3.38$ & $22.69 $ &  $4.54\cdot 10^{-3} $ & $1.17\cdot 10^{-3} $\\
			Semi-implicit Eulerian, weakly compressible &			
			$6 $ & $2.33  $ & $31.37 $ &  $1.85 \cdot 10^{-2} $ & $ 1.62 \cdot 10^{-3} $\\
			Semi-implicit ALE, weakly compressible &			
			$6 $ & $3.02 $ & $40.54 $ &  $3.28\cdot 10^{-2} $ & $1.50\cdot 10^{-3} $\\
			Explicit Eulerian, fully compressible &
			$852 $ & $39.21$ & $3.70 $ & $1.44\cdot 10^{-2} $ &  $8.39\cdot 10^{-3} $\\
			Explicit ALE, fully compressible &
			$852 $ & $39.80$ & $3.76 $ & $8.36\cdot 10^{-2} $ &  $1.42\cdot 10^{-3} $ 
			\\\hline
		\end{tabular}
		 \caption{Number of time steps $n_t$, CPU time $t_w$, CPU time per element and timestep, $t_{e}$, and errors obtained for the vortex test case on mesh $M_{3}$ using the semi-implicit hybrid ALE FV/FE schemes and an explicit ALE finite volume scheme for the compressible Navier-Stokes equations as well as the corresponding purely Eulerian schemes.	}\label{tab:IV_times_errors_explicit}}
	\end{center}
\end{table}

\subsection{Inviscid flow around a moving cylinder}

The second test analyses the flow of an inviscid incompressible fluid around a moving cylinder. As computational domain we consider a circle of outer radius $R=10$ pierced by a concentric circular cylinder of radius $r_c = 0.5$. At the outer boundary we impose the pressure $p=p_0=0$ while over the cylinder we employ inviscid wall boundary conditions. The domain is discretized using 29418 triangular primal elements with a characteristic mesh spacing of $h=0.02$ at the cylinder wall and a mesh spacing of $h=0.2$ at the outer boundary of the domain. The viscosity is $\mu=0$ and the gravity is $\mathbf{g}=\mathbf{0}$. The fluid is initially at rest with zero pressure and the entire mesh performs a rigid body motion to the left with velocity $\mV=(-u_{\infty},0)$, $u_{\infty}=1$. Simulations are carried out until a final time of $t=1$. We expect a potential flow to develop, hence we can compare our numerical results with the exact solution of a potential flow around the cylinder, given by 
\begin{gather*}
	v_r      =  u_{\infty} (1-r_c^2/r^2) \cos(\phi), \qquad 
	v_{\phi} = -u_{\infty} (1+r_c^2/r^2) \sin(\phi), \\ 	
	u_1      =  v_r \cos(\phi) - v_{\phi} \sin(\phi) - u_{\infty}, \qquad 
	u_2      =  v_r \sin(\phi) + v_{\phi} \cos(\phi),  
\end{gather*} 
where the pressure is given by the Bernoulli equation $p = p_0 + \halb u_{\infty}^2 - \halb (v_r^2+v_{\phi}^2)$, the radius is $r = \sqrt{(x+u_{\infty} t)^2+y^2}$ and the angle satisfies $\tan \phi = y/(x+u_{\infty} t)$ at time $t=1$. A comparison between the numerical result obtained with our new ALE hybrid FV/FE scheme and the potential flow reference solution is presented in Figure~\ref{fig.potflow} for the radius $r=0.501$. We can observe an excellent agreement for both, velocity and pressure.

\begin{figure}
	\begin{center}
		\includegraphics[width=0.49\linewidth]{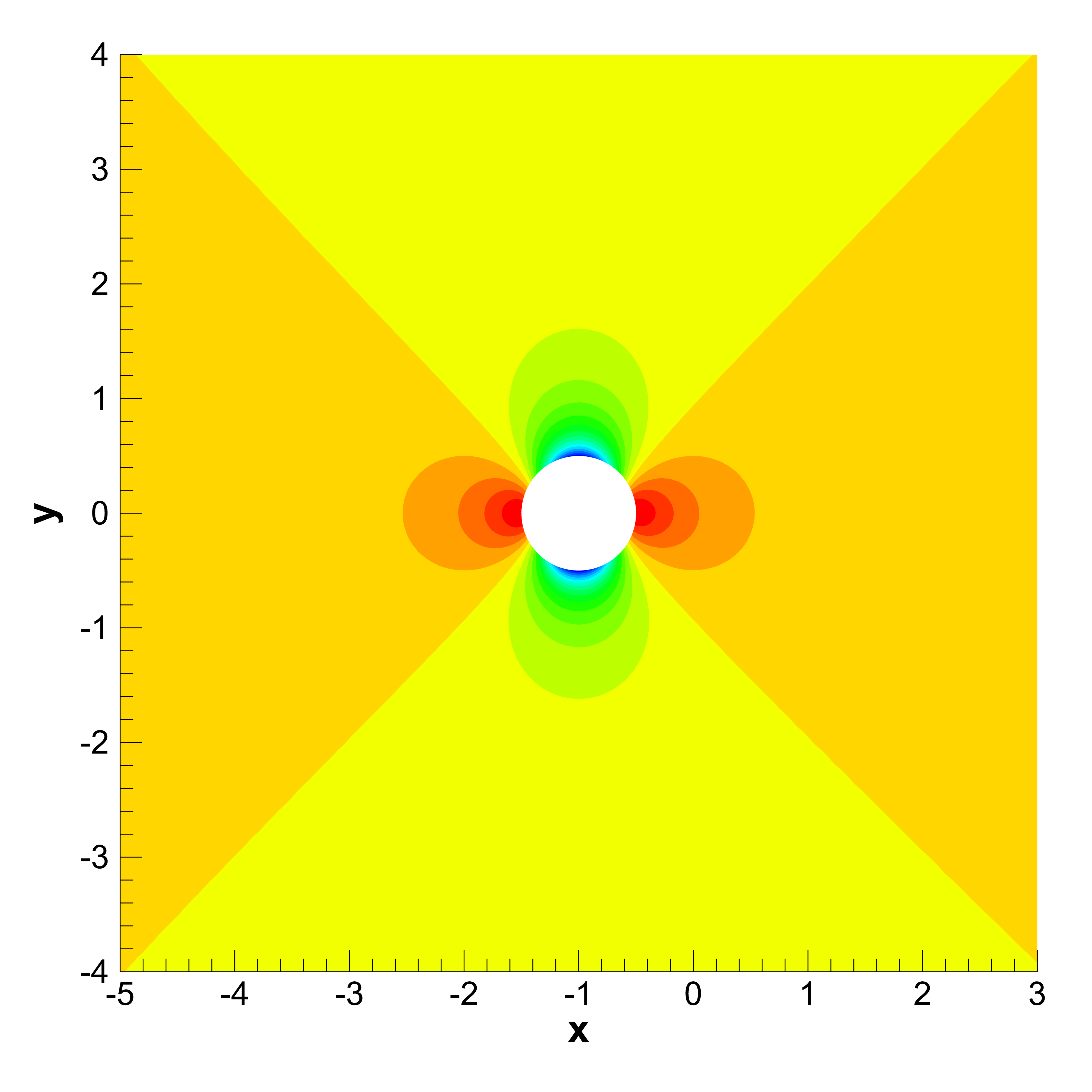}\hfill
		\includegraphics[width=0.49\linewidth]{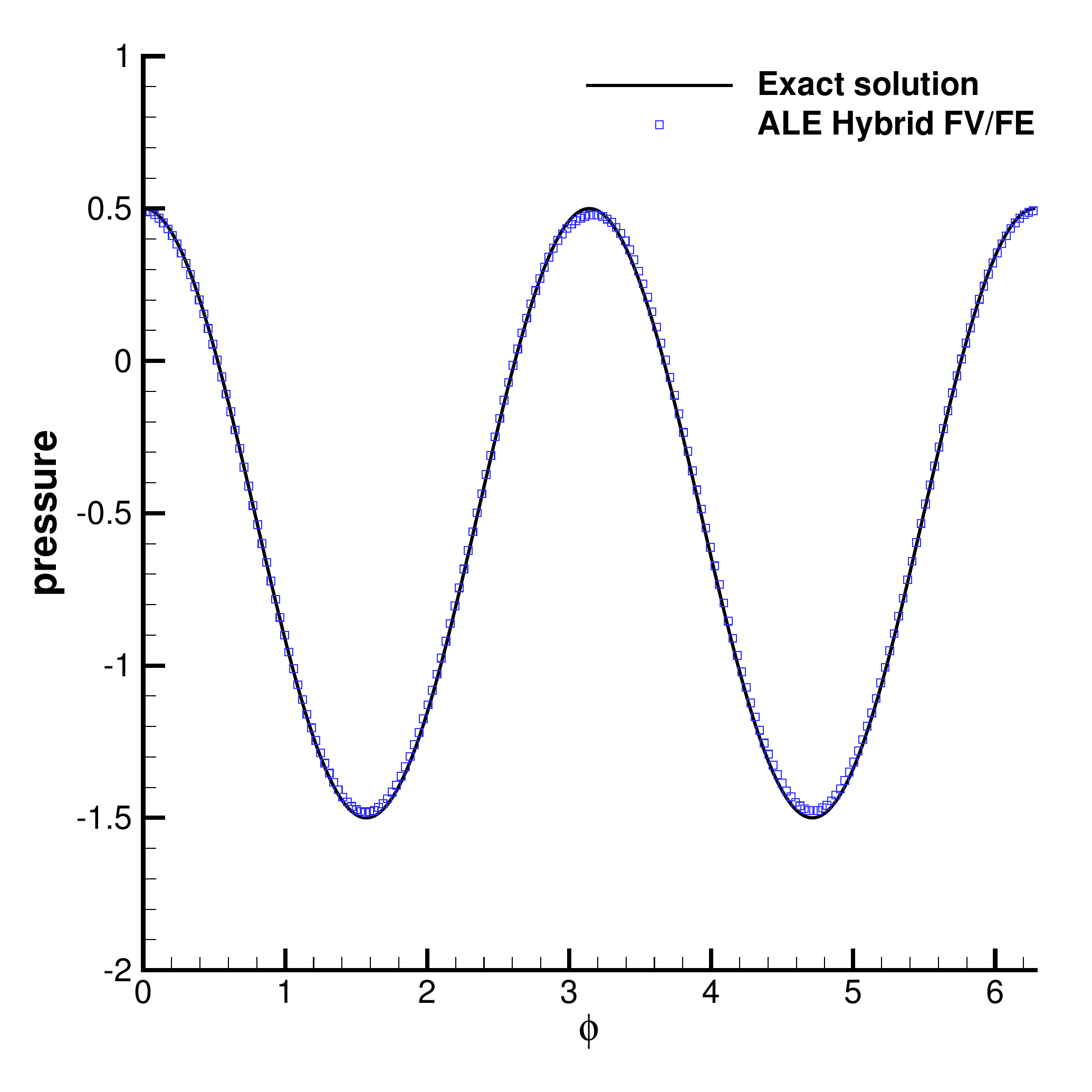}
		
		\vspace{0.2cm}		
		\includegraphics[width=0.49\linewidth]{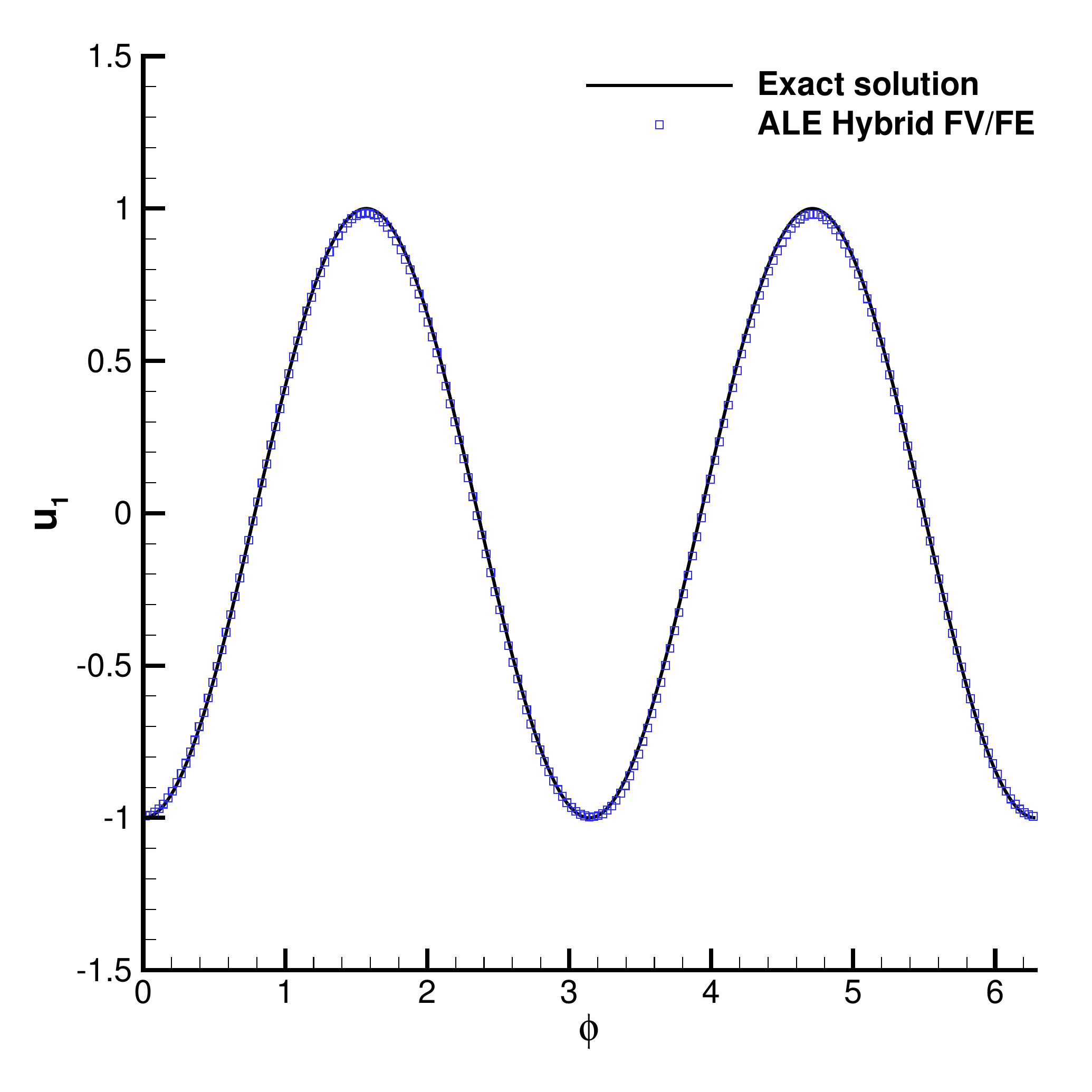}\hfill 
		\includegraphics[width=0.49\linewidth]{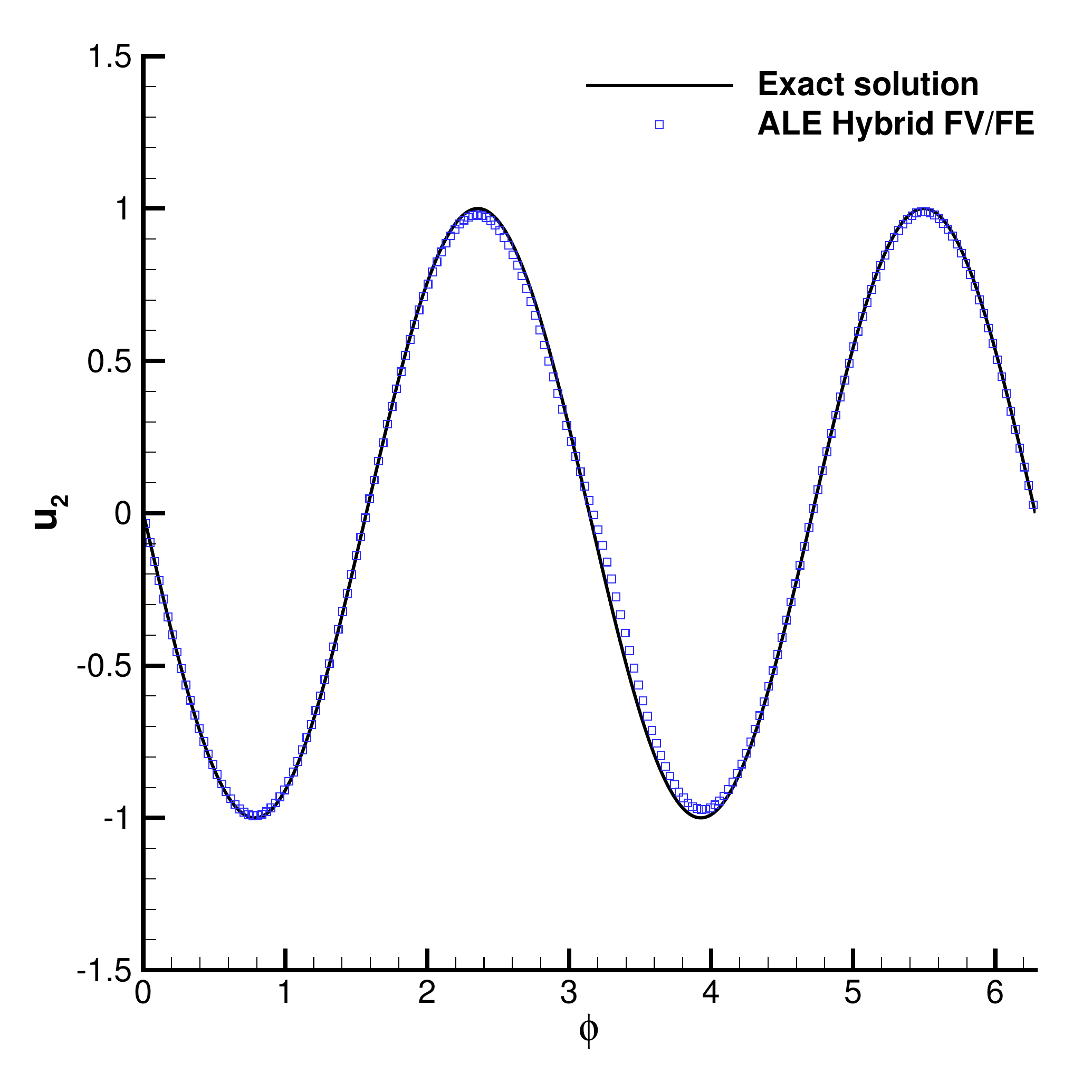}
	\end{center}
	\caption{Solution obtained for the inviscid flow around a cylinder at time $t=1$. The left top figure shows the  contour plot of the pressure field around the cylinder. Meanwhile, the three 1D plots report the values of $u_1$, $u_2$ and $p$ along the circle of radius $r=0.501$ computed with the ALE Hybrid FV/FE scheme (blue squares) and the exact solution (black solid line).}\label{fig.potflow}
\end{figure}  

\subsection{{Incompressible} viscous flow around an oscillating cylinder}

In this test, we analyse the problem of the flow around an oscillating cylinder which has been thoroughly studied in~\cite{BlackburnHenderson1999, Guilmineau2002, RamirezNogueira2017, Lu1996}. 
We first focus on the simulation of the flow around the static cylinder. In this way, we can study the natural frequency of the vortex shedding. The domain considered is the region limited by an outer cylinder, of radius $r=10$, and an inner cylinder of radius $r=0.5$, both centered at $\x=\mathbf{0}$. It is discretized considering three different meshes, $M_1$, $M_2$ and $M_3$ with 13812, 29418 and 45804 triangular primal elements, respectively. 
The mesh velocity is obtained from the solution of the Laplace equation, with zero velocity on the outer boundaries and with a given velocity imposed on the cylinder. 
In Figure~\ref{fig:cyl_mesh}, a zoom of the mesh around the cylinder is shown for all the three grids. 
\begin{figure}[h!]
	\begin{center}
	\includegraphics[width=0.32\linewidth]{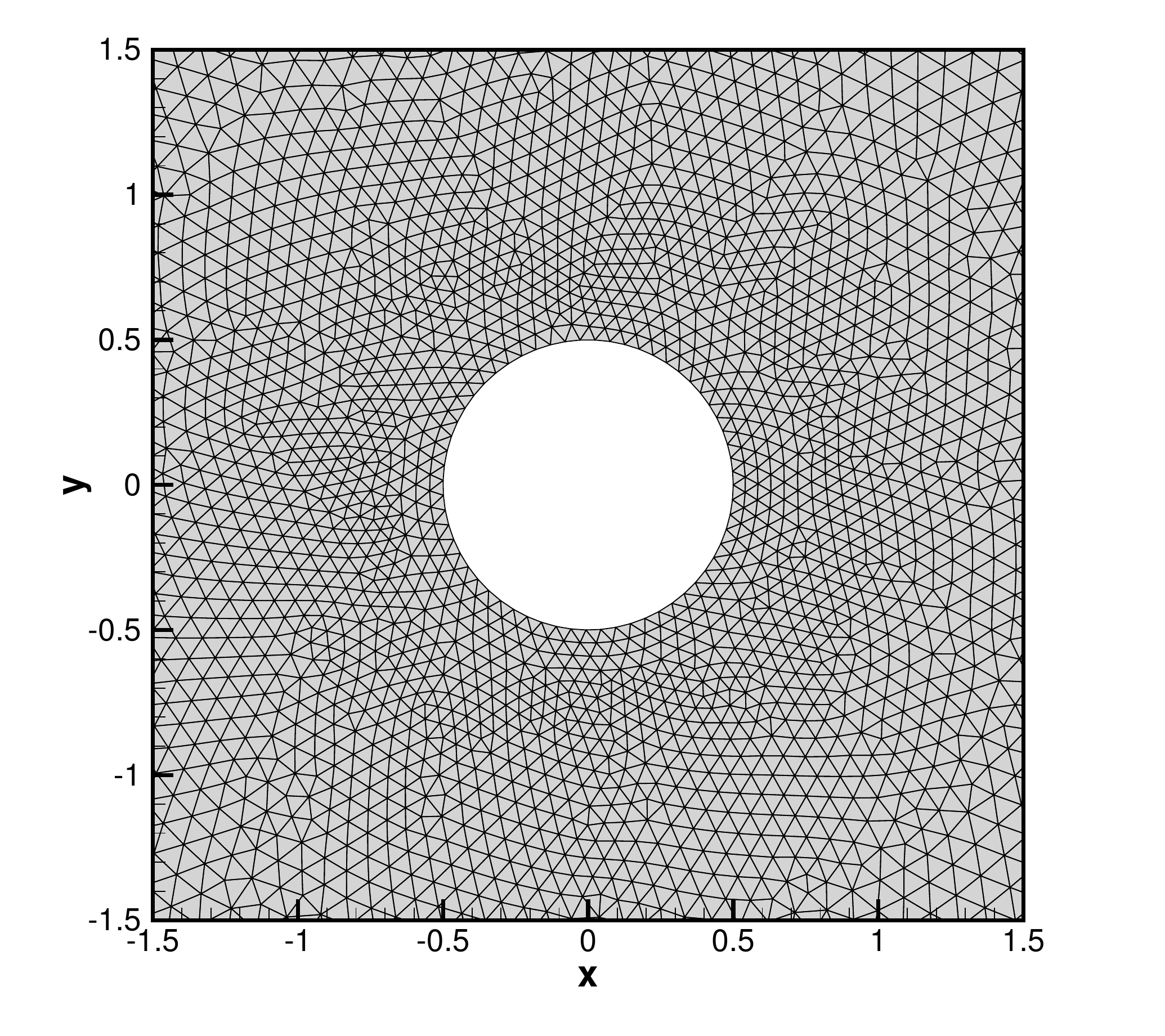}
	\includegraphics[width=0.32\linewidth]{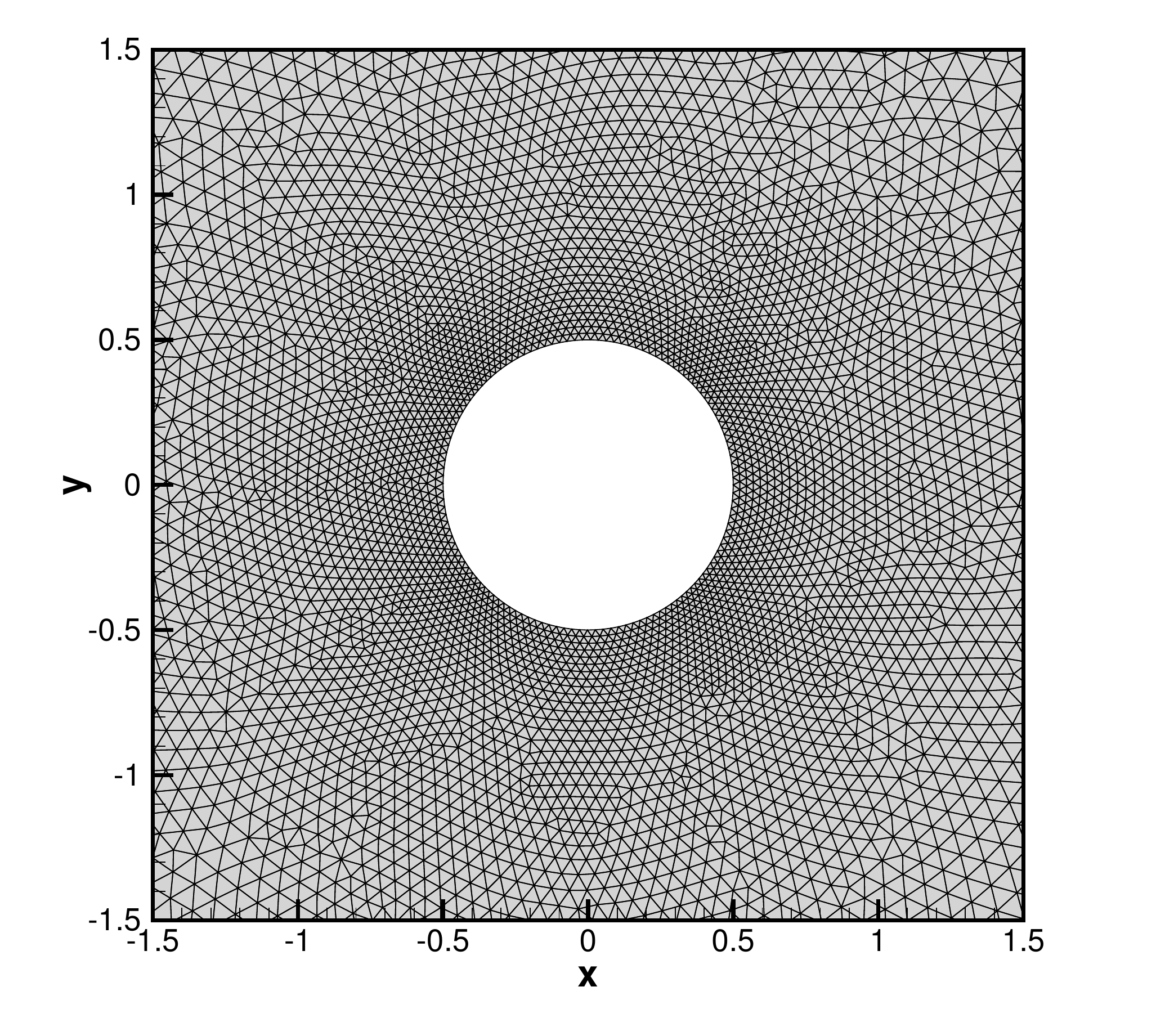}
	\includegraphics[width=0.32\linewidth]{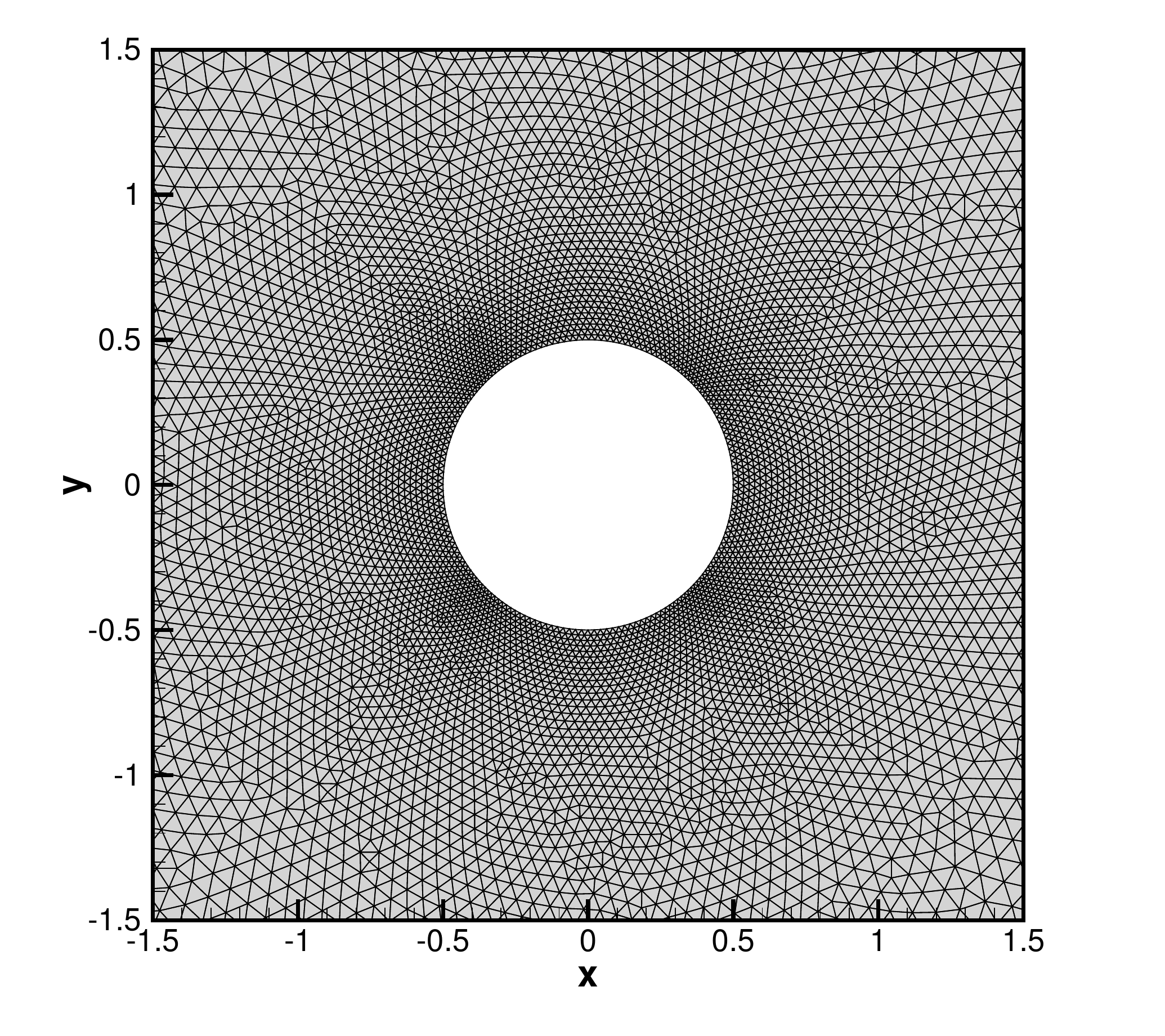}
	\end{center}
	\caption{Zoom around the cylinder of the three primal meshes, $M_1$ (left), $M_2$ (center), $M_3$ (right), used for discretizing the domain for the simulation of a viscous flow passing an oscillating cylinder.}
	\label{fig:cyl_mesh}
\end{figure}
The density is set to $\rho=1$, while the initial velocity and pressure fields are given by the corresponding potential flow solution around the cylinder, {with the reference velocity chosen as the characteristic velocity needed to obtain a desired Re number. In particular, to obtain a Reynolds number of Re$=185$, we have set the laminar viscosity of the fluid to $\mu= 0.0054$ and $u_{\infty}=1$}. In the outer boundary a pressure outlet condition is imposed and for the inner cylinder a viscous wall boundary is considered. As it is well known, after some time, the vortices reach a stable periodic regime, the von K\'arm\'an street appears and we can compute the frequency of the vortex shedding. In the left plot of Figure~\ref{fig:staticcyl}, we can observe the drag and lift coefficients, while in the right plot, the time series of the velocity at $\x=(4,0)$ are portrayed. In this case, the domain has been discretized with the coarsest mesh, $M_1$, obtaining an oscillating period of the static cylinder of $T=5.13$. Thus, the associated Strouhal number is $\St = \frac{f_0 D}{u} = 0.195$, which coincides with the value appearing in the literature (see, e.g.~\cite{Guilmineau2002, Lu1996}). 
\begin{figure}[h!]
	\begin{center}
	\includegraphics[width=0.49\linewidth]{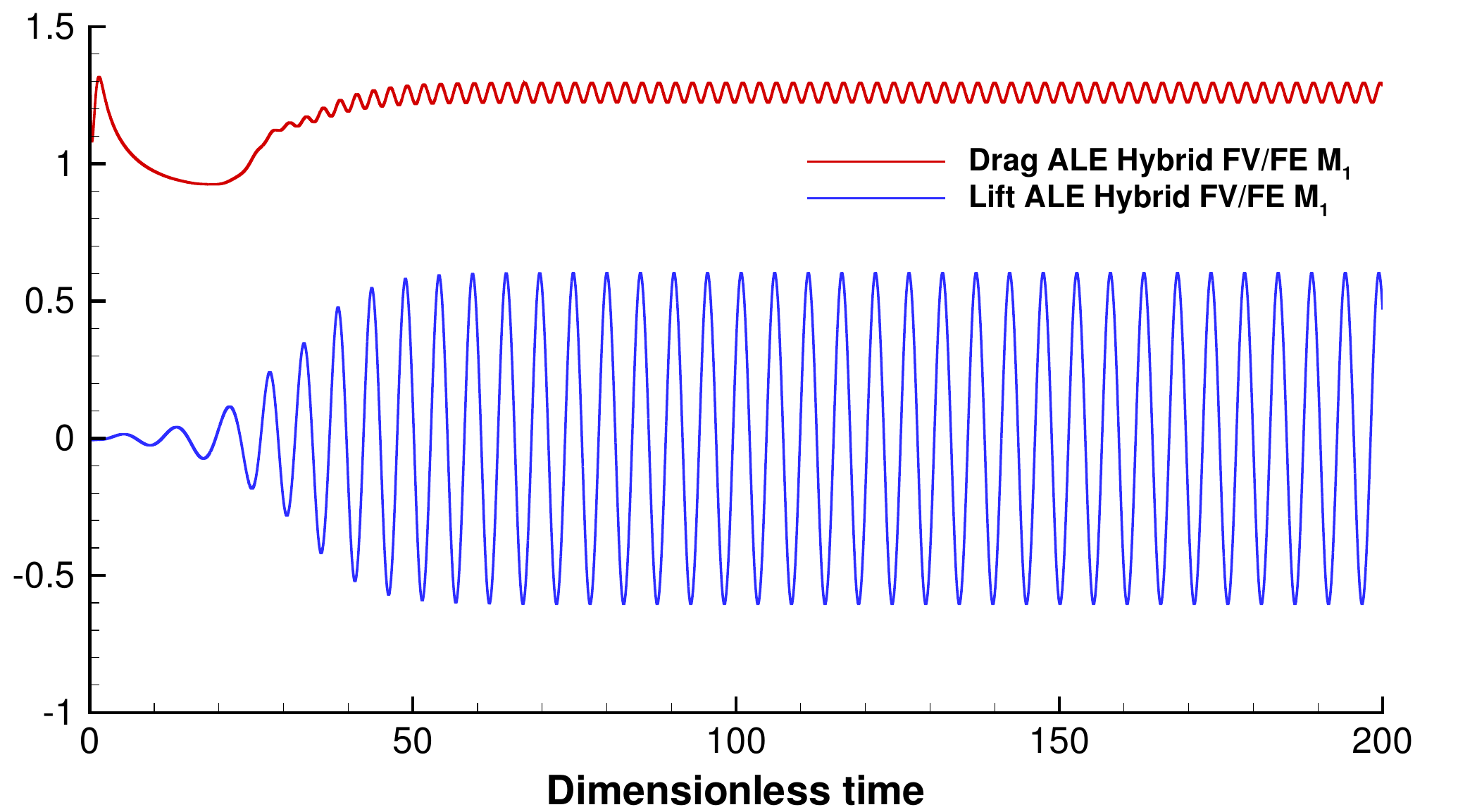}
	\includegraphics[width=0.49\linewidth]{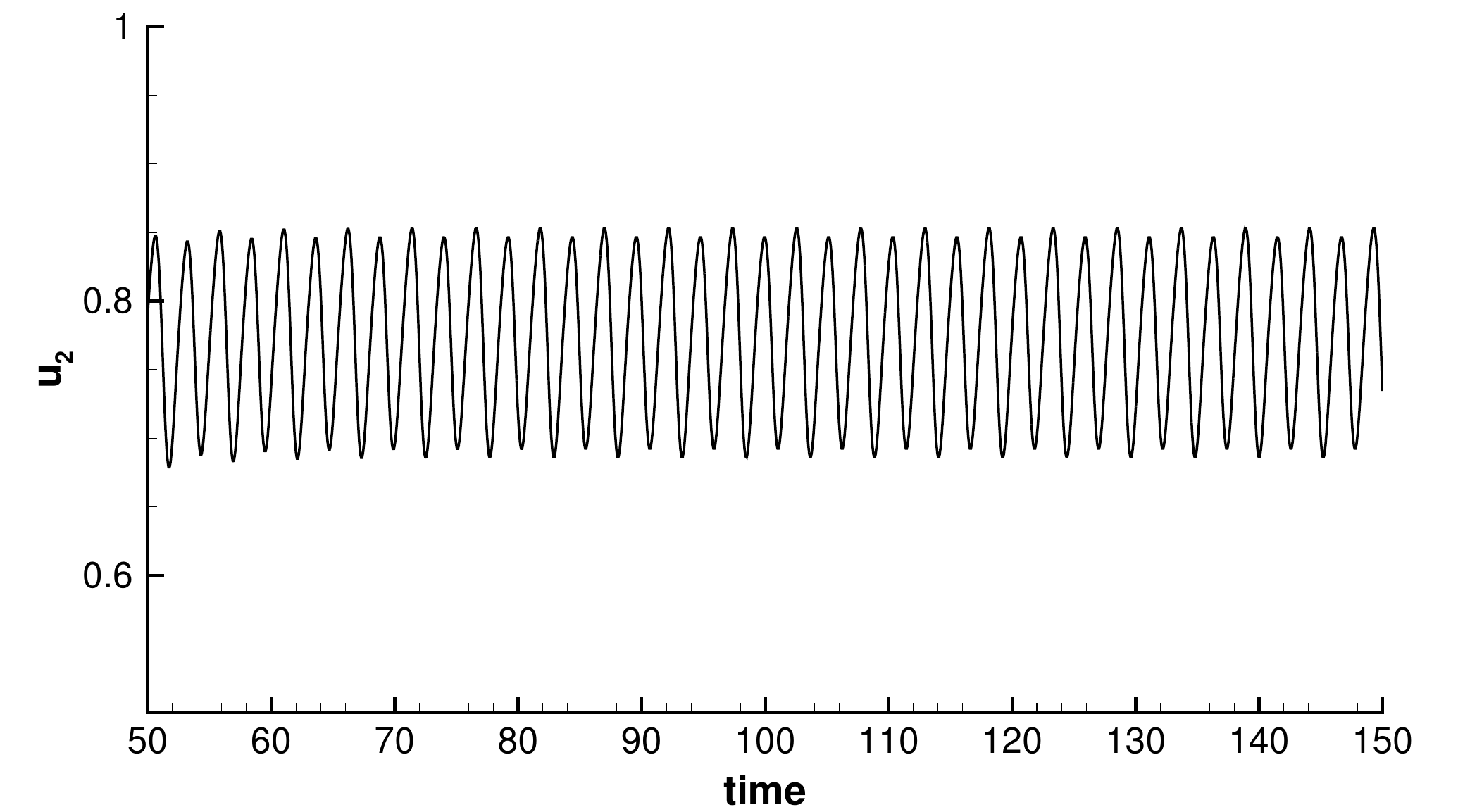}
	\end{center}
	\caption{Left: Drag and lift coefficient of a static cylinder with Re=185. Right: Time series of the velocity $u_2$ at $\x=(4,0)$. The domain has been discretized with the mesh $M_1$.}
	\label{fig:staticcyl}
\end{figure}

We now address two oscillating cylinder tests for Re$=185$ with different frequencies of oscillation. A vertical movement is imposed to the cylinder so that its center is located at $(x_c(t),y_c(t)) =(0, A \sin(2\pi f t))$. 
The oscillation amplitude is set to $A=0.2$ and the two different oscillation frequencies considered are $f=0.8 f_0$ and $f=1.1 f_0$, where $f_0$ is the frequency of the vortex shedding in the static case. We first focus on the smallest oscillation frequency, $f=0.8 f_0$. The left plot of Figure~\ref{fig:osccyl_drag_lift_64} shows the drag and lift coefficients of the oscillating cylinder
obtained running the ALE hybrid scheme on the coarsest mesh $M_1$. A good agreement is observed when compared with the results computed by Guilmineau \textit{et al.}~\cite{Guilmineau2002}. Moreover, the right plot demonstrates the mesh convergence obtained when we perform the simulation with finer grids, $M_2$ and $M_3$.
\begin{figure}[h!]
	\begin{center}
	\includegraphics[width=0.49\linewidth]{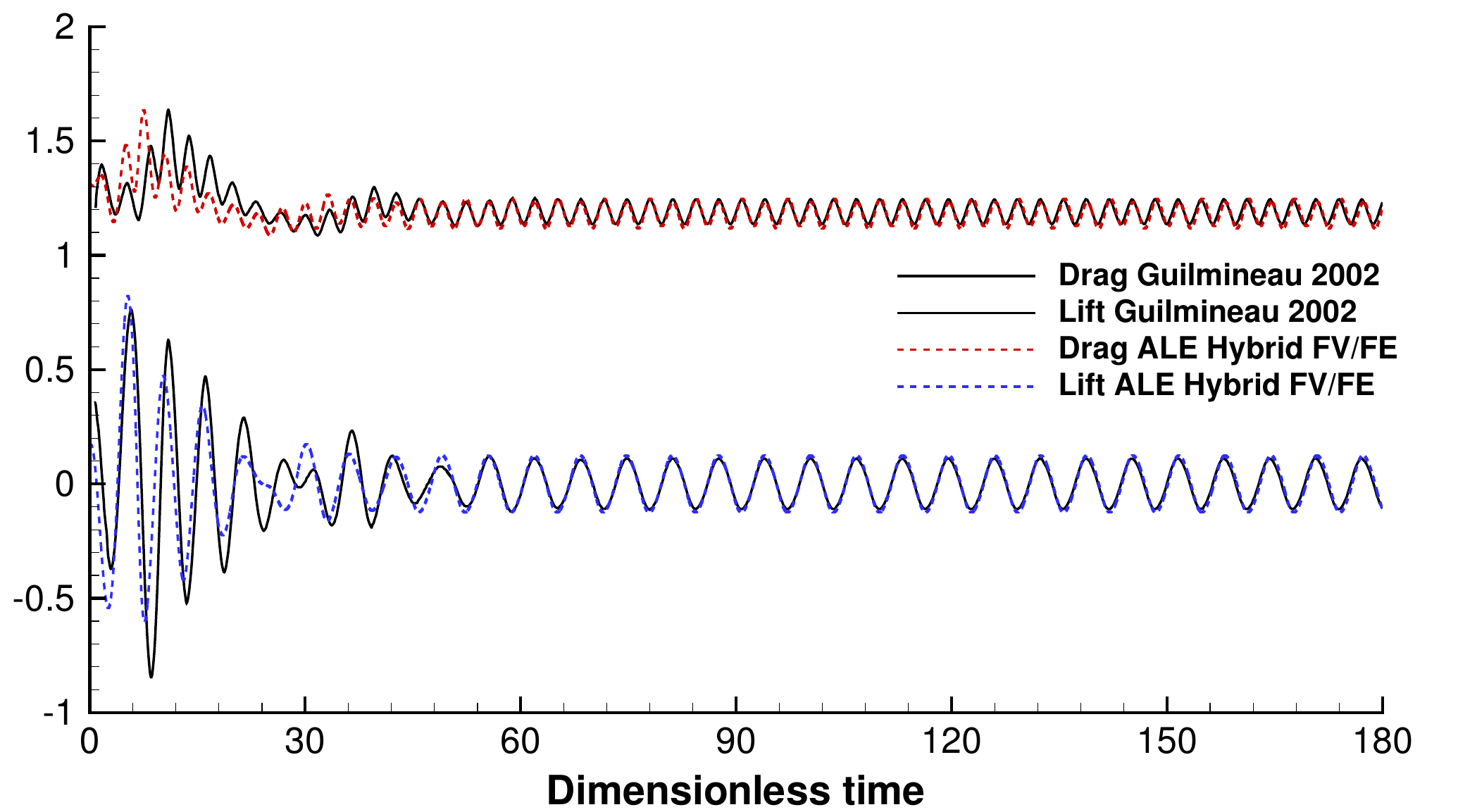}
	\includegraphics[width=0.49\linewidth]{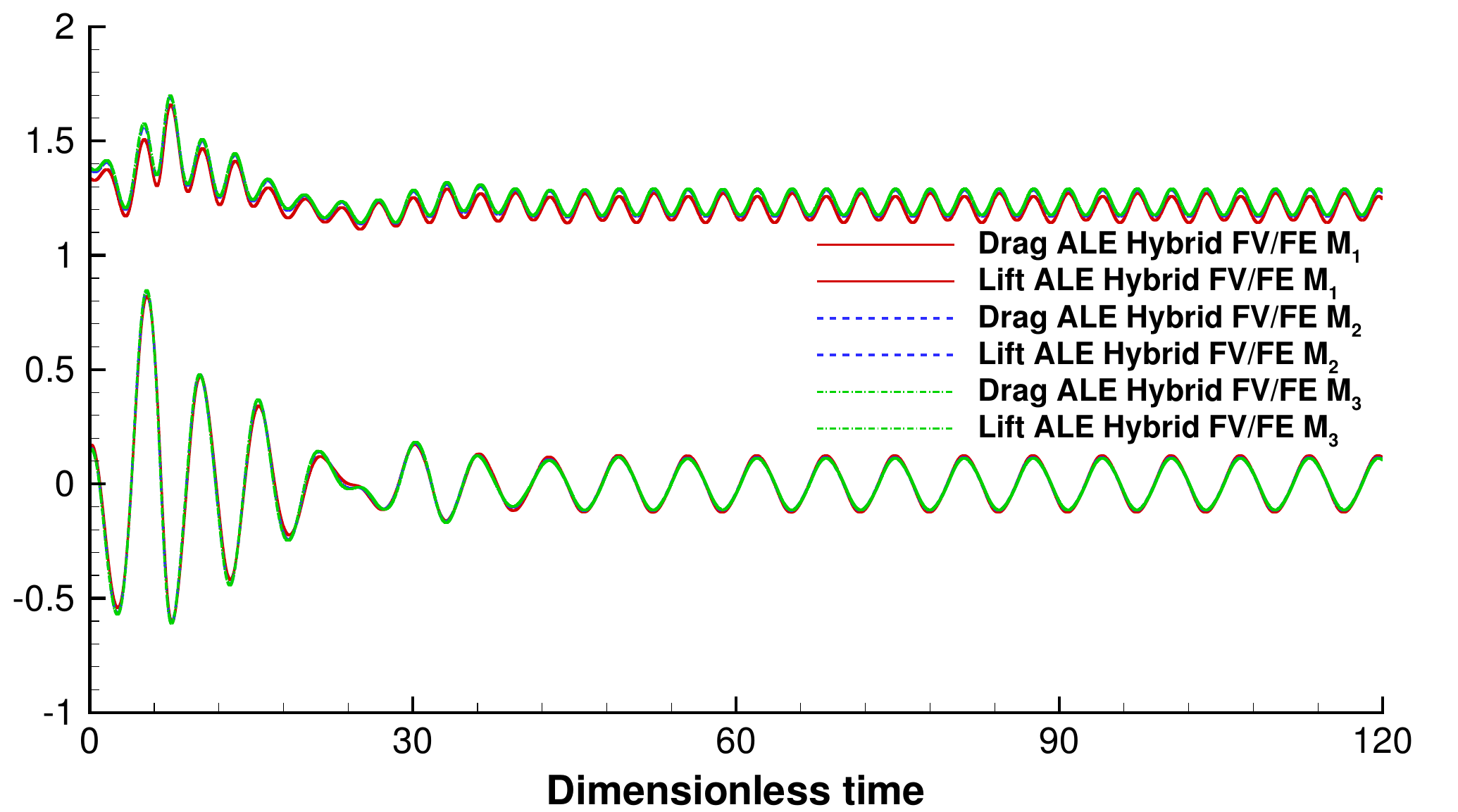}	
	\end{center}
	\caption{Drag and lift coefficients of an oscillating cylinder with Re$=185$, $f=\frac{0.8}{5.13}$, obtained with the ALE hybrid scheme in comparison with the results of Guilmineau \textit{et al.}~\cite{Guilmineau2002} (left). Mesh convergence study for the drag and lift coefficients (right).}
	\label{fig:osccyl_drag_lift_64}
\end{figure}
The periodic change of drag and lift coefficients plotted in terms of the vertical position of the center of the cylinder is shown in Figure~\ref{fig:time_periodic_drag_lift_08}
while the vorticity contours at times $t=8$, $t=9.3$, $t=11.1$, and $t=12.7$ are depicted in Figure~\ref{fig:osccyl_vorticity_64}. The regular symmetric periodic shed of the vortex is observed.
\begin{figure}[h!]
	\begin{center}
	\includegraphics[width=0.425\linewidth]{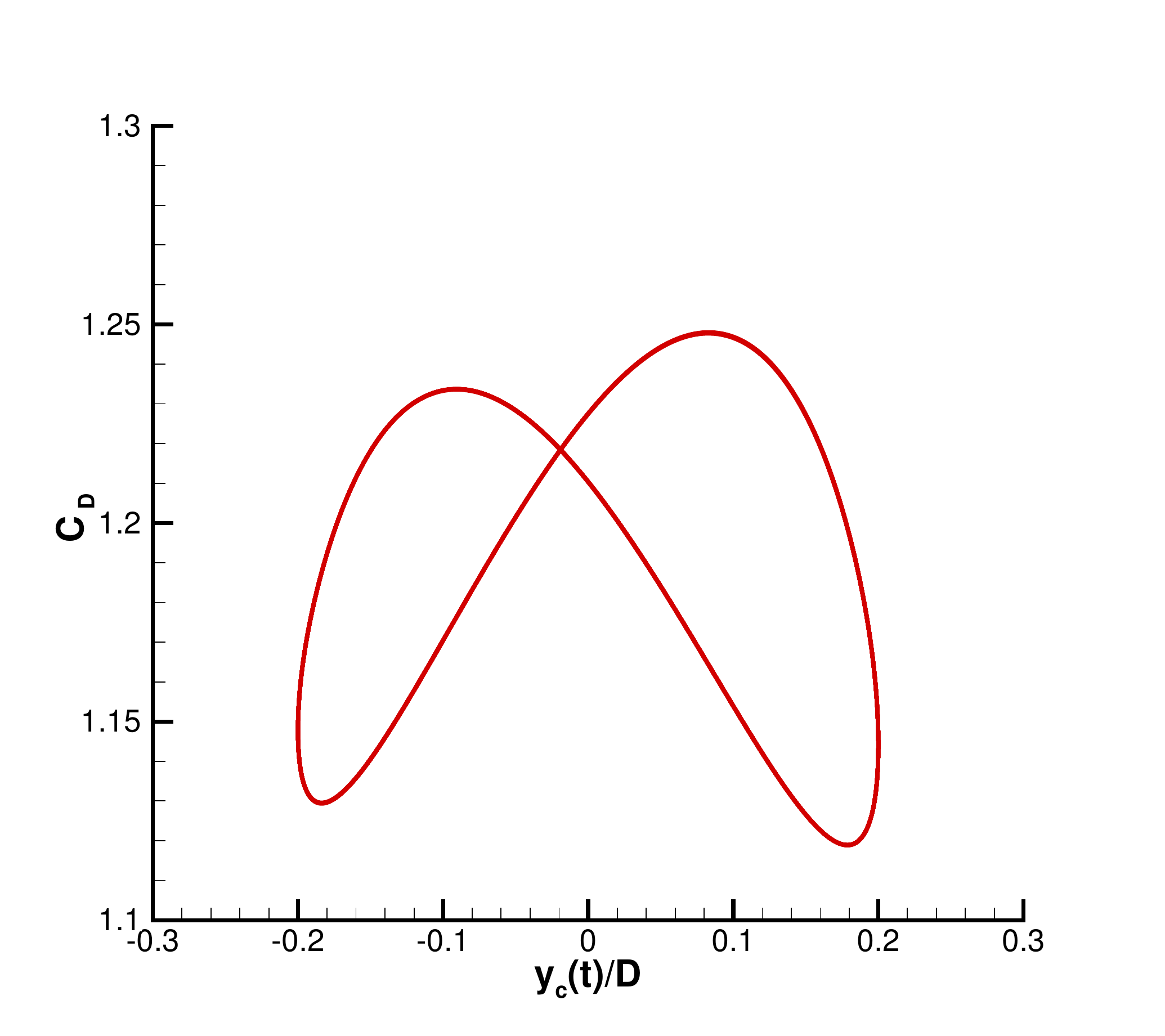}
	\includegraphics[width=0.425\linewidth]{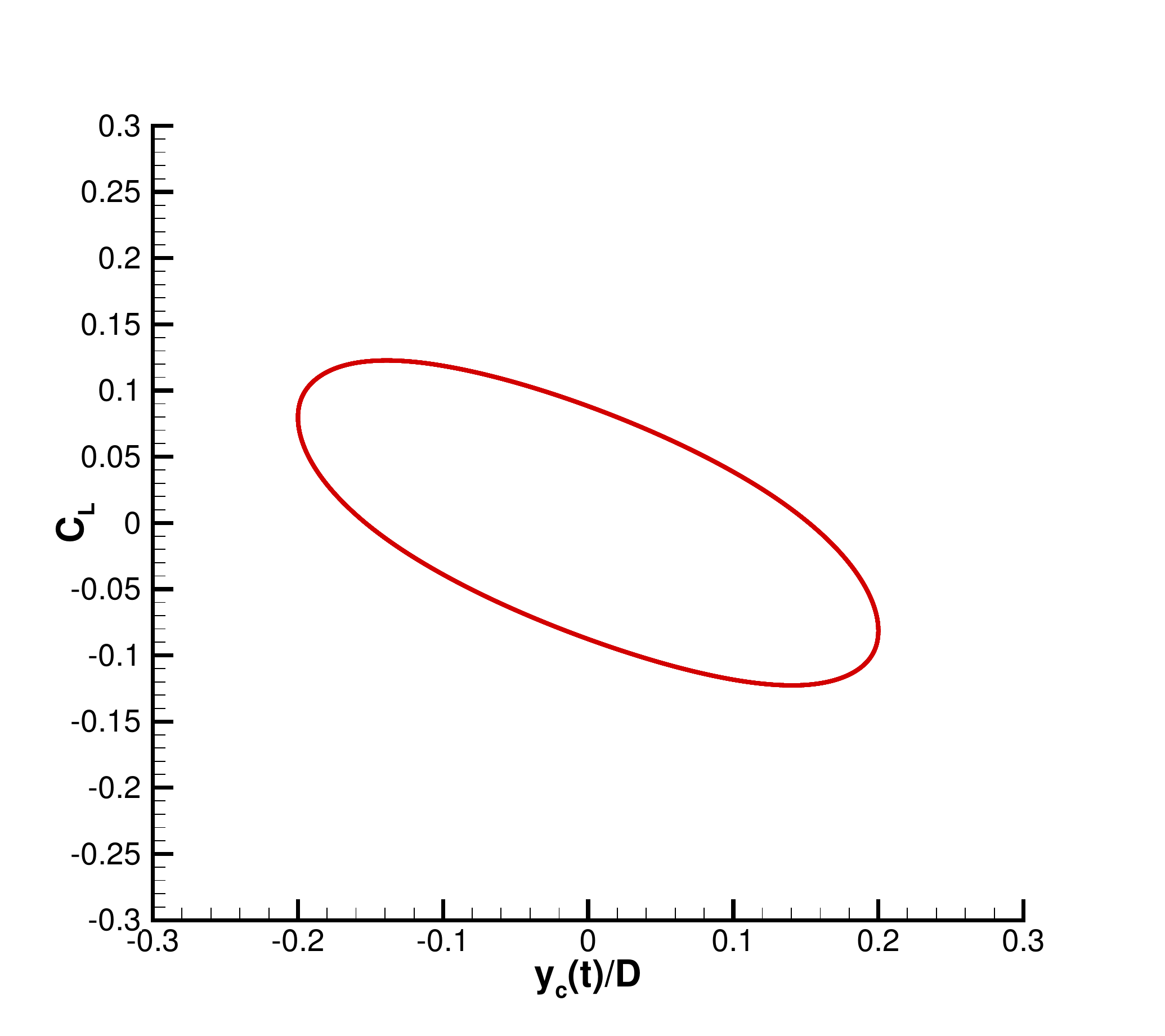}
	\end{center}
	\caption{Time periodic drag (left) and lift (right) coefficients of an oscillating cylinder with oscillation frequency $f=\frac{0.8}{5.13}$.}
	\label{fig:time_periodic_drag_lift_08}
\end{figure}

\begin{figure}[h!]
	\begin{center}
	\includegraphics[trim=5 5 5 5,clip,width=0.48\linewidth]{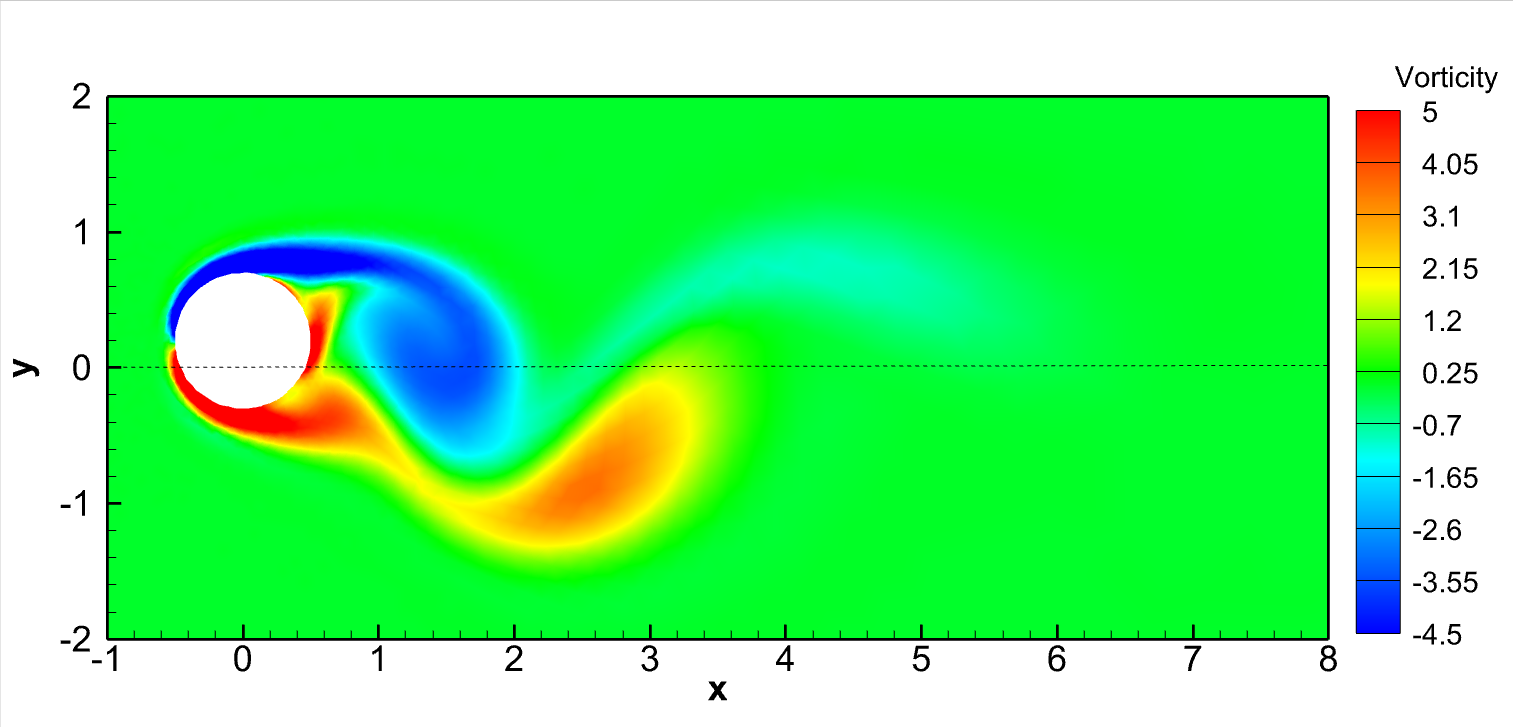}
	\includegraphics[trim=5 5 5 5,clip,width=0.48\linewidth]{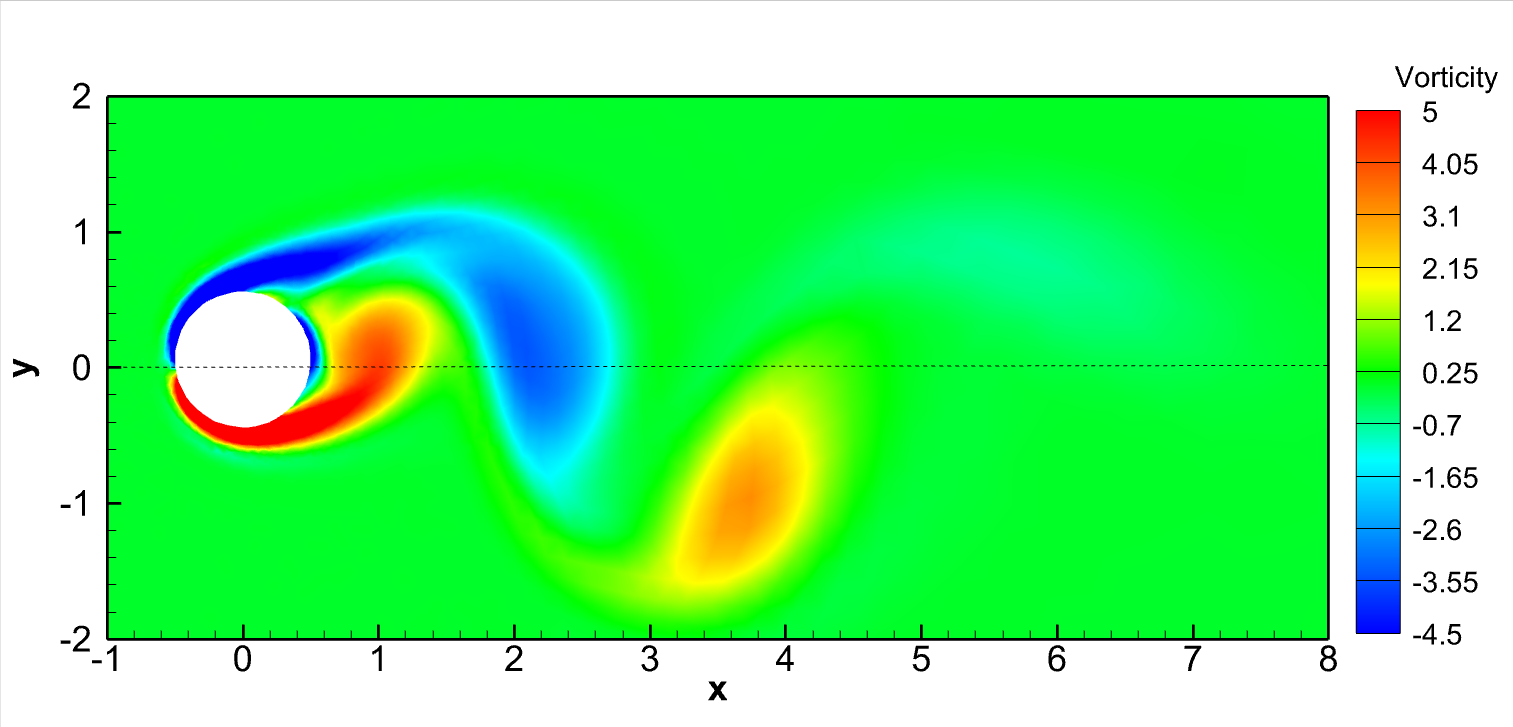}\\
	\includegraphics[trim=5 5 5 5,clip,width=0.48\linewidth]{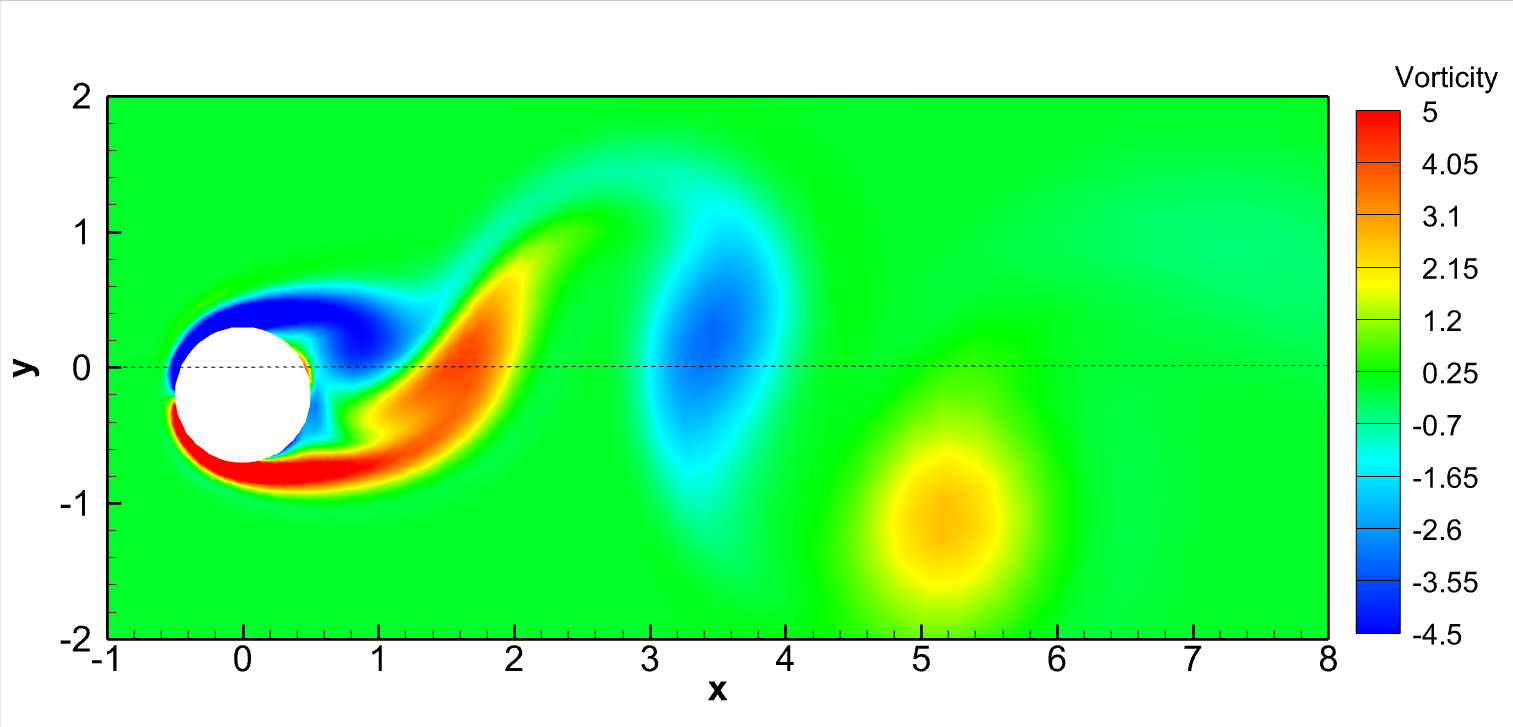}
	\includegraphics[trim=5 5 5 5,clip,width=0.48\linewidth]{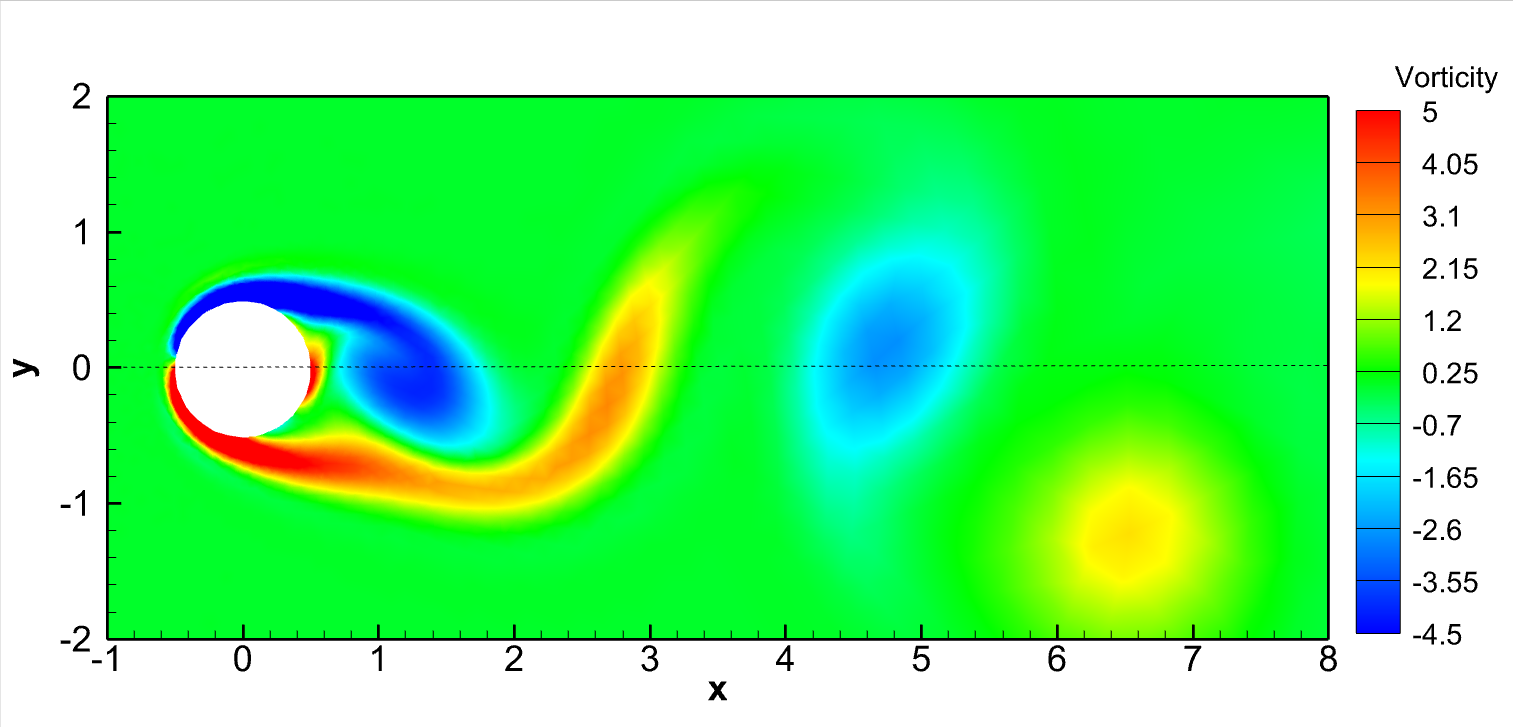}\\[-2cm]
	\end{center}
	\caption{Vorticity field at times $t=8$, $t=9.3$, $t=11.1$, and $t=12.7$, of an oscillating cylinder with $Re=185$, $f=\frac{0.8}{5.13}$.}
	\label{fig:osccyl_vorticity_64}
\end{figure}
Regarding the test with oscillation frequency $f=1.1 f_0$, we expect the lift and drag coefficients to change, from the sinusoidal profile at constant amplitude on time observed previously, to a periodic sinusoidal function with periodical repeated patterns, as it can be observed, e.g. in~\cite{Guilmineau2002,RamirezNogueira2017}. The solution obtained with the proposed ALE hybrid method successfully reproduces the expected behaviour, as it can be observed in Figure~\ref{fig:osccyl_drag_lift_128_f11}, and a pretty good agreement is achieved with respect to numerical results appearing in the literature. 
In this test case a mesh convergence study is performed and reported in the right plot of Figure~\ref{fig:osccyl_drag_lift_128_f11}.

\begin{figure}[h!]
	\begin{center}
	\includegraphics[width=0.49\linewidth]{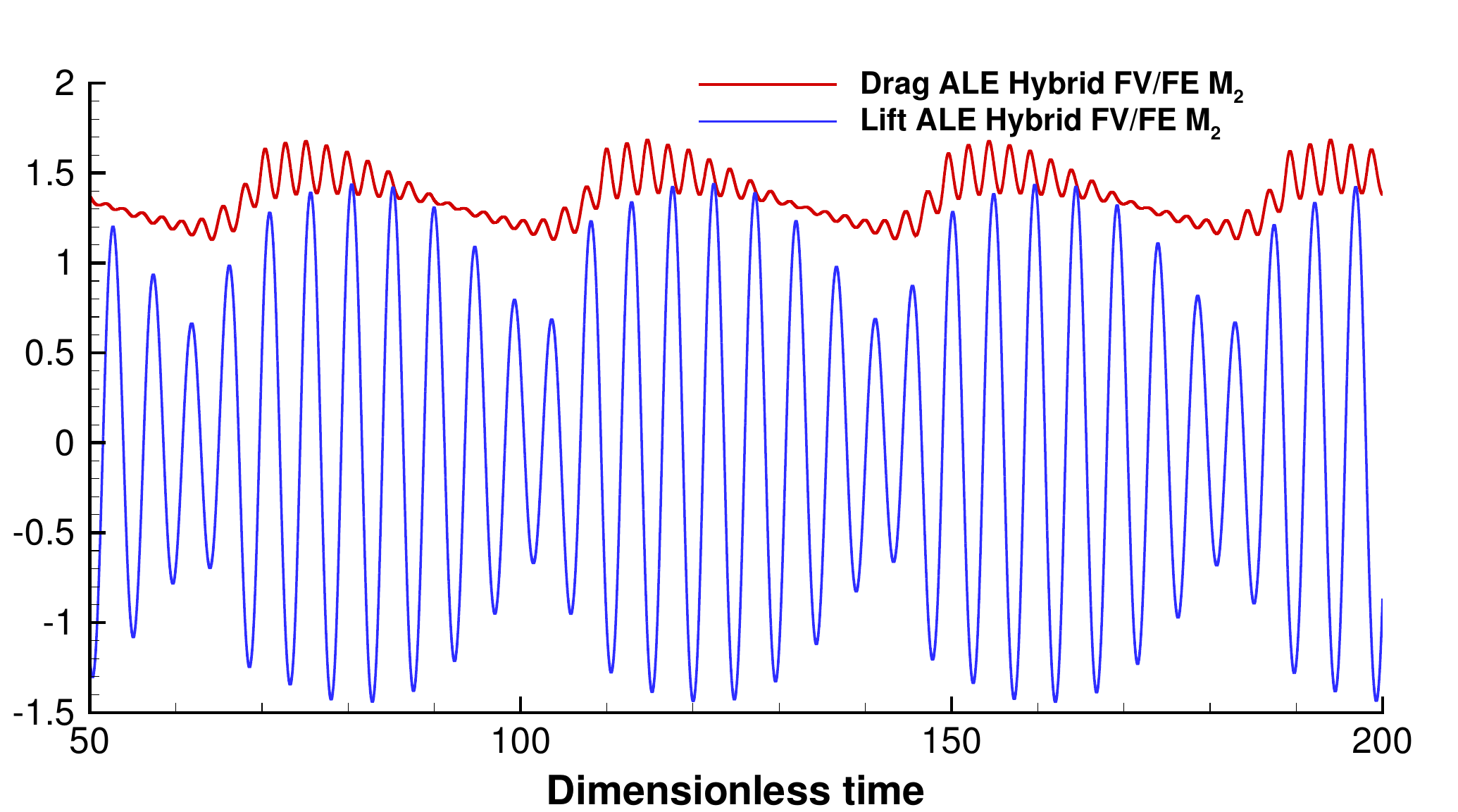}
	\includegraphics[width=0.49\linewidth]{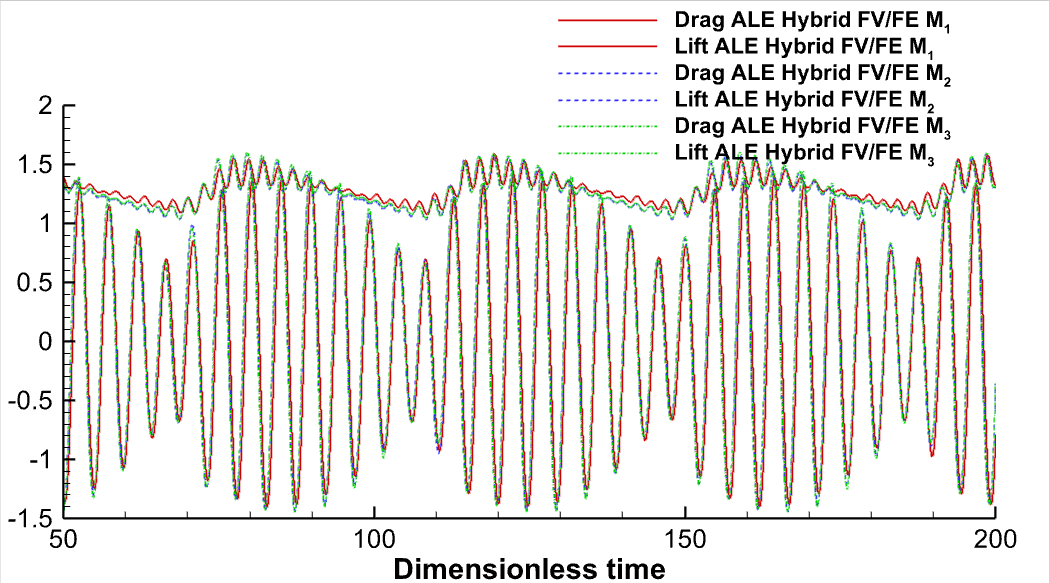}
	\end{center}
	\caption{Drag and lift coefficients of an oscillating cylinder with Re$=185$, $f=\frac{1.1}{5.13}$, obtained with the ALE hybrid scheme (left). Mesh convergence study for the drag and lift coefficients (right).}
	\label{fig:osccyl_drag_lift_128_f11}
\end{figure}

\subsection{Viscous flow around a moving ellipse}

Following~\cite{Wang2000,Wang2005}, we study a flapping flight, characteristic for dragonflies, by simulating {an incompressible} viscous flow around an elliptic wing section. In particular, the wing cross section in the chord direction is represented by an ellipse with chord length $c=0.025$ and thickness $0.003125$, i.e. a thickness ratio of 0.125, and the flapping velocity is given by
\begin{equation*}
u_f(t)=2\pi f A \sin(2\pi f t),
\end{equation*}
where $f$ and $A$ are the flapping frequency and amplitude, respectively. For this test case, and following~\cite{Wang2000}, we consider two Strouhal numbers, one involving the flapping amplitude, $\St_{A}$, and other related to the chord length, $\St_c$. In the context of the insect flight, the first number represents the advance ratio while the second one measures the flapping frequency and are calculated as 
\begin{equation*}
\St_{A}=f\frac{A}{u_0},\qquad \qquad
\St_c=f\frac{c}{u_0}.
\end{equation*}
For the numerical simulation, the domain is defined by an outer cylinder, of radius $r=0.0625$, and an embedded ellipse, with chord length $c=0.025$ and thickness $0.003125$, both centered at $\x=\mathbf{0}$. The domain is discretized with 14076 triangular primal elements.
The mesh is moved according to the boundary velocity imposed on the ellipse, while the outer boundaries are not moving. In the interior of the domain, the mesh velocity is obtained via the solution of the Laplace equation. A detail of the moving mesh around the ellipse at time $t=0$ and $t=0.06$ can be seen in Figure~\ref{fig:ellipse_mesh}.
\begin{figure}
	\begin{center}
	\includegraphics[trim=5 5 5 5,clip,width=0.49\linewidth]{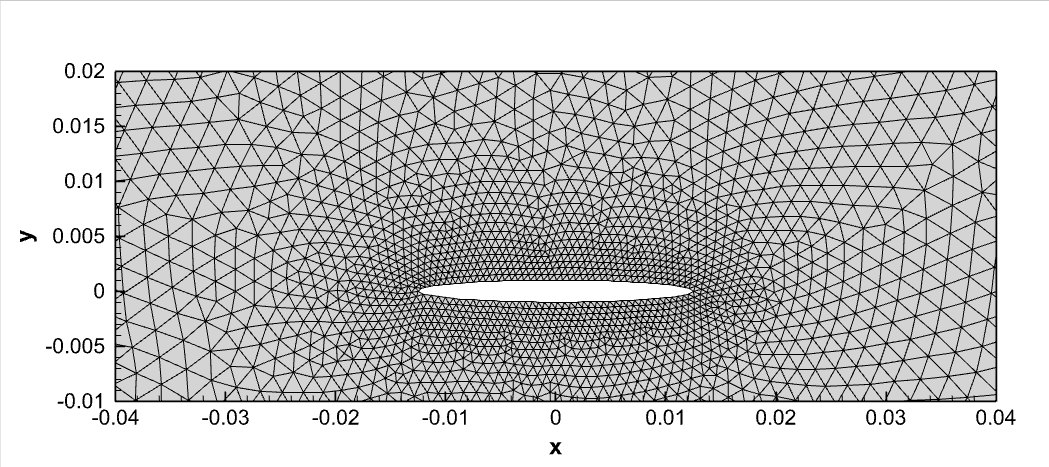}
	\includegraphics[trim=5 5 5 5,clip,width=0.49\linewidth]{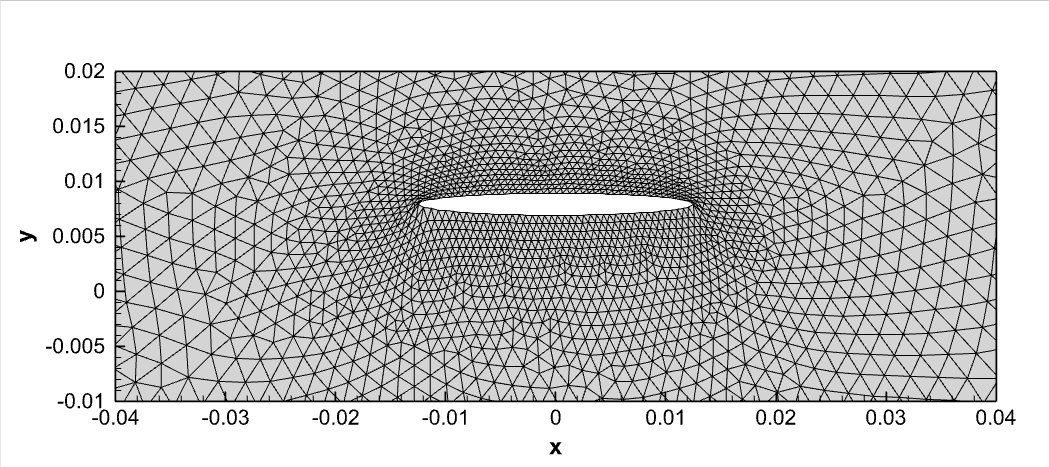}
	\end{center}
	\caption{Detail of the moving mesh at time $t=0$ (left) and $t=0.06$ (right) of the flow around a moving ellipse test case.}
	\label{fig:ellipse_mesh}
\end{figure}
The initial velocity and the density are $u_0=1$ and $\rho=1$, respectively, and the initial pressure is given by $p=0$. {At} the outer boundary a pressure outlet condition is considered and around the ellipse a viscous wall boundary is imposed. To simulate a dragonfly flight, we set $Re=1000$, the flapping frequency $f=40$ and the amplitude $A=4\cdot 10^{-3}$, to obtain the characteristic Strouhal numbers $\St_A=0.16$ and $\St_c=1$. The laminar viscosity of the fluid is set to $\mu= 2.5\cdot 10^{-5}$. In Figure~\ref{fig:ellipse_drag_lift}, we plot the drag and lift coefficients obtained for the moving ellipse. These values present a good agreement with those reported in~\cite{Wang2000}. To get an idea of the generated flow pattern and the vortex shedding, the vorticity contours at times $t=0.08$, $t=0.11$, $t=0.16$, and $t=0.2$ are shown in Figure~\ref{fig:ellipse_vorticity}.

\begin{figure}[h!]
	\begin{center}
	\includegraphics[width=0.75\linewidth]{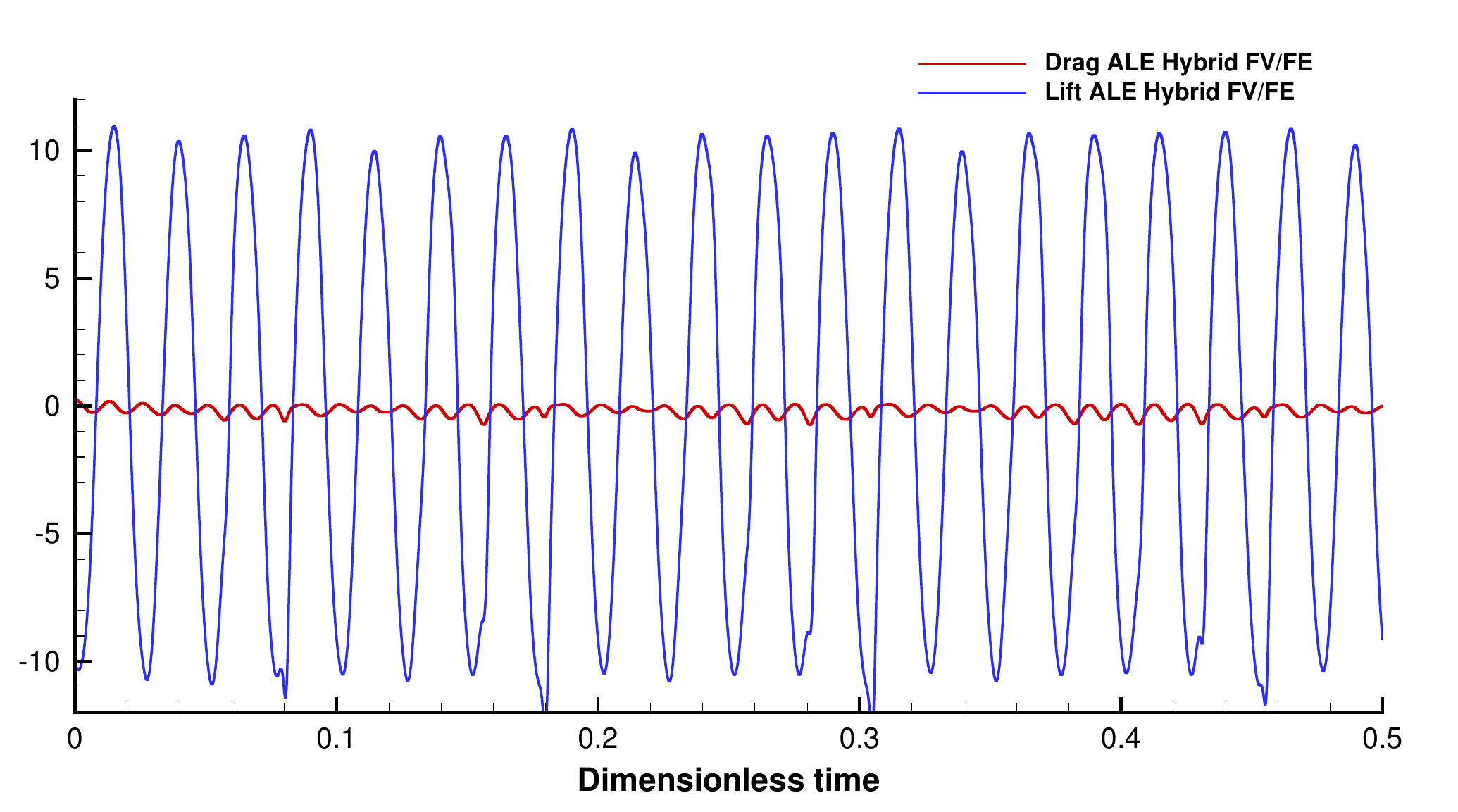}
	\end{center}
	\caption{Drag and lift coefficients of a moving ellipse}
	\label{fig:ellipse_drag_lift}
\end{figure}

\begin{figure}[h!]
	\begin{center}
	\includegraphics[trim=5 5 5 5,clip,width=0.49\linewidth]{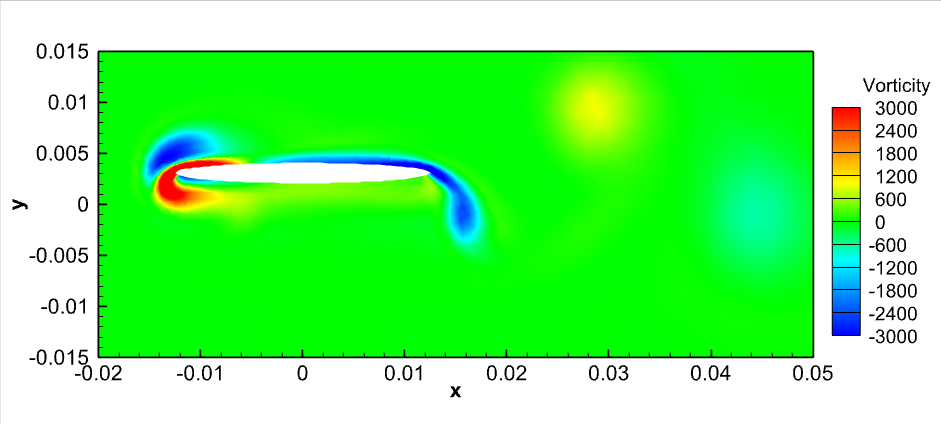}
	\includegraphics[trim=5 5 5 5,clip,width=0.49\linewidth]{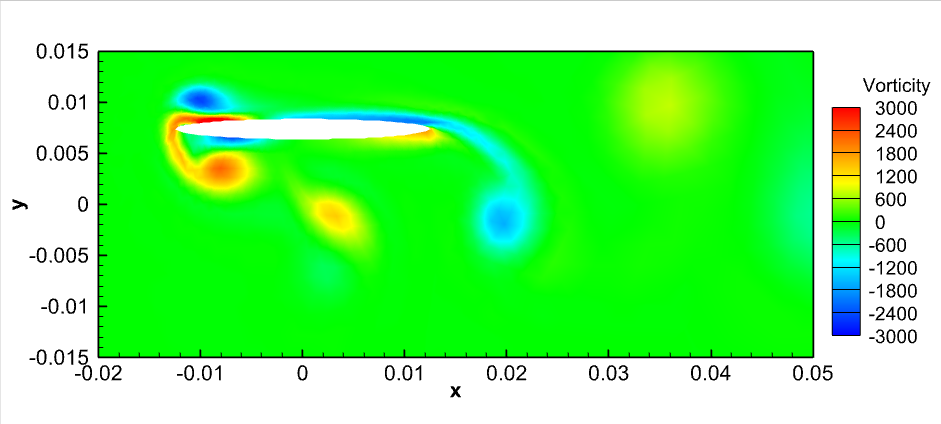}\\
	\includegraphics[trim=5 5 5 5,clip,width=0.49\linewidth]{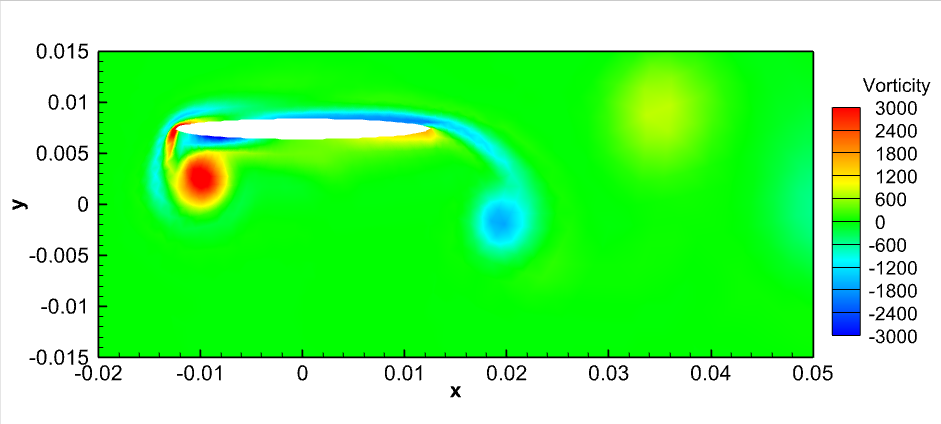}
	\includegraphics[trim=5 5 5 5,clip,width=0.49\linewidth]{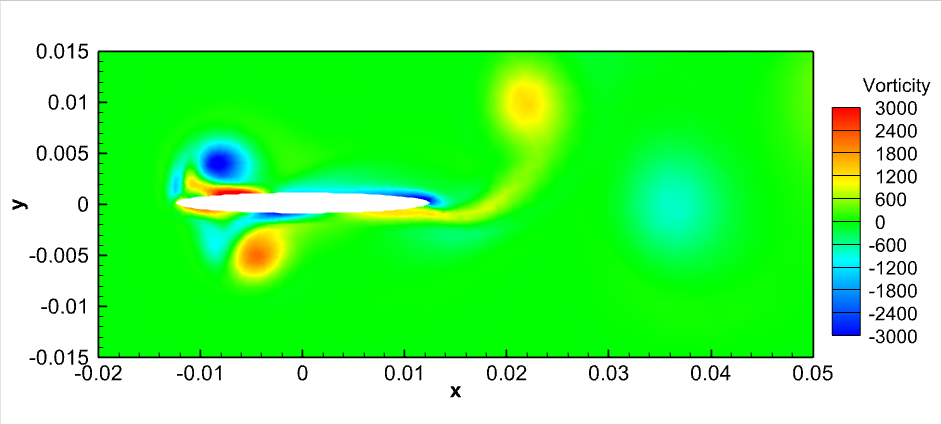}
	\end{center}
	\caption{Contour plots of the vorticity at times $t=0.08$, $t=0.11$, $t=0.16$, and $t=0.2$, of the viscous flow around a moving ellipse test.}
	\label{fig:ellipse_vorticity}
\end{figure}

\subsection{Rising bubble}
This test case aims at assessing the capability of the proposed method to deal with natural convection problems.
Following \cite{MBGW13,YMLGW14,BLY17,Yi18,BTBD20,BBDFSVC20}, we consider a closed domain  $\Omega=[0,2]\times[0,3]$ and define a heated bubble embedded in an initial fluid at rest:
\begin{equation*}
	\rho\left(\x,0\right)=1,\qquad \mathbf{u} \left(\x,0\right) = 0,\qquad \press \left(\x,0\right) = 10^{5}+ yg_{2}, \qquad
	\T\left(\x,0\right) =  \left\lbrace \begin{array}{lr}
		T_0+0.5 e^{-\frac{1}{0.08} r^2} & \mathrm{ if } \; r^{2} \le 0.2,\\[10pt]
			T_0 & \mathrm{ otherwise},
	\end{array}\right.
\end{equation*} 
with $r$ the distance to the centre of the bubble located at $\x_b=(1,1)$. The left and right boundaries are assumed to be fixed walls while {at} the top and bottom boundaries we set Dirichlet boundary conditions for all variables. 
The mesh is moved with a smoothed version of the local fluid velocity, setting  $\varsigma= 10$. 

Firstly, we compare the solution obtained with the ALE scheme  against the Eulerian hybrid FV/FE method for a viscosity of $\mu=2\cdot 10^{-4}$ and which has already been validated in the framework of natural convection problems in \cite{BBDFSVC20}. {Here, both methods are applied to the incompressible Navier-Stokes equations coupled with the temperature transport equation through the Boussinesq assumption. For a comparison between this approach and the solution of the compressible Navier-Stokes equations see \cite{BTBD20,BBDFSVC20}}. The temperature contour plot at times $t=\left\lbrace 0.25, 0.5, 0.75\right\rbrace$ together with 1D cuts of the velocity and temperature field are reported in Figure~\ref{fig.bubblecomparative}, showing an excellent agreement between both methods. 

As second test, we run a simulation with $\mu=0$  using the ALE hybrid FV/FE scheme. The results presented in Figure~\ref{fig.bubblemu0} for $t=0.75$ show the elevation of the heated bubble and the corresponding mesh deformation. Moreover, we start seeing the formation of the classical Kelvin-Helmholtz instabilities on the bubble surface, as expected for a zero viscosity rising bubble, \cite{BTBD20}.

\begin{figure}
	\begin{center}
		\includegraphics[width=0.29\linewidth]{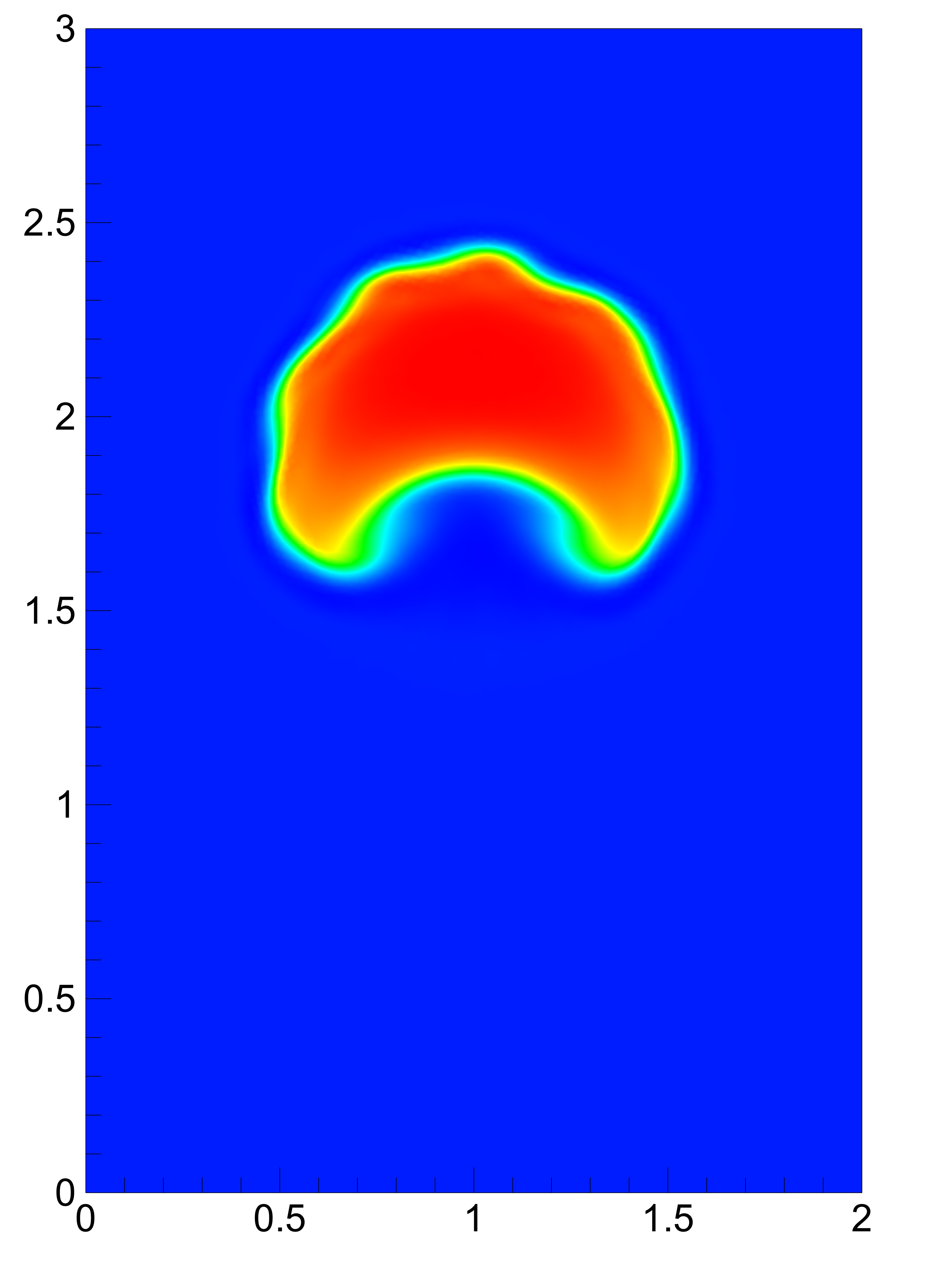}\hfill
		\includegraphics[width=0.29\linewidth]{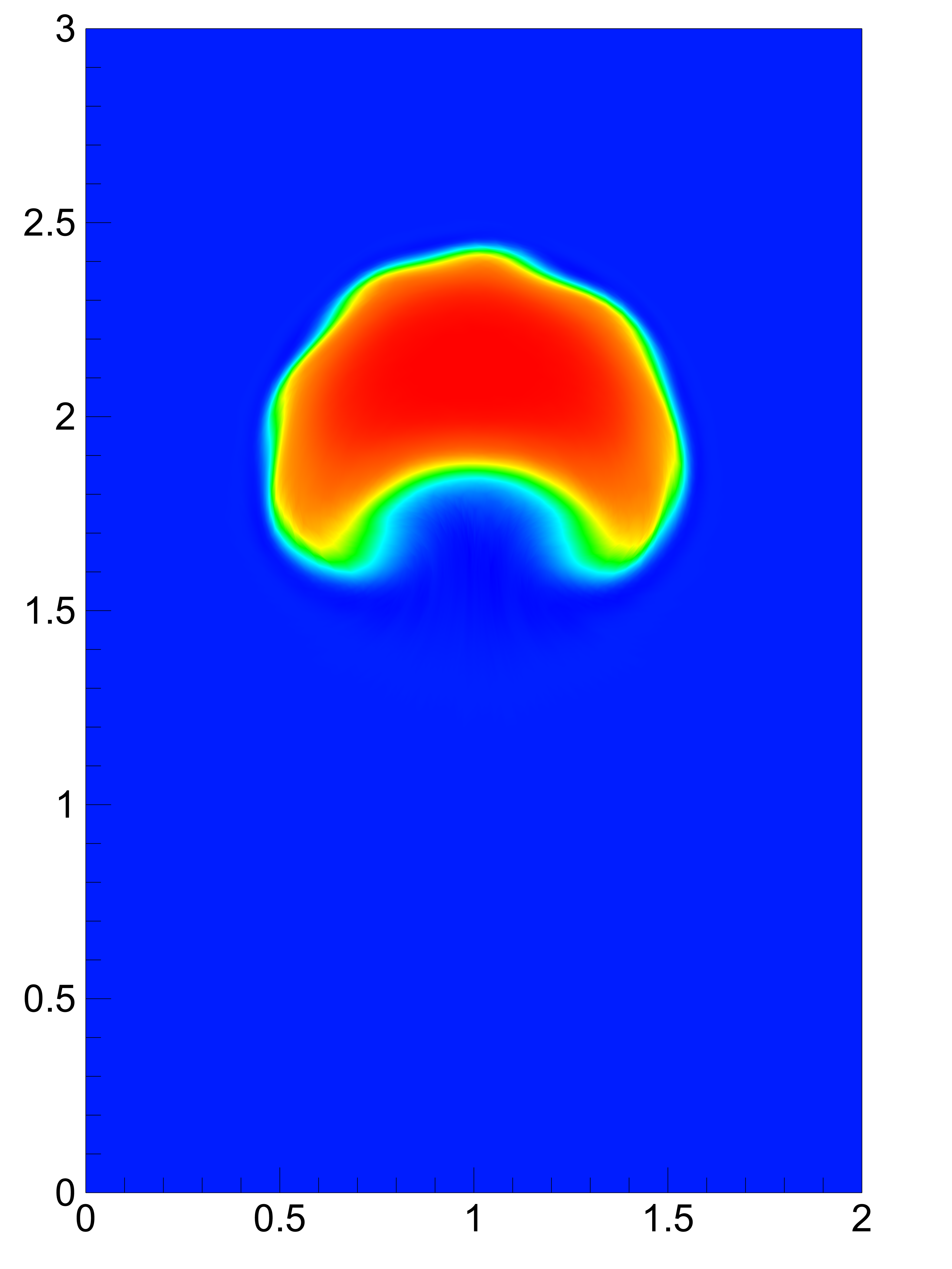}\hfill
		\includegraphics[width=0.29\linewidth]{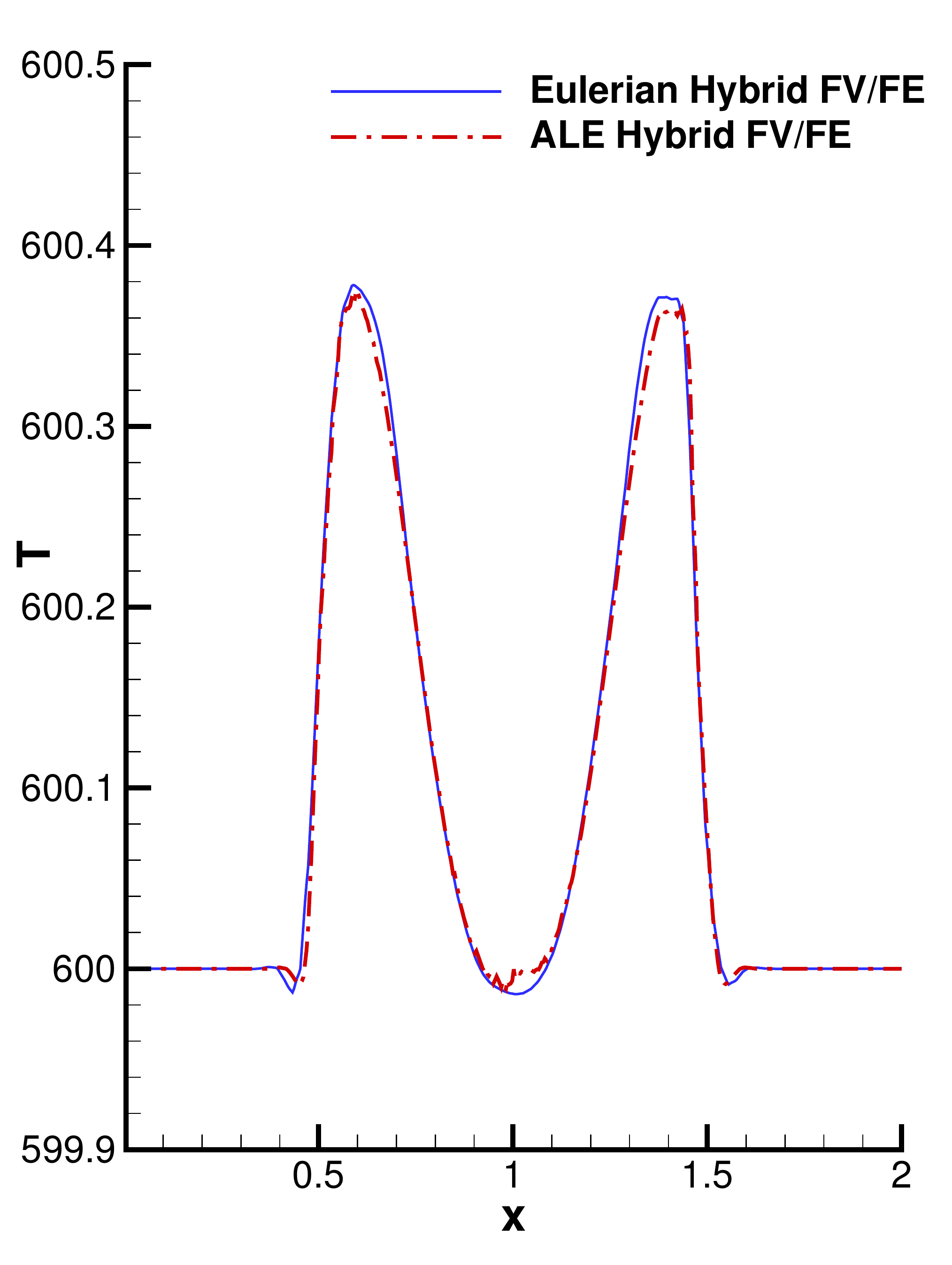}
		
		\vspace{0.2cm}		
		\includegraphics[width=0.29\linewidth]{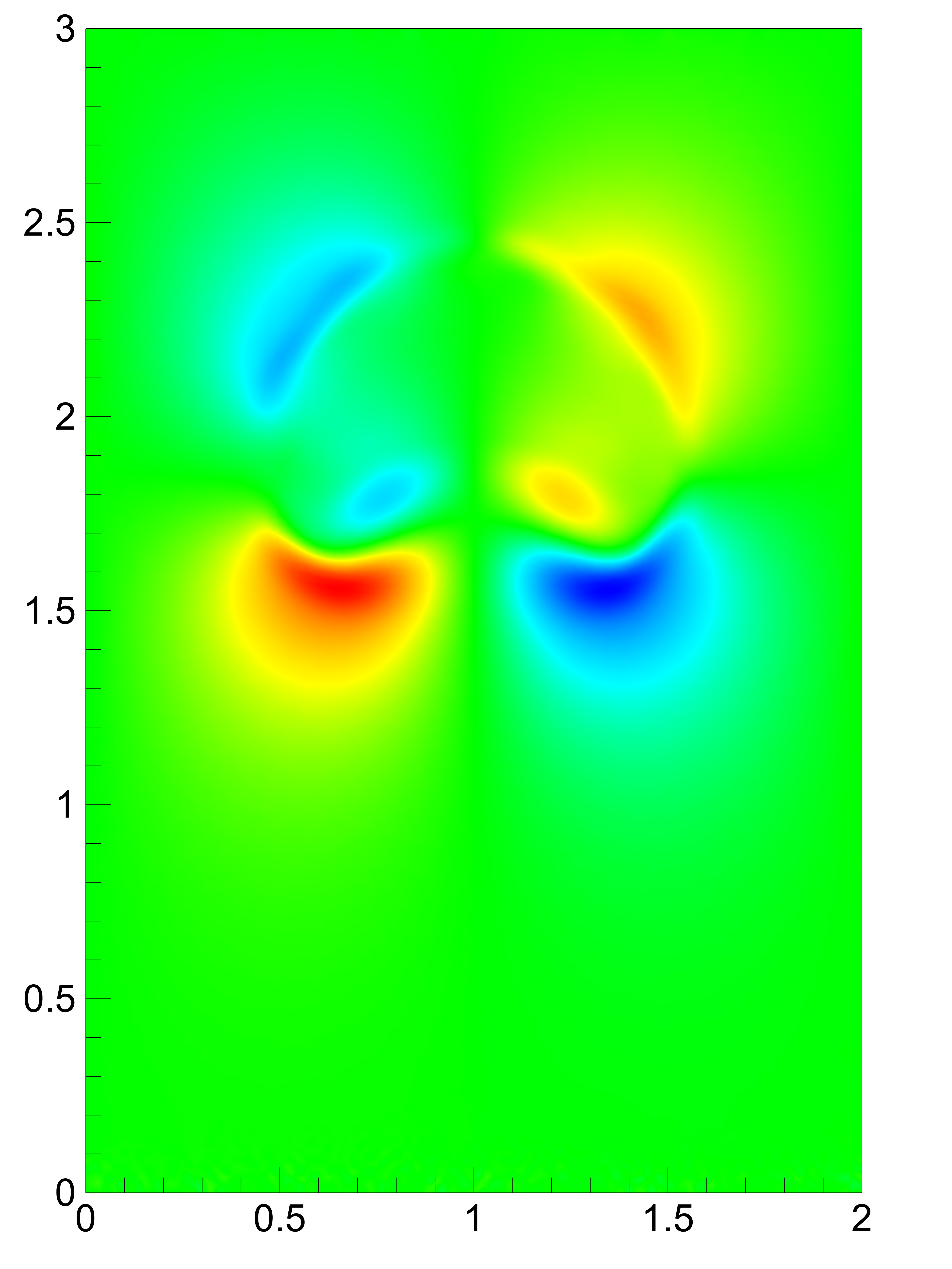}\hfill
		\includegraphics[width=0.29\linewidth]{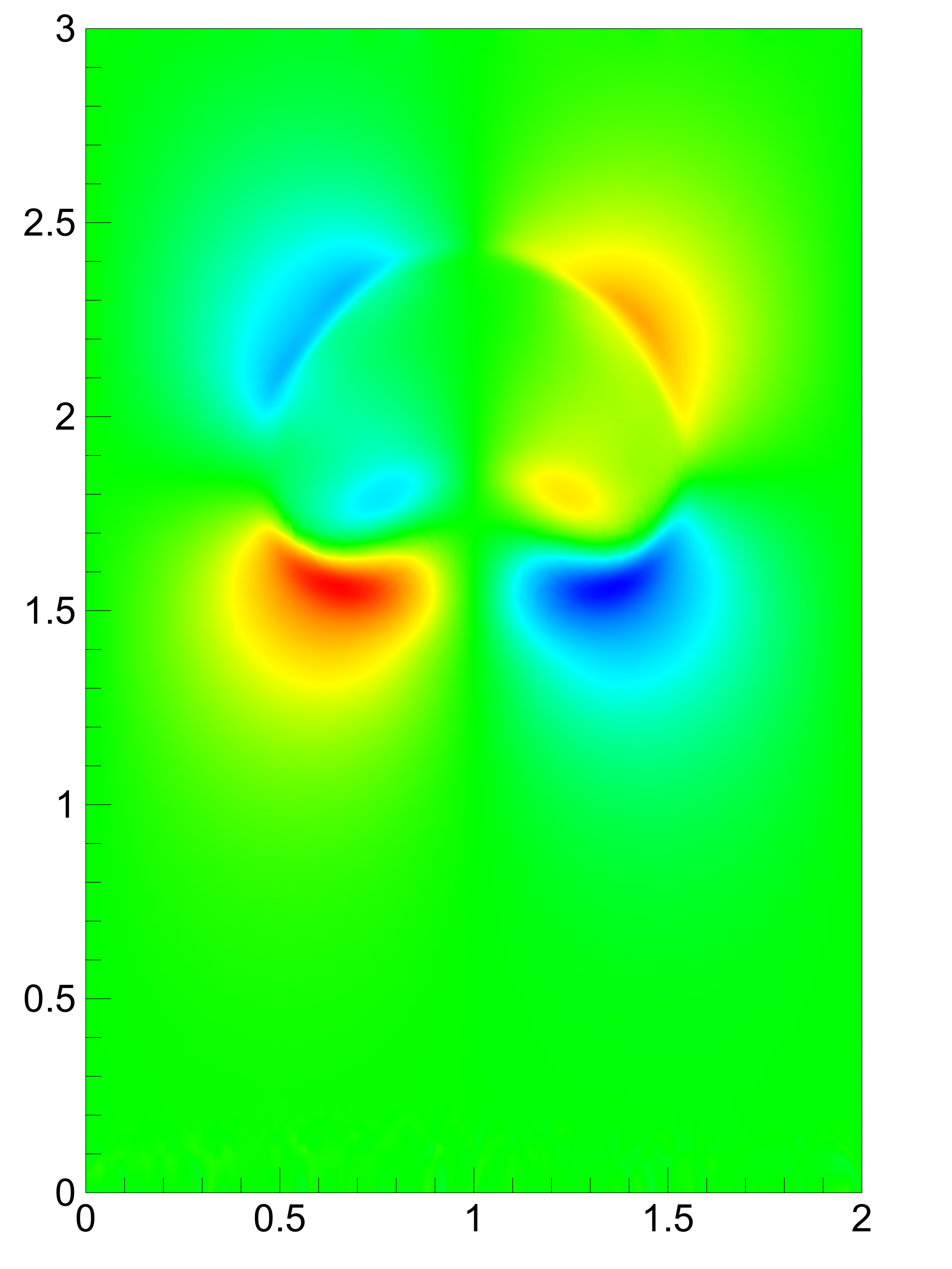}\hfill
		\includegraphics[width=0.29\linewidth]{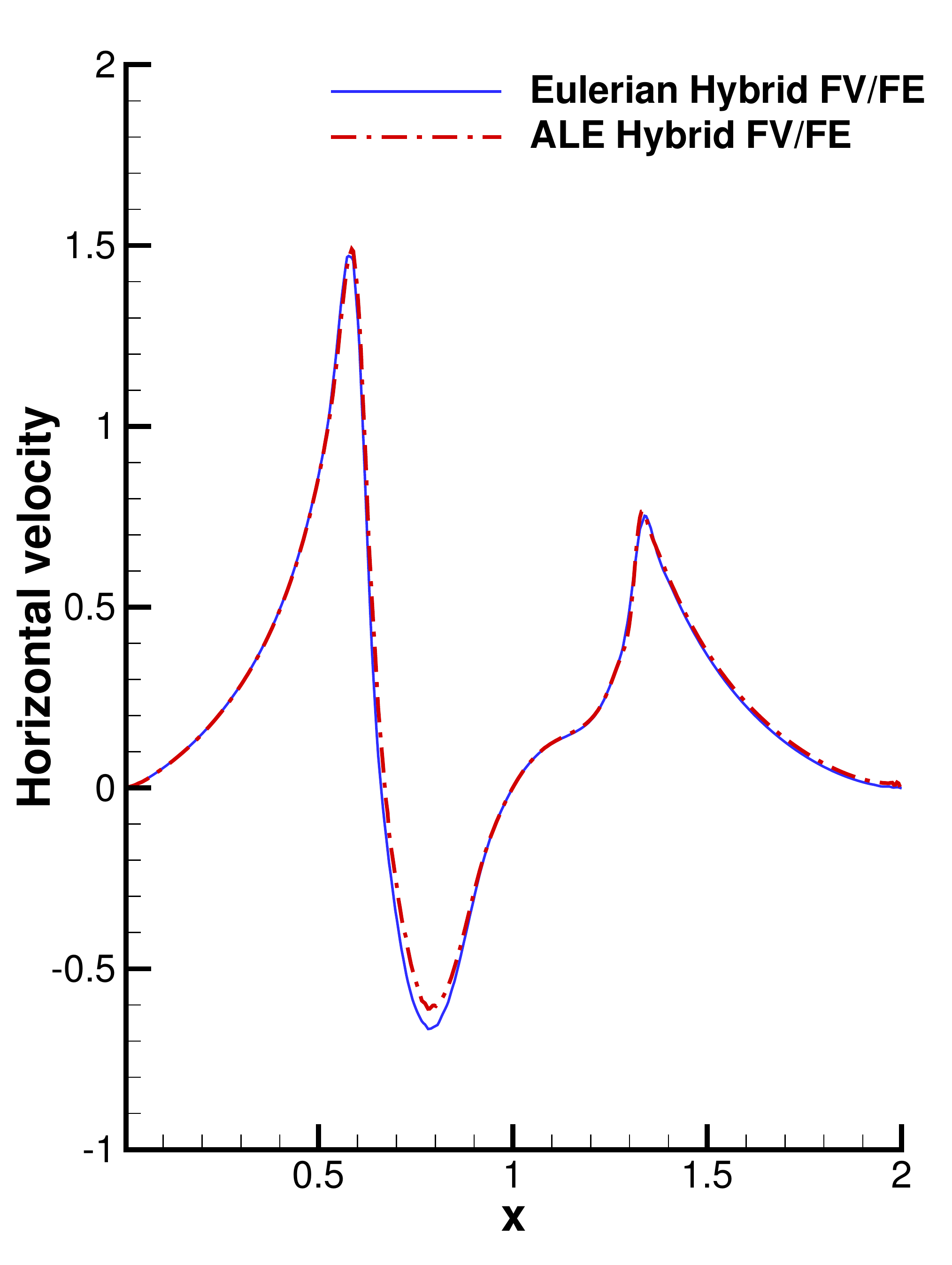}
		
		\vspace{0.2cm}
		\includegraphics[width=0.29\linewidth]{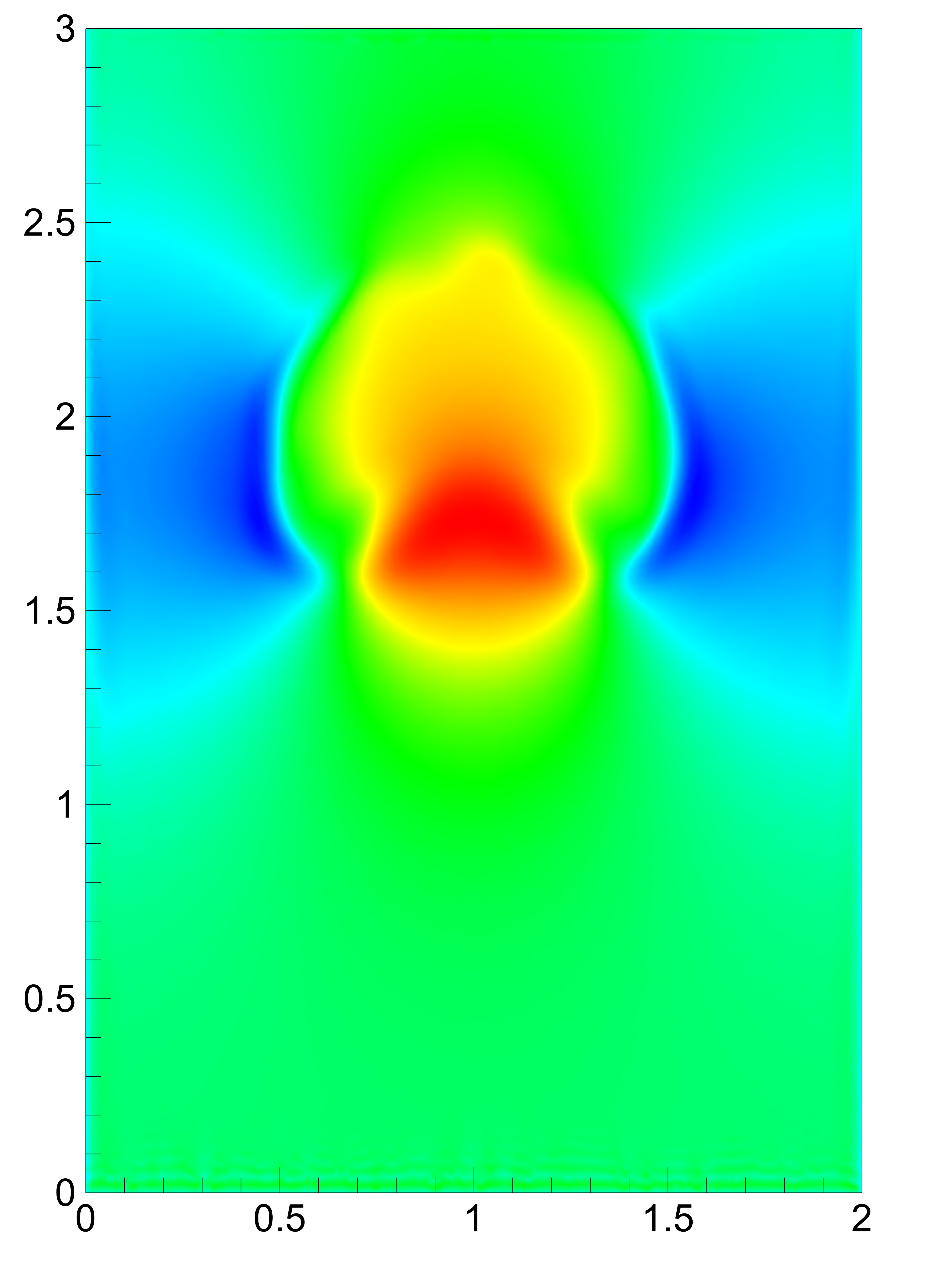}\hfill
		\includegraphics[width=0.29\linewidth]{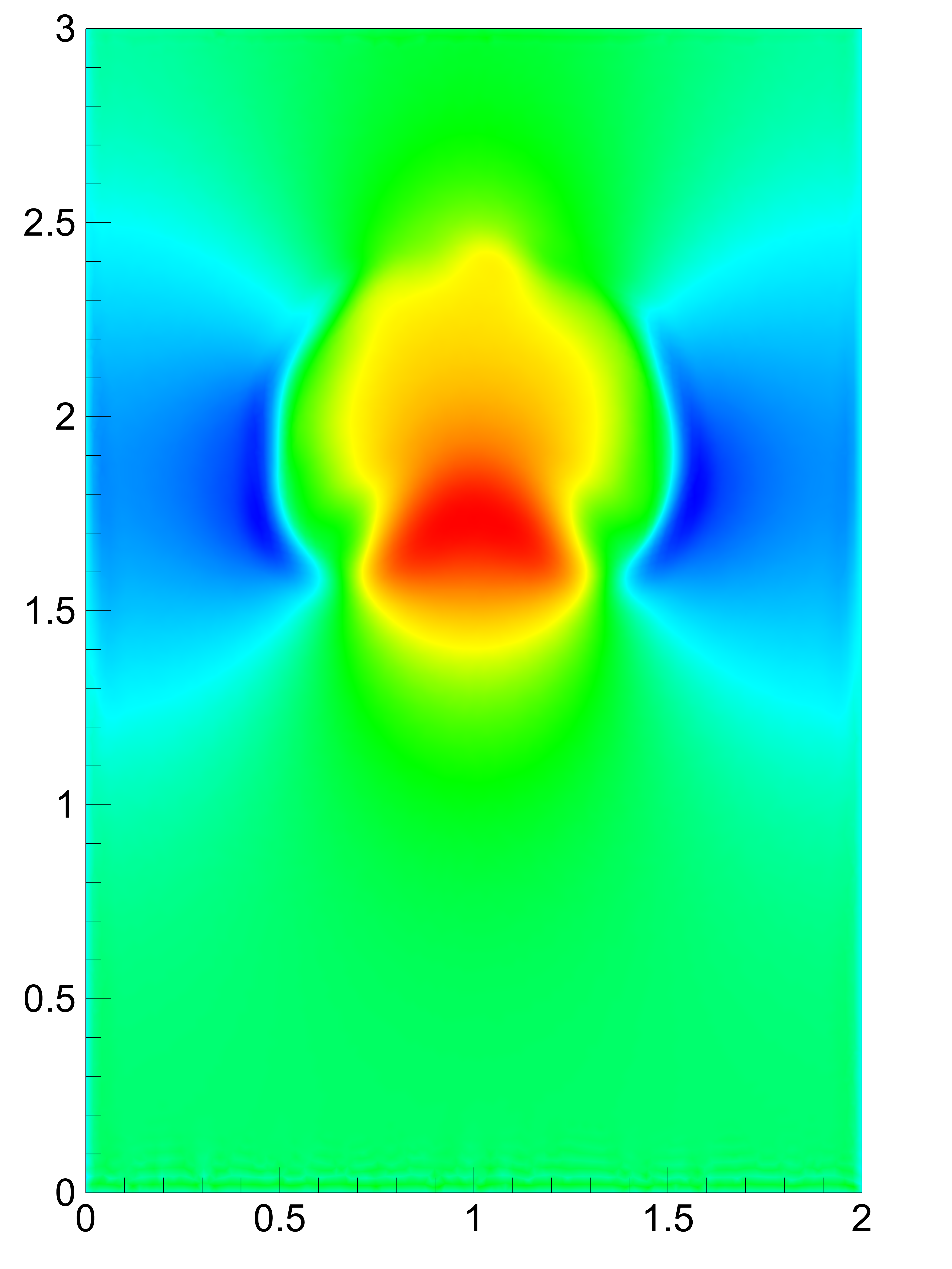}\hfill
		\includegraphics[width=0.29\linewidth]{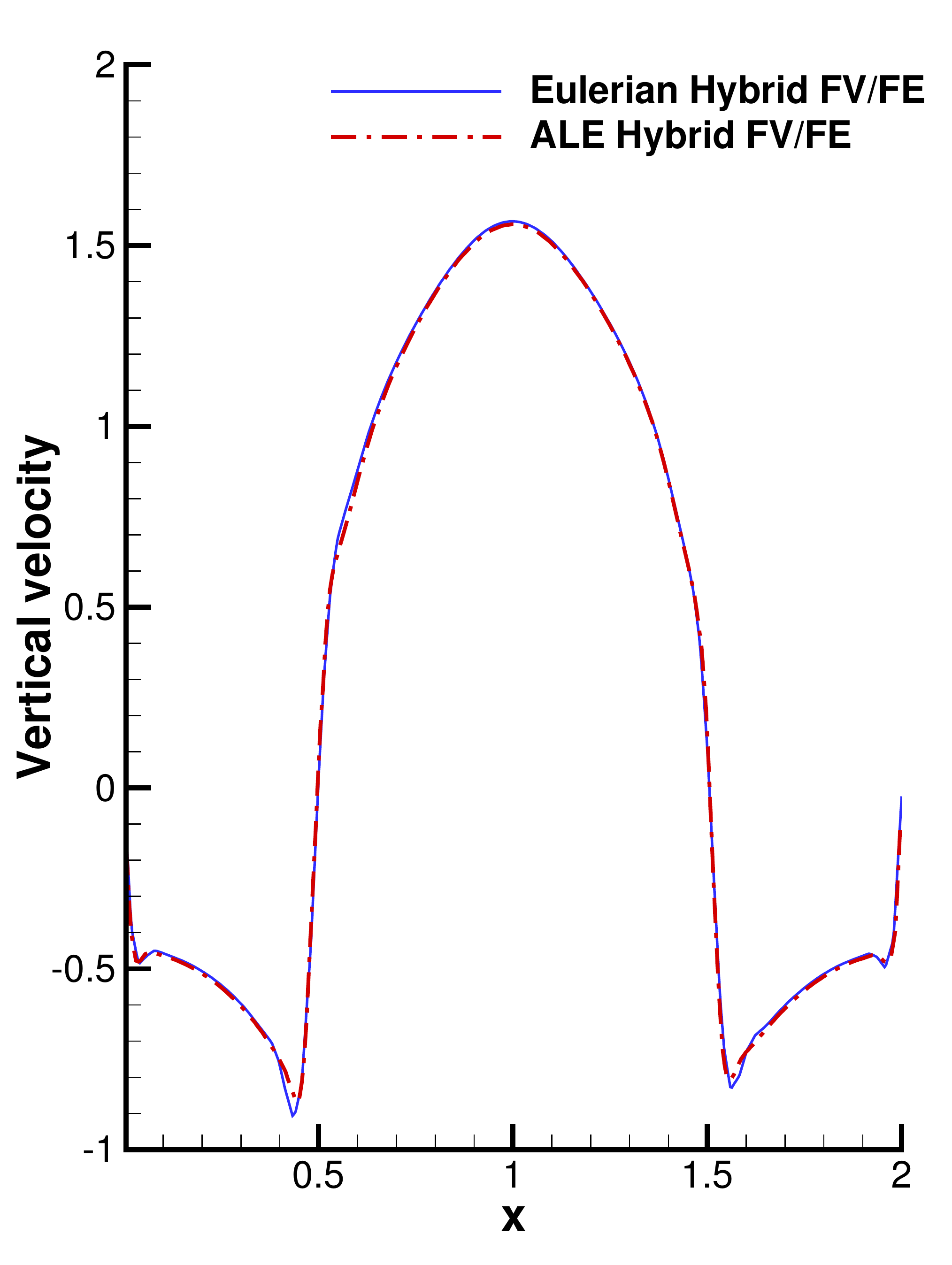}
	\end{center}
	\caption{Comparison of the numerical results obtained with the ALE and the Eulerian hybrid FV/FE schemes at $t=0.75$. The left and middle columns contain the contour plots of the temperature, horizontal and vertical velocity for the ALE with $\varsigma= 10$, $c_{\alpha}=0.5$, and the Eulerian hybrid FV/FE schemes, respectively. The right column shows the 1D profile over the cut $y=1.7$ for the temperature, $y=1+x$ for the horizontal velocity and  $y=2$ for the vertical velocity. From top to bottom rows: temperature, horizontal velocity and vertical velocity. }\label{fig.bubblecomparative}
\end{figure}

\begin{figure}
	\begin{center}
		\includegraphics[trim=5 5 5 5,clip,width=0.485\linewidth]{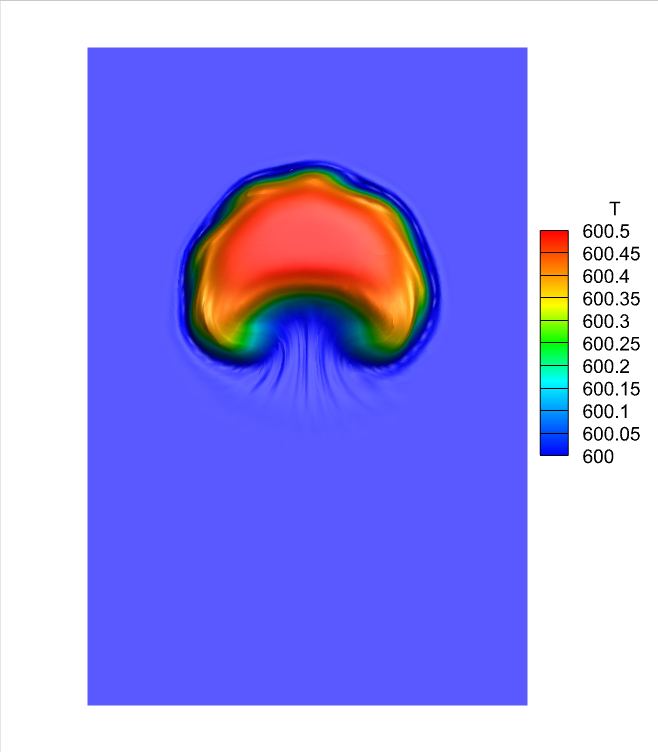}	\hfill
		\includegraphics[width=0.48\linewidth]{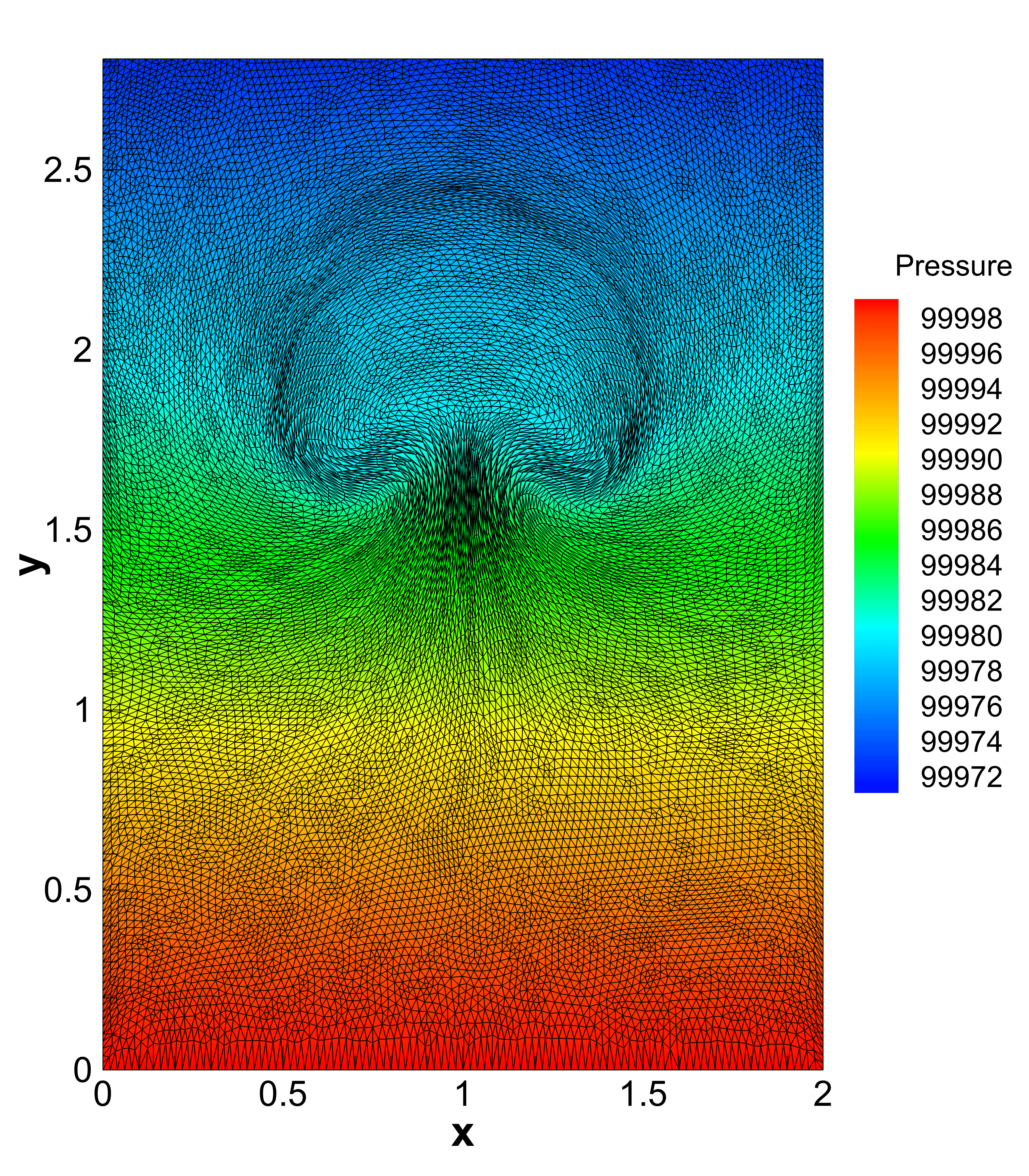}
	\end{center}
	\caption{Temperature contour plot and moving mesh coloured by the pressure field obtained solving the rising bubble test case with $\mu=0$ using the ALE hybrid FV/FE scheme with $\varsigma=10$, $c_{\alpha}=0.5$; $t=0.75$. }\label{fig.bubblemu0}
\end{figure}

\subsection{Dambreak}
We now consider {an incompressible} dambreak test case for $\mu=10^{-3}$ proposed in \cite{BCDGP20,Ferrari2021}. The initial fluid at rest  fills the region $\Omega=[0,2] \times [0,1]$ and the {upper} and right boundaries are assumed to disappear leading to a dambreak. To simulate this situation, pressure outlet boundary conditions with zero pressure are defined in the upper and right surfaces of the fluid, while the left and bottom boundaries are assumed to be fixed inviscid walls with zero velocity. The simulation is run in parallel, \cite{HybridMPI}, using MPI on 64 CPU cores of an AMD Ryzen Threadripper 3990X workstation until $t=0.5$ on a mesh made of $65536$ primal elements. To avoid a fast degeneration of elements in the free surface, the regularization parameter is set to $\varsigma=1$ and the results are depicted in Figure~\ref{fig:dambreak}. A good agreement is observed with the reference solution computed solving the Godunov-Peshkov-Romenski (GPR) model using a high order ADER DG scheme on a fixed mesh \cite{BCDGP20}. {To illustrate the mesh motion, the initial and final configuration of the MPI zones is depicted in Figure \ref{fig:dambreak_mpi}.}

\begin{figure}[h]
	\begin{center}
	\includegraphics[trim=5 5 5 5,clip,width=0.49\linewidth]{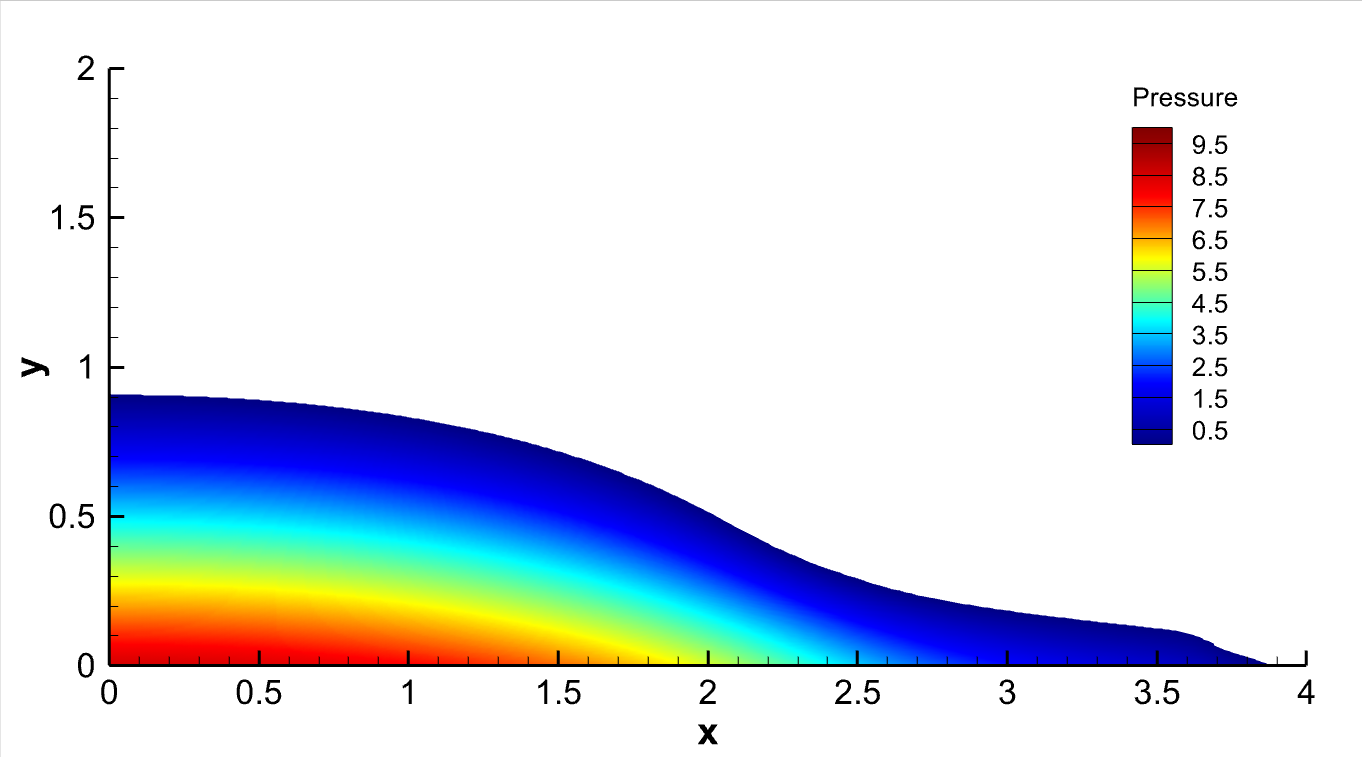}
	\includegraphics[width=0.49\linewidth]{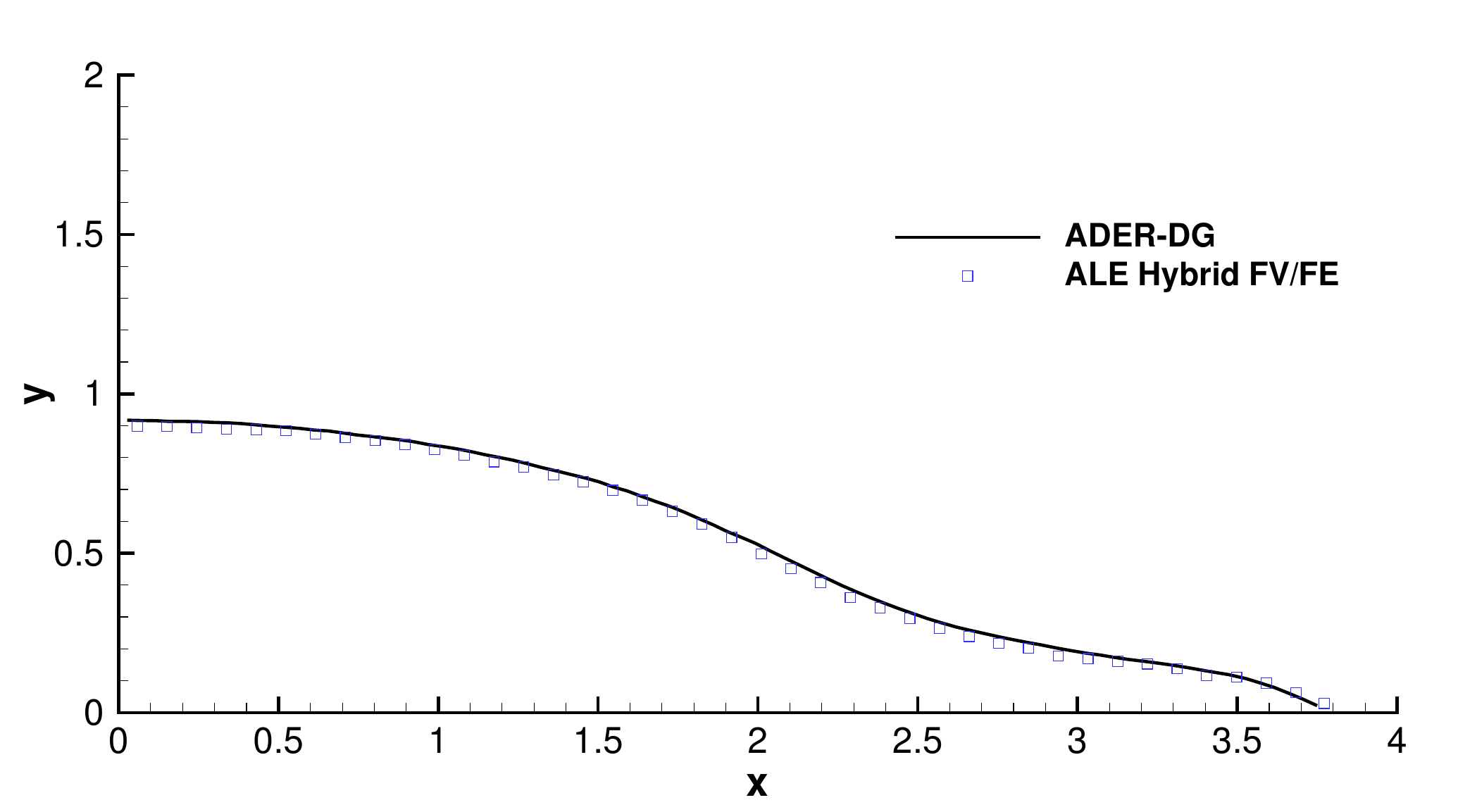}
	\end{center}
	\caption{Dambreak test case. Left: position of the fluid coloured by the pressure field at $t=0.5$. Right: free surface elevation of the dam break test case obtained at $t=0.5$ using the Hybrid FV/FE (blue squares) and with a fourth order ADER-DG scheme, \cite{BCDGP20} (solid black line). }
	\label{fig:dambreak}
\end{figure}

\begin{figure}[h]
	\begin{center}
		\includegraphics[trim=5 5 5 5,clip,width=0.49\linewidth]{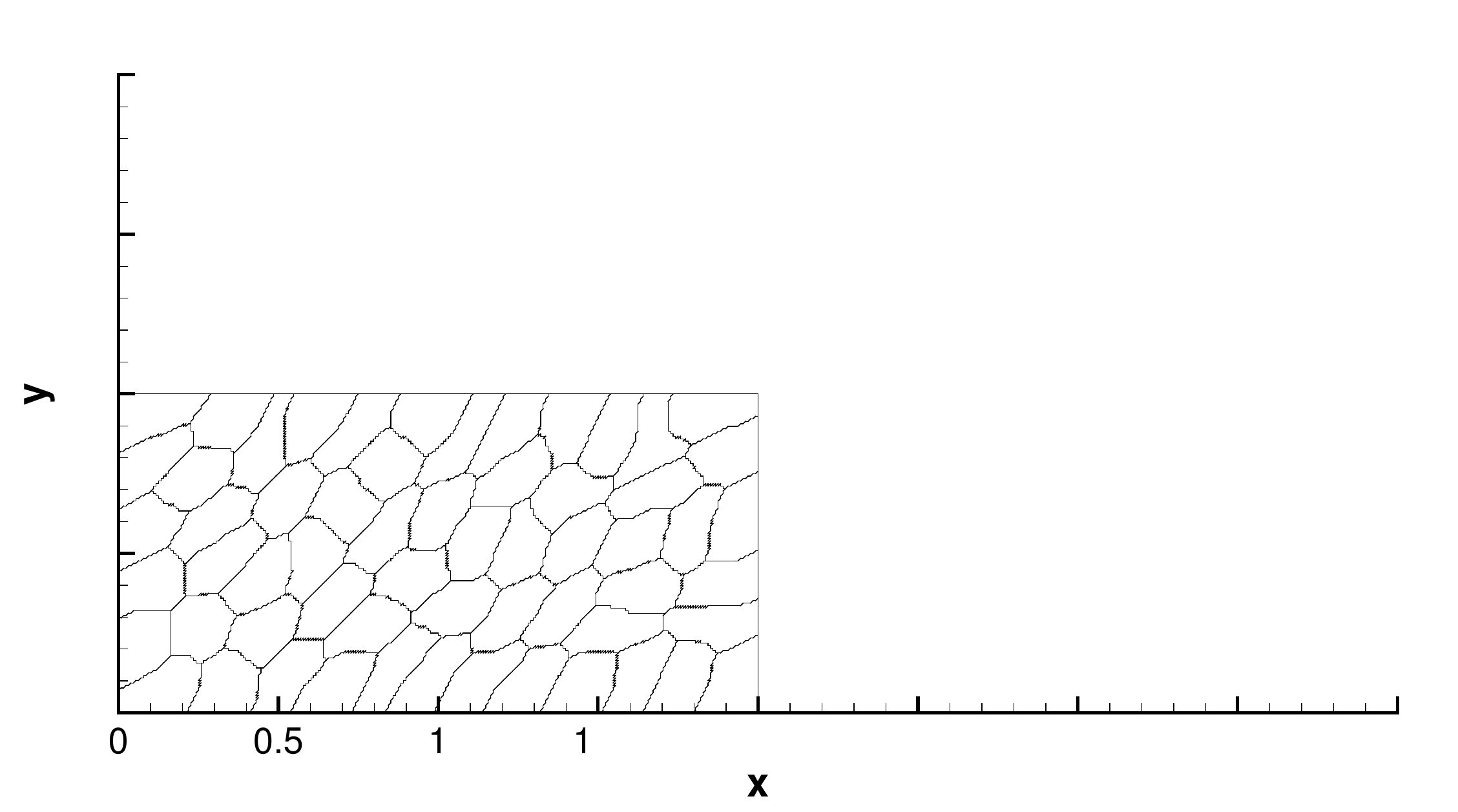}
		\includegraphics[trim=5 5 5 5,clip,width=0.49\linewidth]{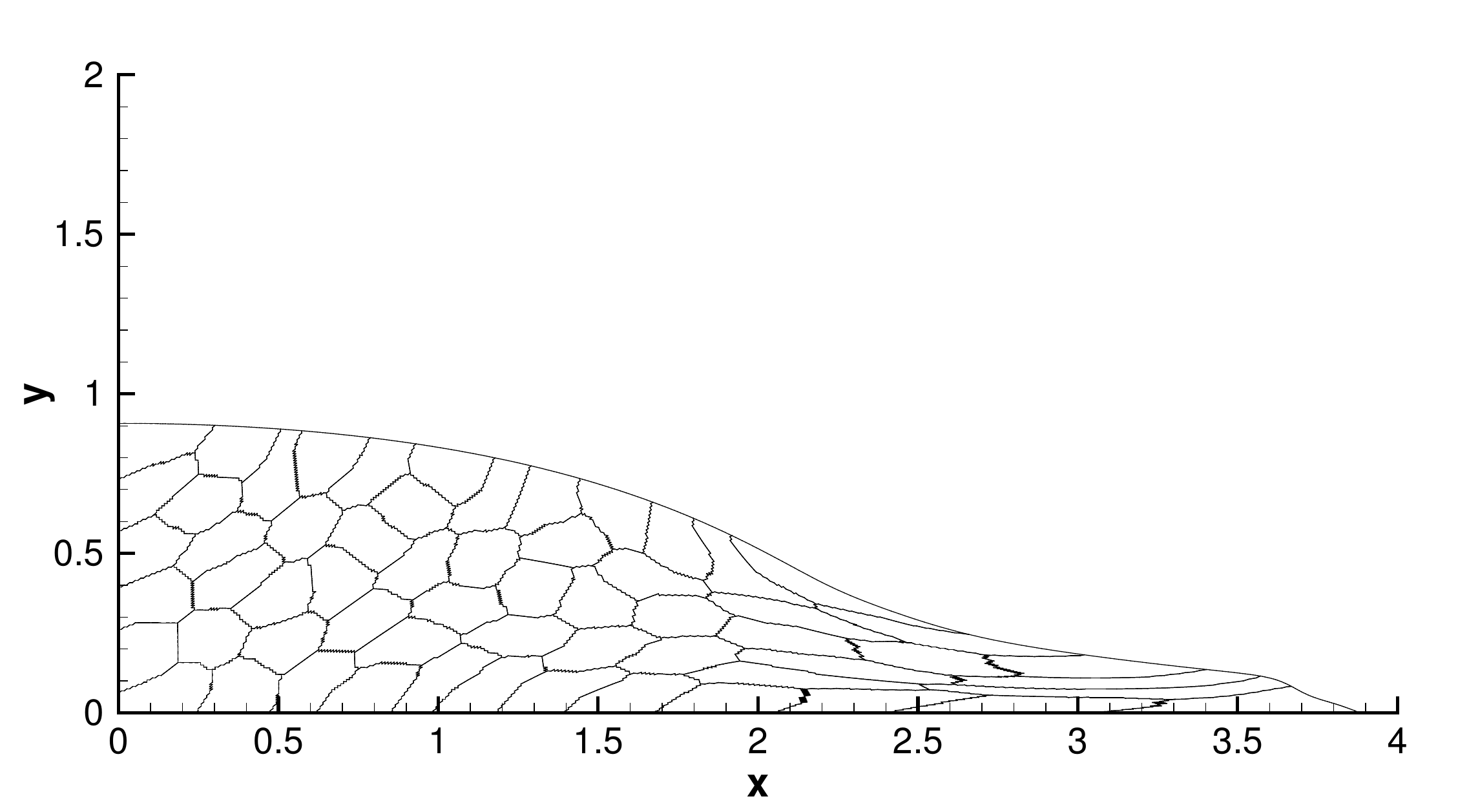}
	\end{center}
	\caption{{Distribution of the domain onto 64 MPI ranks for the dam break test case. Left: $t=0$. Right: $t=0.75$. }}
	\label{fig:dambreak_mpi}
\end{figure}

\subsection{Sloshing}

In this section, we simulate the phenomenon of sloshing of an incompressible inviscid fluid in a moving tank. Viscosity is neglected, i.e. $\mu=0$. We use the setup detailed in \cite{LagrangeNC} and references therein. The tank has a length of $L=1.73$ and the initial water height is $H_0=0.6$. The fluid inside the tank is initially at rest and the initial pressure distribution is hydrostatic. The gravity vector is set to $\mathbf{g}=(0,-9.81)$. The boundary conditions are as follows: we impose a free surface pressure boundary condition with $p=0$ at the free surface on the top and moving slip wall boundary conditions on the left, right and bottom. The horizontal velocity of the tank walls is prescribed according to the law $u_1(t)= - \omega A \sin(\omega t)$ with $A=0.032$, $\omega = 2 \pi / T$ and $T=1.3$. Inside the tank the fluid moves with the regularized local fluid velocity, where the regularization parameter in the parabolic mesh smoothing equation is chosen as $\varsigma = 10^3$. Simulations are run until $t=9$ on a mesh composed of 4256 triangular elements with an initial mesh spacing of $h=0.02276$. In Figure \ref{fig:sloshing2d}, we show the free surface elevation and the pressure contours for several time instants, while in Figure 
\ref{fig:sloshingseries}, a time series of the relative free surface elevation $H-H_0$ of a tracer point initially located in $\mathbf{x}=(0.05,0.6)$ is depicted. Figure \ref{fig:sloshingseries} also contains a comparison with experimental data \cite{FaltinsenTimokha2000} and a third order ADER WENO ALE scheme applied to a simplified diffuse interface model of weakly compressible free surface flows on moving meshes, see \cite{LagrangeNC}. 

\begin{figure}
	\begin{center}
		\includegraphics[trim=8 8 8 8,clip,width=0.65\linewidth]{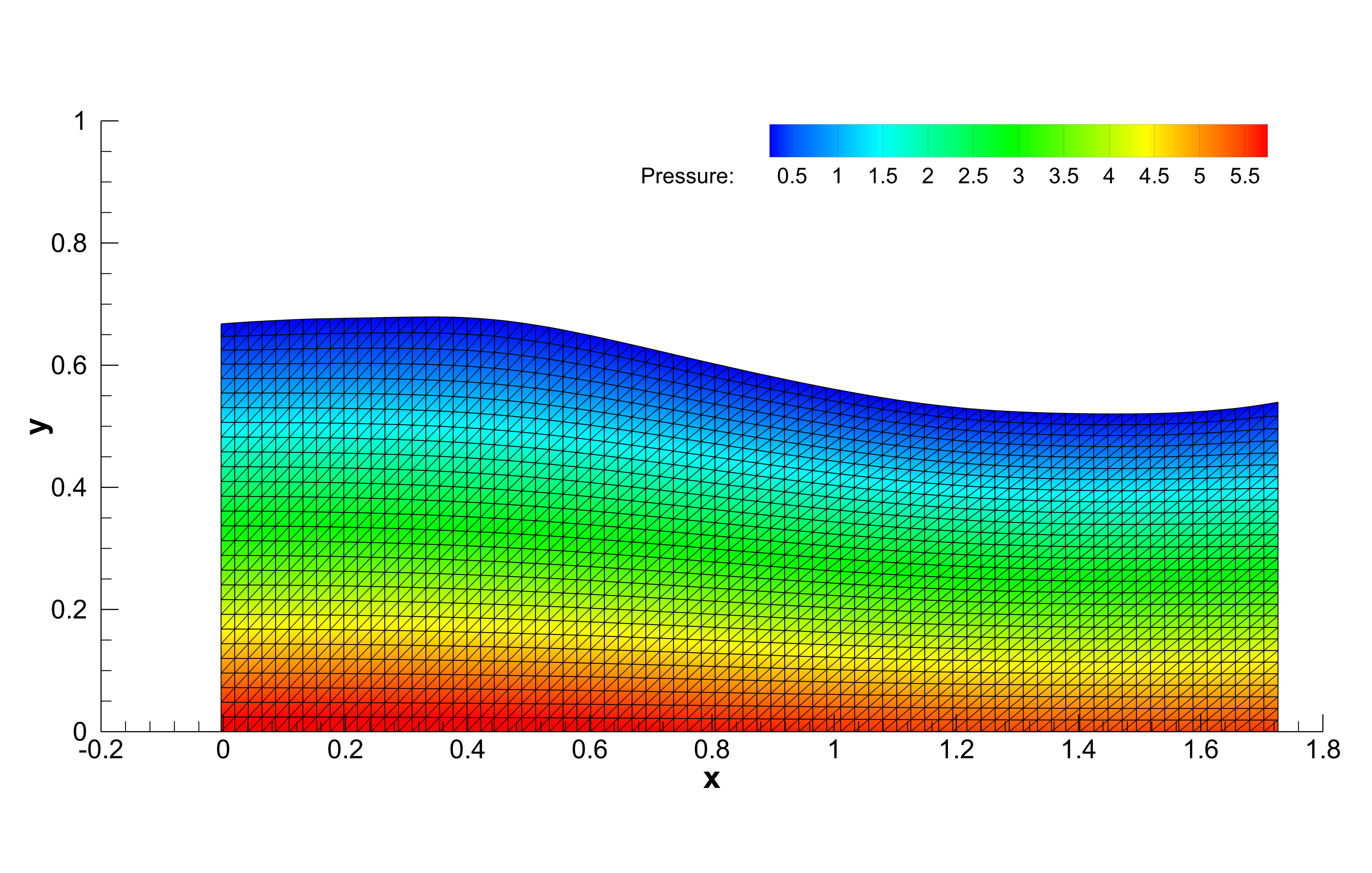} \\ 
		\includegraphics[trim=8 8 8 8,clip,width=0.65\linewidth]{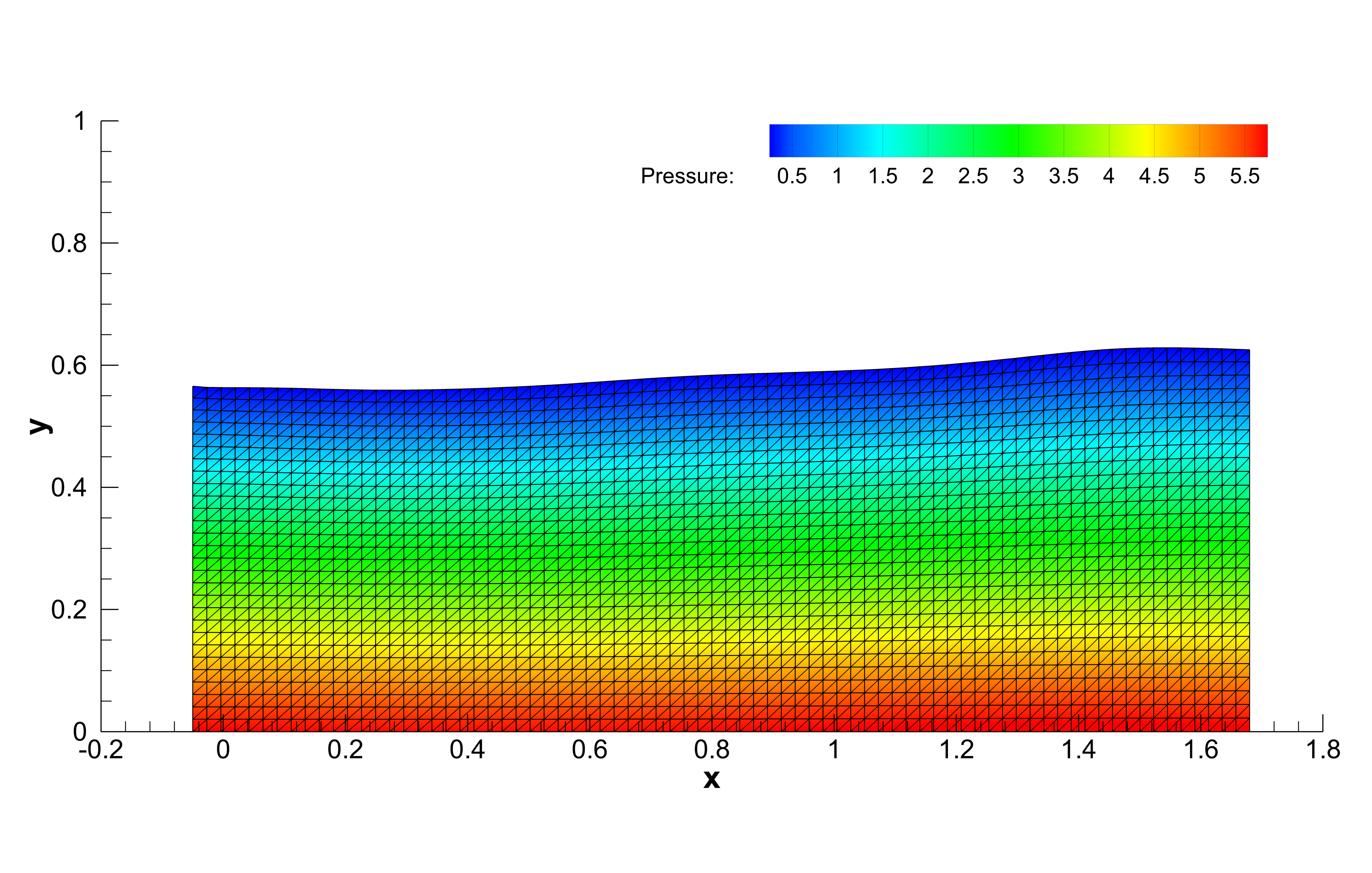} \\ 
		\includegraphics[trim=8 8 8 8,clip,width=0.65\linewidth]{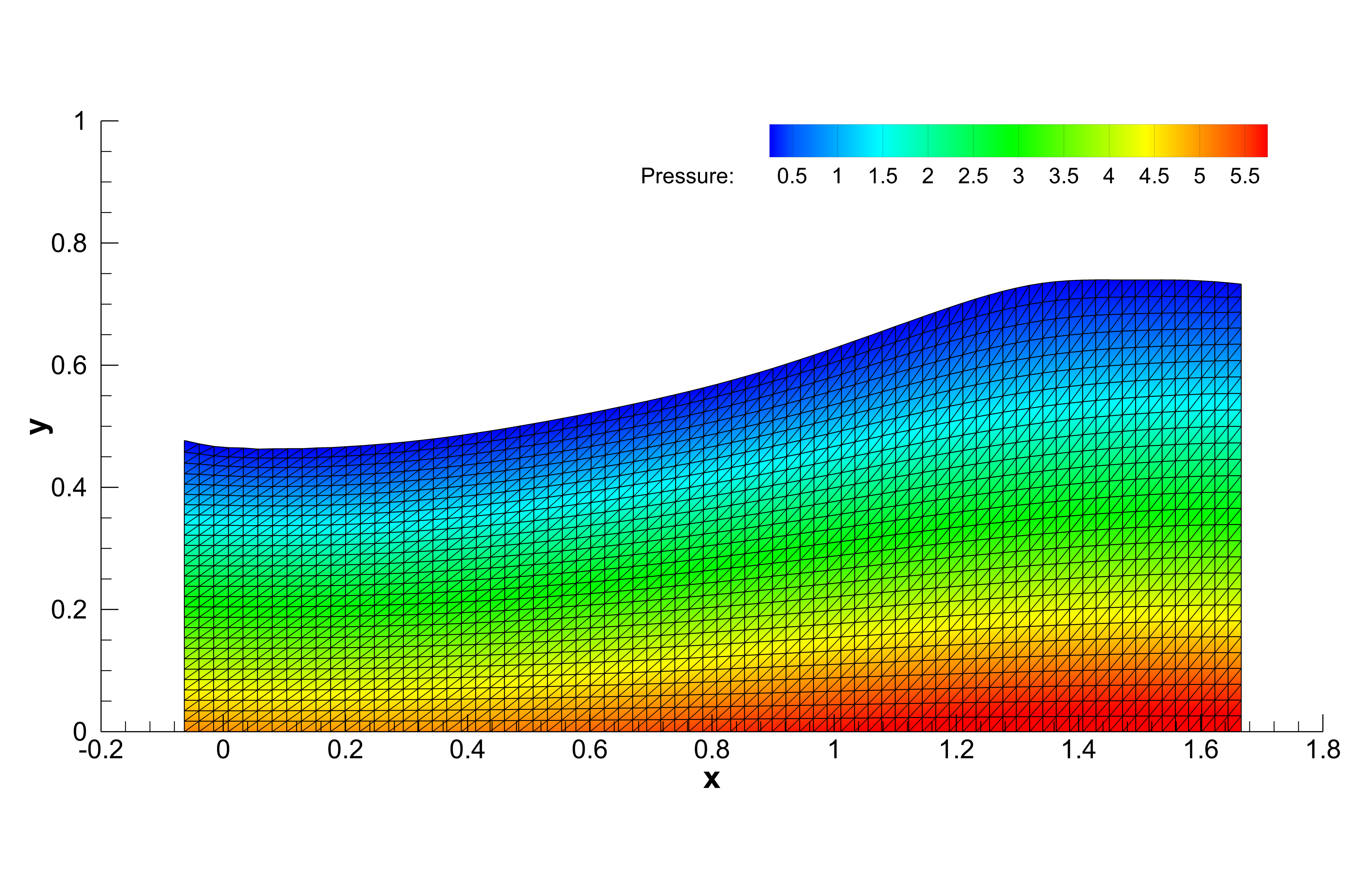}  
	\end{center}
	\caption{Free surface profile and pressure contours for the sloshing problem at time $t=1.2$, $t=2.15$ and $t=3.28$.}
	\label{fig:sloshing2d}
\end{figure}
 
\begin{figure}
	\begin{center}
		\includegraphics[trim=8 8 8 8,clip,width=0.65\linewidth]{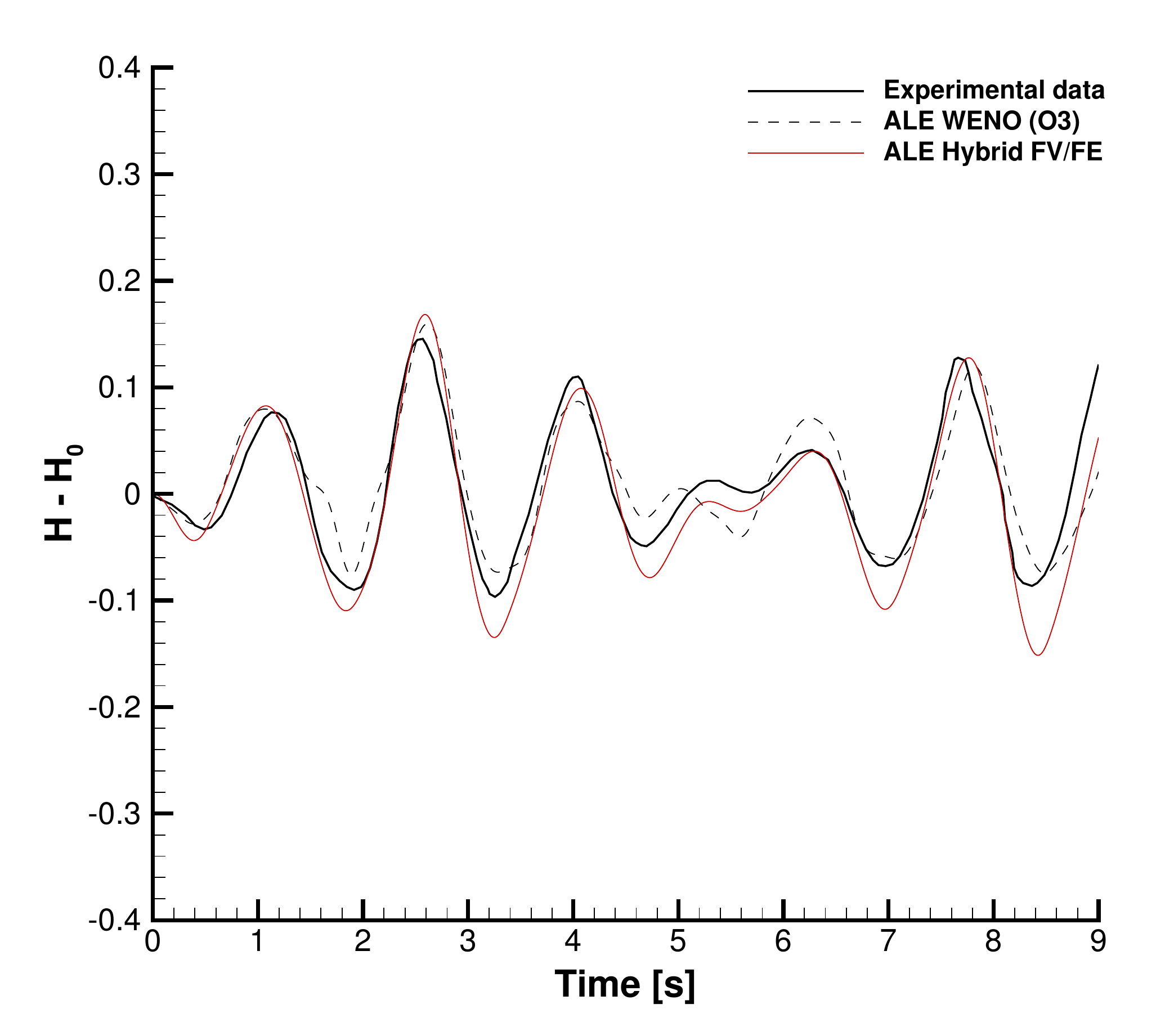} \\ 
	\end{center}
	\caption{Time series of the relative free surface elevation $H-H_0$ of a tracer point initially located in $\mathbf{x}=(0.05,0.6)$. Comparison with experimental data and a third order ADER WENO scheme applied to a simplified diffuse interface model of weakly compressible free surface flows. } 
	\label{fig:sloshingseries}
\end{figure}

\subsection{Circular explosion 2D}\label{sec:2DCE}
The last test case is a circular explosion and is run with the weakly compressible solver. The initial conditions correspond to the 2D version of the well-known Sod problem, \cite{Toro,TT05,DPRZ16}, i.e. 
\begin{equation*}
	\bw(t,\mathbf{x}) =  \left\lbrace 
	\begin{array}{lr}
		\bw_L & \mathrm{ if } \; r \le 0.5,\\
		\bw_R & \mathrm{ if } \; r > 0.5,
	\end{array}
	\right. 
	\qquad \bw_L = \left( 1, 0, 0, 1\right),\quad \bw_R = \left( 0.125, 0, 0, 0.1\right) .
\end{equation*}
Besides, we set $\mu=2\cdot 10^{-4}$ as physical viscosity and an artificial viscosity of $c_{\alpha} = 1$ (see~\cite{BBDFSVC20} for further details about this parameter). The computational domain $\Omega=[-1,1]^2$ is assumed to be periodic everywhere. Regarding the numerical method, we consider the hybrid second order ALE scheme with ENO reconstruction and the Rusanov flux function \eqref{eqn.Rusanov_flux}. 
To avoid a fast degeneration of elements the mesh velocity has been determined from the local fluid velocity considering a smoothing parameter of $\varsigma=5$.
The spatial mesh, made of $85344$ primal elements, properly follows the wave fronts, as it can be observed in Figure~\ref{fig.CE85_ale_t025_mesh} where the  initial and the deformed mesh, at $t_{\mathrm{end}}=0.25$, are depicted. 
To validate the numerical results obtained, we follow  \cite{Toro} and solve a simplified 1D partial differential equation in radial direction derived from the Euler equations.
Figure~\ref{fig.CE85_ale_t025} shows a good agreement between the results obtained with the new ALE hybrid FV/FE method and the reference solution computed using a second order TVD scheme on a grid made of $10000$ elements.

\begin{figure}[h]
	\centering
	\includegraphics[trim= 10 5 5 5,clip,width=0.45\linewidth]{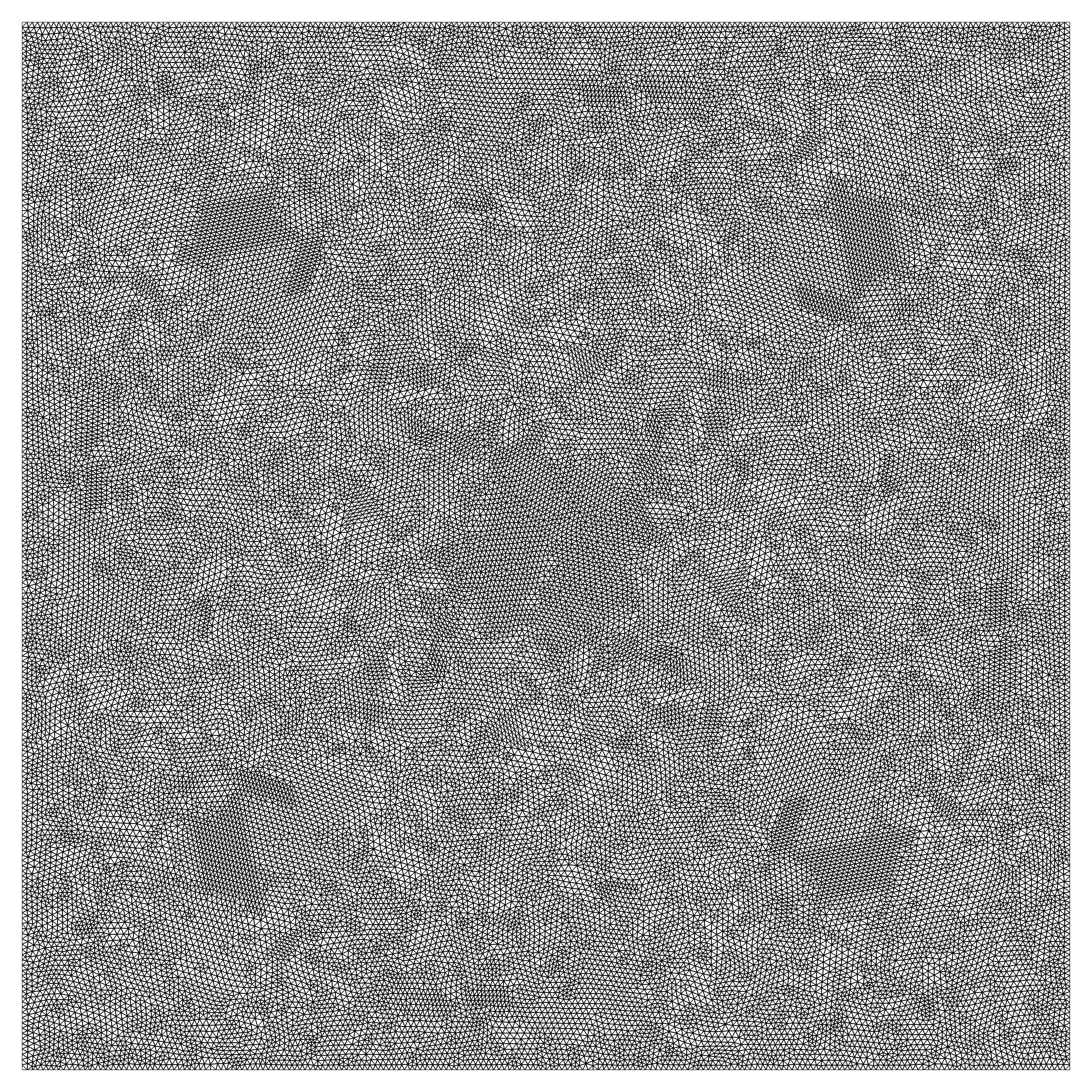}
	\includegraphics[trim= 10 5 5 5,clip,width=0.45\linewidth]{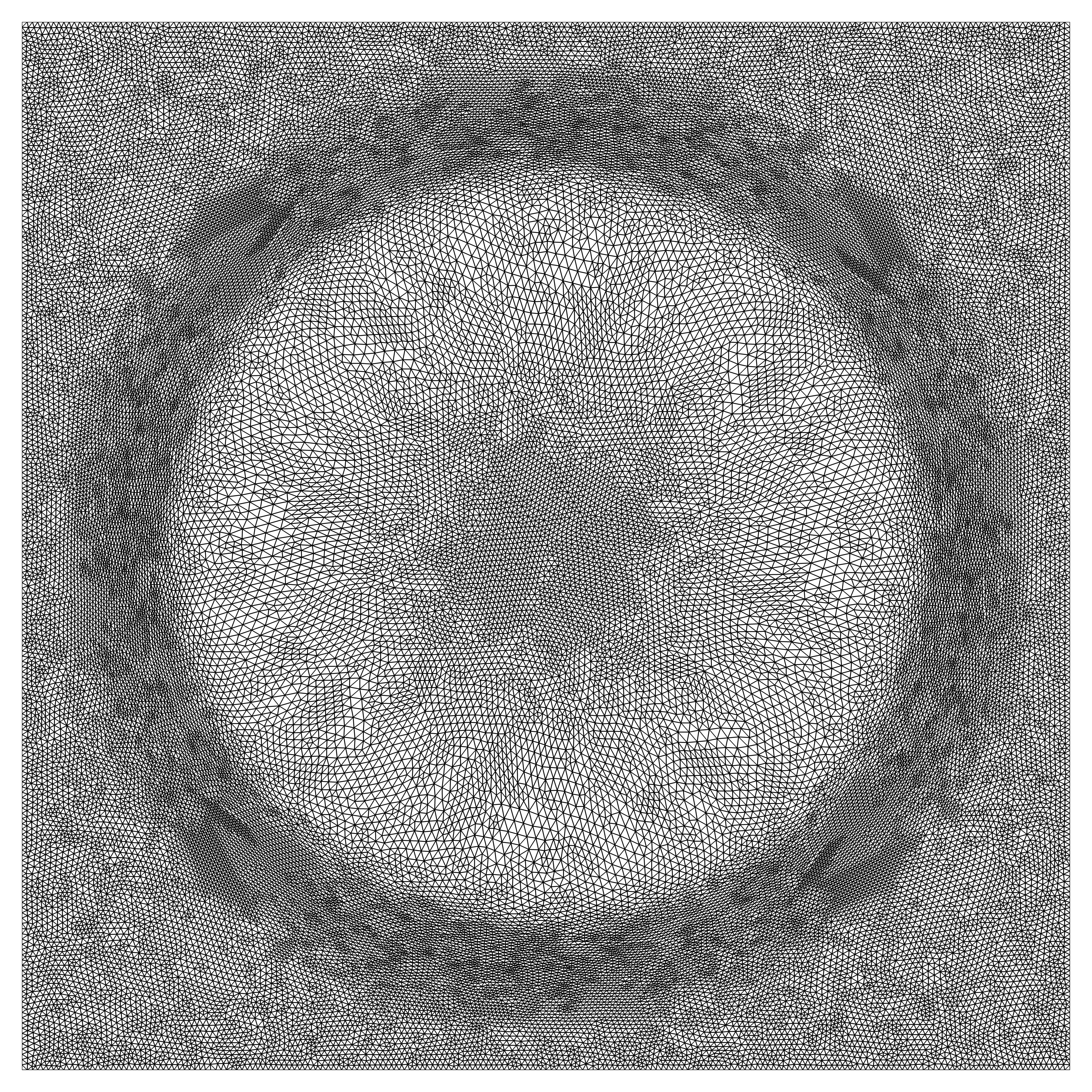}
	\caption{Initial grid and deformed mesh obtained for the circular explosion problem using the ALE hybrid FV/FE method.}
	\label{fig.CE85_ale_t025_mesh}
\end{figure} 
\begin{figure}[!htbp]
	\centering
	\includegraphics[trim= 10 5 5 5,clip,width=0.45\linewidth]{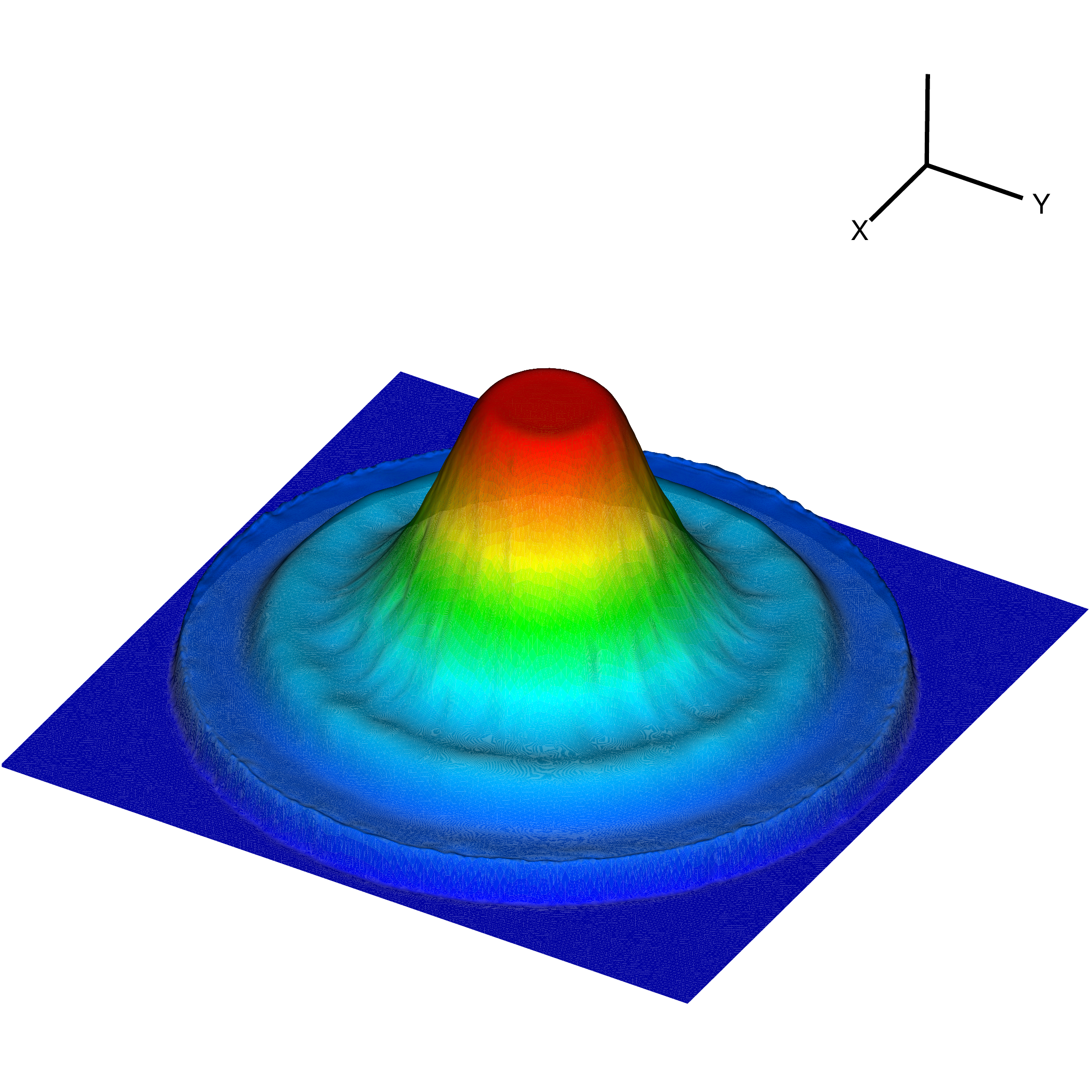} \hspace{0.05\linewidth}
	\includegraphics[trim= 10 5 5 5,clip,width=0.45\linewidth]{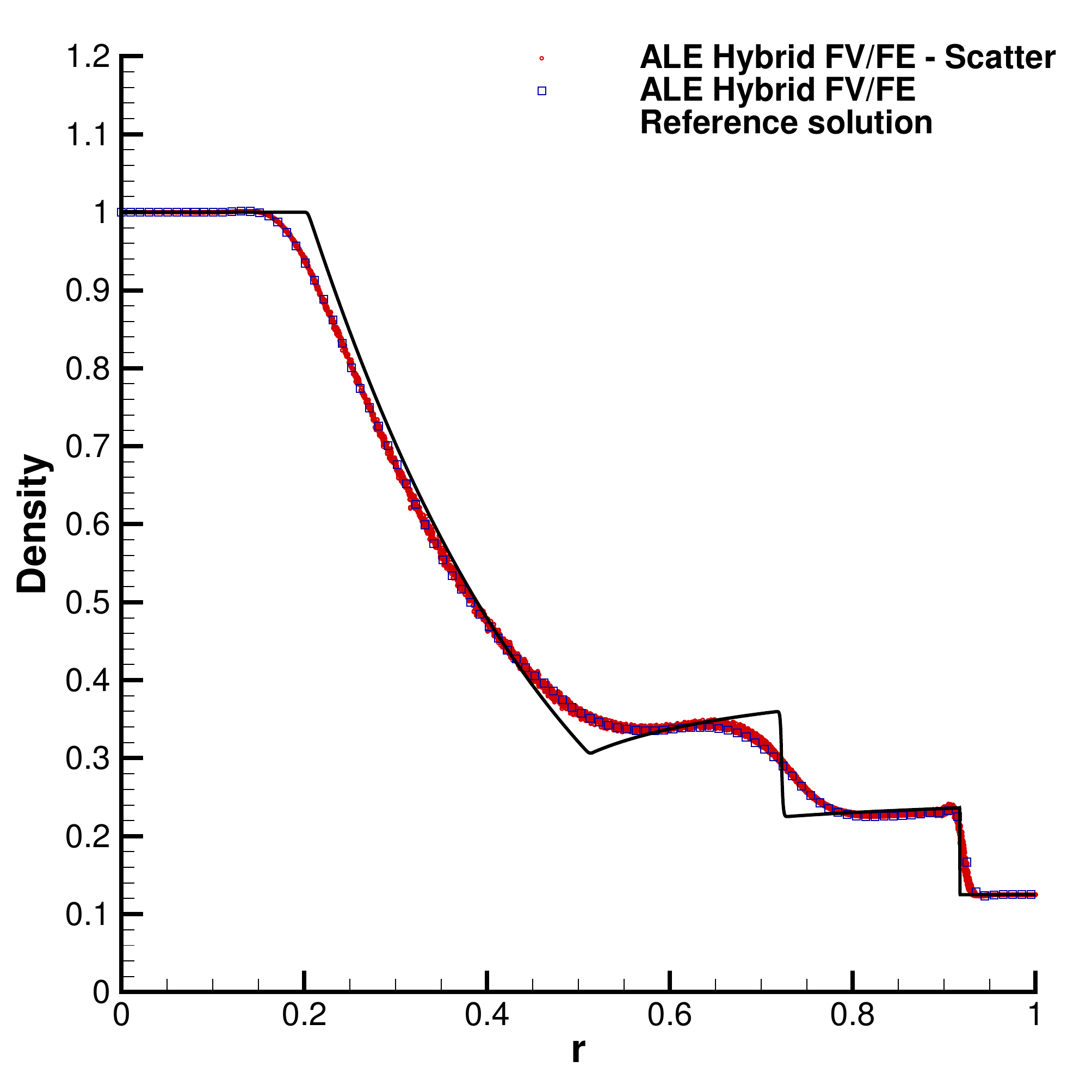}
	
	\vspace{0.05\linewidth}
	\includegraphics[trim= 10 5 5 5,clip,width=0.45\linewidth]{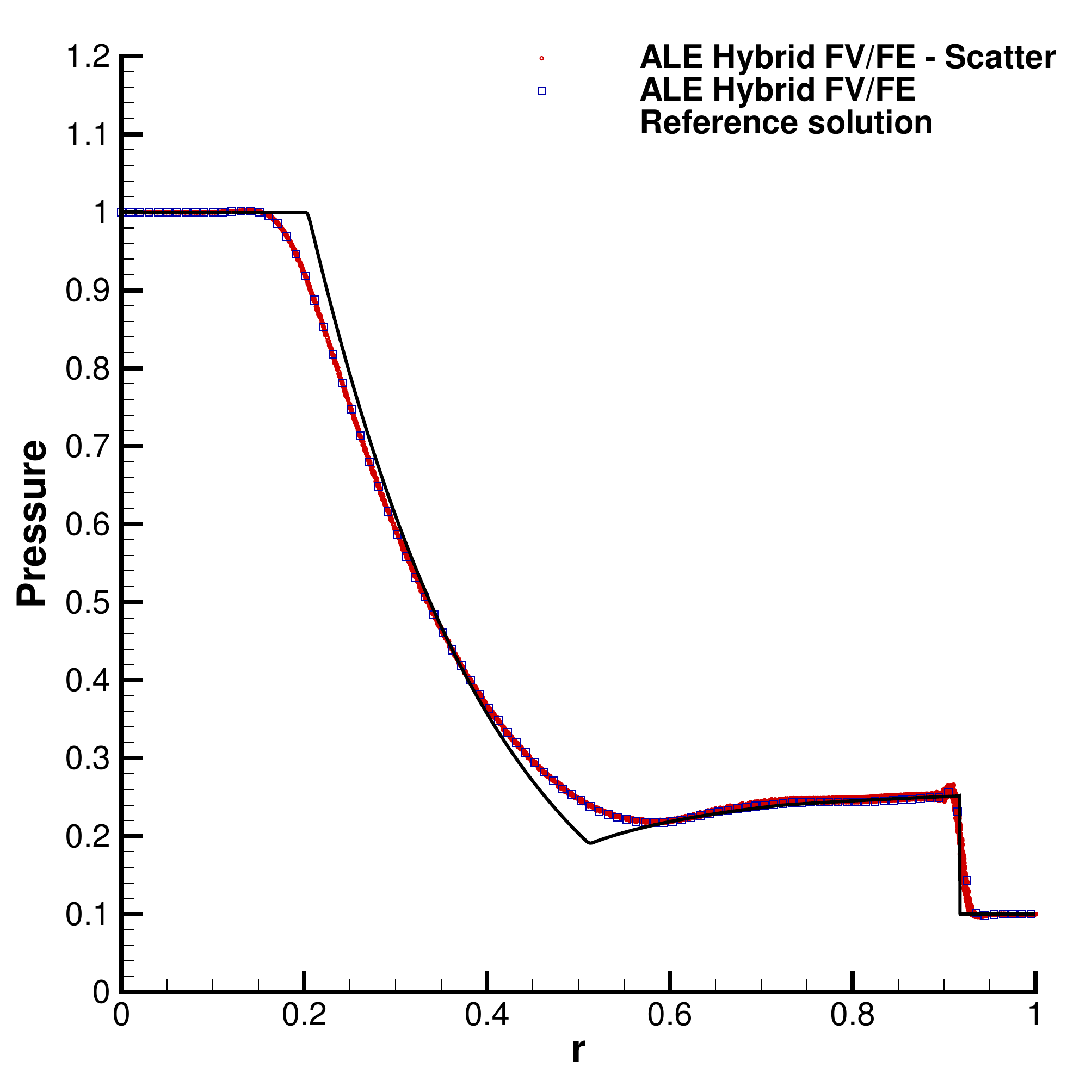} \hspace{0.05\linewidth}
	\includegraphics[trim= 10 5 5 5,clip,width=0.45\linewidth]{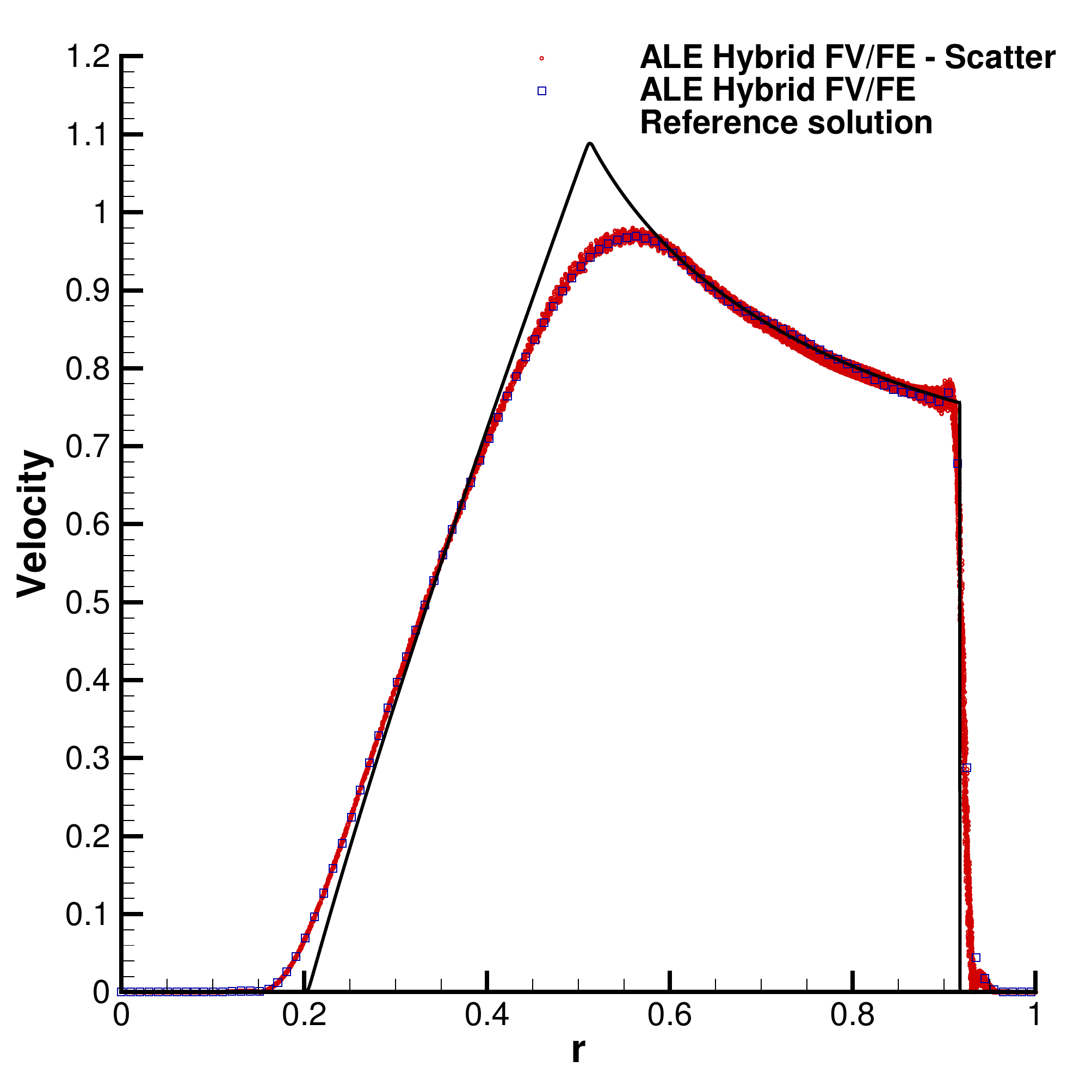}
	\caption{Circular explosion 2D computed using the ALE Hybrid FV/FE at $t=0.25$. Left top: contour plot of the density field. From right top to right bottom: 1D cut (blue squared line), scatter plot (red dots), reference solution (black continuous line), of the density, pressure and velocity.}
	\label{fig.CE85_ale_t025}
\end{figure}

\section{Conclusions}\label{sec:conclusions}
In this paper, we have introduced a novel semi-implicit hybrid finite volume / finite element method for the solution of the incompressible and weakly compressible Navier-Stokes equations on moving unstructured staggered meshes. 
The hybrid FV/FE scheme is based on a suitable split of the governing PDE system into subsystems, each of which is then discretized with the most suitable method. For the nonlinear convective and viscous terms, i.e. for the hyperbolic-parabolic subsystem, an explicit ALE finite volume method on moving meshes is employed, making use of a space-time divergence form of the related subsystem. For incompressible flows, instead of the simple Rusanov flux, which was used in previous work on hybrid FV/FE schemes, we adapt the Ducros flux to the context of moving ALE meshes and thus obtain a provably kinetic energy preserving / kinetic energy stable discretization of the nonlinear convective terms. 
The mesh motion can be prescribed in two different ways: i) the mesh can be moved with the local fluid velocity, in combination with a suitable smoothing step that is based on the solution of a diffusion equation for the mesh velocity and which allows keeping a good mesh quality during the simulation; ii) the mesh can be moved by imposing a velocity on the boundary and by solving a Laplace equation for the mesh velocity in the interior of the domain. The pressure subsystem is discretized at the aid of continuous $\mathbb{P}_1$ finite elements. 

The new numerical approach has been successfully applied to a series of test problems for the incompressible and the weakly compressible Navier-Stokes equations, including also weak shock waves, rarefactions and contact discontinuities. Whenever possible, we have compared our results to available exact or numerical reference solutions. 

In the future, we plan to extend this approach to full fluid-structure interaction (FSI) problems, where the dynamics of the solid are governed either by the equations of linear elasticity or nonlinear hyperelasticity, see, e.g.~\cite{GodunovRomenski72,Rom1998,PeshRom2014}.

\section*{Acknowledgements}
The Authors acknowledge the financial support of the Italian Ministry of Education, University and Research (MIUR) via the Departments of Excellence Initiative 2018--2022 attributed to DICAM of the University of Trento (grant L. 232/2016) and in the framework of the PRIN 2017 project \textit{Innovative numerical methods for evolutionary partial differential equations and applications}. L.R. acknowledges funding from the Spanish Ministry of Universities and the European Union-Next GenerationEU under the project RSU.UDC.MS15. S.B. was funded by the Spanish Ministry of Science and Innovation, grant number PID2021-122625OB-100. 
SB, MD and LR are members of the GNCS group of INdAM.

The authors would like to acknowledge support from the CESGA, Spain, for the access to the FT3 supercomputer and to the Leibniz Rechenzentrum (LRZ), Germany, for granting access to the SuperMUC-NG supercomputer (project number pr63qo).

Also, they would like to acknowledge the support provided by the Shark-FV 2022 conference where this work has been finished.

%
%
%

\bibliographystyle{elsarticle-num}
\bibliography{./mibiblio}


\end{document}